\documentclass[11pt,leqno,a4paper]{article}
\usepackage[hyperref]{mymacros}
\usepackage{comment}

\newcommand{\LL}{\bm{L}}
\newcommand{\PP}{\bm{P}}
\newcommand{\EE}{\bm{E}}

\newcommand{\lz}{\mathsf{z}}
\newcommand{\lZ}{\mathsf{Z}}

\newcommand{\GFF}{\mathrm{GFF}}
\newcommand{\SG}{\mathrm{SG}}

\newcommand{\bu}{\mathbf{u}}

\newcommand{\ZDM}{{\rm Z}}

\newcommand{\exteps}{E_\epsilon}
\newcommand{\extepsI}{I_\epsilon}

\usepackage[titles]{tocloft}
\author{Roland Bauerschmidt\footnote{University of Cambridge, Statistical Laboratory, DPMMS. E-mail: {\tt rb812@cam.ac.uk} and {\tt mh901@cam.ac.uk}.} \and
  \and Michael Hofstetter$^*$}

\title{Maximum and coupling of the sine-Gordon field}
\date{\vspace*{-2em}}

\begin{document}
\maketitle
\begin{abstract}
  For $0<\beta<6\pi$, we prove that the distribution of the centred maximum of the $\epsilon$-regularised continuum sine-Gordon field on the two-dimensional torus
  converges to a randomly shifted Gumbel distribution as $\epsilon \to 0$.
  Our proof relies on a strong coupling at all scales
  of the sine-Gordon field with the Gaussian free field,
  of independent interest, and extensions of existing methods
  for the maximum of the lattice Gaussian free field.
\end{abstract}

\section{Introduction and main results}
\label{sec:intro}

\subsection{Main results}

We consider the (continuum) sine-Gordon field on the two-di\-men\-sion\-al torus $\Omega=\T^2$
with mass $m=1$ and coupling constants $z\in \R$ and $0<\beta<6\pi$.
Regularised using a lattice of mesh $\epsilon$ such that $1/\epsilon$ is an integer,
its distribution is the probability measure on $\varphi: \Omega_\epsilon \to \R$ given by
\begin{equation}   \label{eq:DefinitionSGMeasure}
  \nu^{\SG_\epsilon}(d\varphi) = 
  \frac{1}{Z_\epsilon}
  \exp\qa{-\epsilon^2\sum_{\Omega_\epsilon} \pa{\frac{1}{2}\varphi(-\Delta^\epsilon\varphi) + \frac12 \varphi^2 + 2 z \epsilon^{-\beta/4\pi} \cos(\sqrt{\beta}\varphi)}} \prod_{x\in\Omega_\epsilon} d\varphi(x)
\end{equation}
where $\Omega_\epsilon =\T^2 \cap \epsilon\Z^2$ is the discretised unit torus
and $\Delta^\epsilon$ is the discretised Laplacian, i.e.,
$\Delta^\epsilon f(x) = \epsilon^{-2} \sum_{y \sim x} (f(y)-f(x))$
with $x \sim y$ denoting that $x$ and $y$ are neighbours in $\epsilon\Z^2$.

Under suitable assumptions on $z$ and $\beta$,
it is known that the probability measure $\nu^{\SG_\epsilon}$ converges weakly as $\epsilon \downarrow 0$ to a non-Gaussian
probability measure $\nu^{\SG}$ on $H^{-\kappa}(\T^2)$.
For $\beta < 6\pi$ and $z\in \R$, this convergence
along with a strong coupling to the Gaussian free field (GFF)
is also 
a by-product of our results, see Theorem~\ref{thm:coupling-intro}.
The limiting sine-Gordon field is an example of a superrenormalisable (or subcritical) Euclidean field theory.
This means that its short-distance behaviour is that of the GFF (in a suitable sense).
In particular, the limiting field $\Phi^\SG$
is not a random function, but only a generalised function of negative regularity,
with short-distance correlations diverging logarithmically.
Such local regularity properties are by now relatively well understood in
a number of examples of superrenormalisable field theories.

In this article, we study more global probabilistic properties of the sine-Gordon field,
which go beyond weak convergence of the field.
In particular, in the following theorem,
we determine the distribution of the centred maximum of the regularised field
as $\epsilon \downarrow 0$.
Such extremal behaviour of log-correlated random fields
is a topic on which a lot of progress has been made in the last few years.
For Gaussian fields, a very complete understanding now exists,
but much less is known for non-Gaussian fields
and understanding the distributional properties of their extrema (as in our result below) is a mostly open problem.
For further discussion and references, see Section~\ref{sec:literature}.

\begin{theorem}
  \label{thm:convergence-to-Gumbel}
  Let $0<\beta<6\pi$ and $z\in \R$.
  Then the centred maximum of the $\epsilon$-regularised sine-Gordon field $\Phi^{\SG_\epsilon} \sim \nu^{\SG_\epsilon}$ converges in law
  to a randomly shifted Gumbel distribution:
  \begin{equation} \label{e:thm-max-convergence}
    \max_{\Omega_\epsilon} \Phi^{\SG_\epsilon}
    - \frac{1}{\sqrt{2\pi}} \pa {2\log \frac{1}{\epsilon} - \frac34 \log \log \frac{1}{\epsilon}+b}
    \to \frac{1}{\sqrt{8\pi}}X + \frac{1}{\sqrt{8\pi}}\log \ZDM^{\SG},
  \end{equation}
  where $\ZDM^\SG$ is a nontrivial positive random variable (depending on $\beta$ and $z$),
  $X$ is an independent standard Gumbel random variable,
  and $b$ is a deterministic constant.
\end{theorem}

The result for the GFF, namely $z=0$, was proved in \cite{MR3433630}
(see also \cite{MR3414451,MR3729618,MR4043225}).
The constant in \eqref{e:thm-max-convergence} is the same as in the analogous
statement for the GFF, but the distribution of the
random variable $\ZDM^\SG$ is different from the version for the GFF; see Remark~\ref{rk:ZSG} below.
Note that there is also a trivial difference in the factors involving $2\pi$
that results from different normalisations of the fields.
Our normalisation is such that $\var(\Phi^{\SG_\epsilon}(0)) \sim \frac{1}{2\pi}\log \frac{1}{\epsilon}$.

Our proof of Theorem~\ref{thm:convergence-to-Gumbel} relies on a strong coupling between the sine-Gordon field and the GFF,
based on the methods of \cite{1907.12308,MR914427},
combined with the methods developed for the study of the maximum of the GFF from \cite{MR3433630}.
This coupling, which is of independent interest, provides uniform control over the difference
between the two fields on any scale.
Concretely, we construct the sine-Gordon field as the final solution $\Phi_0^\SG$
to the $H^{-\kappa}(\T^2)$-valued (backward) SDE
\begin{equation} \label{e:intro-sde}
  \Phi_t^{\SG} = -\int_t^\infty \dot c_s \nabla v_s(\Phi_s^{\SG}) \, ds + \Phi^{\GFF}_t,
  \qquad (t\in [0,\infty]),
\end{equation}
where $\dot c_s \nabla v_s$ is the gradient
in the field direction of the \emph{renormalised potential} defined in Section~\ref{sec:vt},
and where $\Phi^\GFF_t$ is the decomposed GFF defined by
\begin{equation}
  \Phi_t^{\GFF} = \int_t^\infty e^{\frac12 \Delta s - \frac12 m^2 s} \, dW_s,
    \qquad (t \in [0,\infty]),
\end{equation}
where $W$ is a cylindrical Brownian motion in $L^2(\T^2)$,
$\Delta$ is the Laplace operator on $\T^2$
and $m=1$.
This scale-by-scale coupling, constructed in Section~\ref{sec:coupling},
implies in particular the following global coupling for the limiting
continuum sine-Gordon field with the continuum Gaussian free field.

\begin{theorem} \label{thm:coupling-intro}
  Let $0<\beta<6\pi$ and $z\in \R$.
  Then the $\epsilon$-regularised sine-Gordon field $\Phi^{\SG_\epsilon} \sim \nu^{\SG_\epsilon}$
  converges weakly to a (non-Gaussian) limiting field $\Phi^{\SG}$ in $H^{-\kappa}(\T^2)$.
  This continuum sine-Gordon field $\Phi^{\SG}$ on $\T^2$
  can be coupled to a continuum Gaussian free field $\Phi^{\GFF}$ on $\T^2$
  such that   $\Phi^{\Delta} := \Phi^{\SG}-\Phi^{\GFF}$
  satisfies the following H\"older continuity estimates:
  \begin{align}
    \max_{x\in\T^2} |\Phi^{\Delta}(x)|+ \max_{x\in \T^2} |\partial \Phi^{\Delta}(x)| +        
    \max_{x,y\in \T^2}\frac{|\partial \Phi^\Delta(x)-\partial \Phi^\Delta(y)|}{|x-y|^{1-\beta/4\pi}}
    &\leq O_\beta(|z|),
    && (0<\beta<4\pi),
           \nnb
    \label{e:PhiDelta-bd-bis}       
    \max_{x\in\T^2} |\Phi^{\Delta}(x)|+ \max_{x,y\in \T^2}
    \frac{|\Phi^\Delta(x)-\Phi^\Delta(y)|}{|x-y| (1+|\log|x-y||)}
    &\leq O_\beta(|z|) ,
    && (\beta=4\pi),
    \\
    \max_{x\in\T^2} |\Phi^{\Delta}(x)|+ \max_{x,y\in \T^2}
    \frac{|\Phi^\Delta(x)-\Phi^\Delta(y)|}{|x-y|^{2-\beta/4\pi}}
    &\leq O_\beta(|z|) ,
    && (4\pi< \beta<6\pi),
       \nonumber
  \end{align}
  with deterministic constants $O_\beta(|z|)$.
\end{theorem}

Theorem~\ref{thm:coupling-intro} provides an analogue for the (non-Gaussian) sine-Gordon field of
results of \cite{MR4047992} for Gaussian log-correlated fields.
In particular, Theorem~\ref{thm:coupling-intro} implies (with essentially the same proofs) analogues of the corollaries of \cite[Theorem~A]{MR4047992}
such as the convergence
of the derivative martingale measure (which is not a martingale here)
or the critical (non-Gaussian) multiplicative chaos measure
associated with the sine-Gordon field.

The tightness of the left-hand side of \eqref{e:thm-max-convergence}
along with Gumbel type tail bounds on its distribution are also 
immediate from Theorem~\ref{thm:coupling-intro} and the corresponding results
for the GFF from \cite{MR3101848,MR3433630,MR3729618}.
However, to obtain the actual convergence
of the centred maximum to a randomly shifted Gumbel distribution, we use more than
what is given by Theorem~\ref{thm:coupling-intro}. Our full multiscale coupling, stated in
Theorems~\ref{thm:coupling}--\ref{thm:coupling-bd} in Section~\ref{sec:coupling},
provides finer control and in particular approximate independence
of small scales from the large scales,
which we make use of to approximate the field by a GFF on small scales.

The distribution of the sine-Gordon field is not expected to be absolutely continuous
with respect to the GFF when $\beta \geq 4\pi$
(while it is absolutely continuous with respect to the GFF for $\beta<4\pi$, assuming 
the volume is finite as in our setting).
For $\beta<6\pi$, Theorem~\ref{thm:coupling-intro} shows that the sine-Gordon field
can nonetheless be coupled to a GFF up to a random H\"older continuous difference.
It remains a very interesting problem to find a best possible
coupling in the full regime $\beta<8\pi$ for which the sine-Gordon model exists \cite{MR1777310}.

Our coupling bears some resemblence with the recent
constructions of pathwise solutions to singular SPDEs
(see, e.g., \cite{MR2016604,MR3274562,MR3406823}).
Note however that our coupling provides uniform global bounds,
and is based on the renormalisation group dynamics \eqref{e:intro-sde} rather than Glauber dynamics.
In principle, it is possible to derive couplings based on Glauber dynamics as well.
These do not provide an equally useful notion of scale,
but in particular for the $\Phi^4_2$ model, it would be interesting to investigate if
these (see, for example, \cite{MR3825880,10.1002/cpa.21925})
can be used for the study of the distribution of the maximum nonetheless.
Other approaches introduced for continuum $\Phi^4$ models that have the potential
to be useful in our context (though not immediately) are those
of \cite{1805.10814} or \cite{2006.15933}.

\begin{remark} \label{rk:ZSG}
  We expect that arguments along the lines of \cite{MR4082182,MR4043225} would allow to
  characterise
  the random variable $\ZDM^{\SG}$ in Theorem~\ref{thm:convergence-to-Gumbel}
  as the total mass of the derivative martingale measure $\ZDM^{\SG}(dx)$ associated to $\Phi^\SG$.
  Using the above coupling, as briefly discussed below the statement of the theorem,
  this random measure can be represented as
  \begin{equation} \label{e:ZSGPhiDelta}
    \ZDM^{\SG}(dx) = e^{\sqrt{8\pi}\Phi^{\Delta}(x)} \ZDM^{\GFF}(dx)
  \end{equation}
  where $\ZDM^{\GFF}(dx)$ is the suitably normalised
  derivative martingale measure associated to the 
  Gaussian free field $\Phi^{\GFF}$  on $\T^2$. Alternatively,
  we expect that one can show that $\ZDM^{\SG}$ is the distributional limit of the random variables
    \begin{equation} \label{e:ZSGeps}
\ZDM^{\SG_\epsilon} = \epsilon^2\sum_{\Omega_\epsilon} (\frac{2}{\sqrt{2\pi}} \log \frac{1}{\epsilon} - \Phi^{\SG_\epsilon}) e^{-2\log\frac{1}{\epsilon}+\sqrt{8\pi}\Phi^{\SG_\epsilon}}.
  \end{equation}
  More generally, we expect that the full extremal process of the sine-Gordon field
  is characterised in  terms of $\ZDM^{\SG}(dx)$, as in \cite{MR3787554,MR4043225} for the GFF.
\end{remark}

\begin{remark}
  In Theorem~\ref{thm:convergence-to-Gumbel},
  we consider the maximum of the lattice approximation of the sine-Gordon field
  as $\epsilon \downarrow 0$.
  As in \cite{acostathesis}, it would be possible to prove an analogous statement
  for $\Phi^\SG * \eta_\epsilon$ where $\Phi^\SG$ is the limiting sine-Gordon field
  and $\eta$ is a mollifier.
\end{remark}

\subsection{Literature}
\label{sec:literature}

The extremal behaviour of the GFF and more general 
Gaussian log-correlated fields is now rather well understood. For the maximum of the GFF,
we refer in particular to 
\cite{MR3433630} as well as \cite{MR3414451} for closely related results,
and also to \cite{MR3729618} for a general class of
Gaussian log-correlated fields that includes the GFF.
There are also earlier results on branching Brownian motion and branching random walks,
see \cite{MR956064,MR3098680} and references therein and also \cite{PhysRevE.63.026110}.
The complete extremal process of the GFF was understood in \cite{MR3787554,MR3509015,MR4082182};
see in particular the excellent review \cite{MR4043225}.
Further works on related Gaussian fields include
\cite{MR4089492,1912.13184}.

For several non-Gaussian log-correlated fields,
progress on the order of the maximum has been made as well.
The following ones are most relevant to our work;
for none of these, the limiting distribution of the centred maximum has been identified.
For gradient interface models with uniformly convex interaction,
the leading logarithmic order of the maximum has been identified \cite{1610.04195} and
subsequential tightness of the centred maximum has been established as well \cite{MR3933043}, but the understanding of the subleading orders remains open.
For the (integer-valued) Discrete  Gaussian model,
it has been shown that the maximum is of logarithmic order \cite{1907.08868}.
There are also results for the maximum of the logarithm of the characteristic polynomial
of certain random matrices, see for example \cite{MR3594368,MR3848227,MR3848391,cpa.21899}.
Like these random matrix ensembles, the sine-Gordon model is closely related to a Coulomb gas
type particle ensemble,
but our point of view is completely different.

The continuum sine-Gordon field is a prototypical example of a Euclidean field theory.
It is an interesting problem to show that an analogous result for the maximum holds also
for other such examples like the continuum $\Phi^4_2$ model.

There are also lattice versions of these models. We emphasise that these are
not the same as the lattice regularisations of the continuum models.
In physics terminology, the continuum version is related to the ultraviolet problem,
while the lattice model is related to the infrared problem. These problems differ
by the scaling of the coupling constants with the lattice spacing, and the expected distribution of the maximum
is also different.

We also mention that recent results on the (continuum) sine-Gordon model include
\cite{1907.12308},
\cite{1806.02118,1903.01394},
and older ones \cite{MR649810,MR849210},
\cite{MR914427}, \cite{MR2461991}, and \cite{MR1777310,MR1672504}.
The dynamics of sine-Gordon model has been studied in \cite{MR3452276,1808.02594,1907.12308}.
The sine-Gordon model is also closely related to the two-component plasma
\cite{MR0434278,MR3606477}.

Finally, we comment on some recent works that use methods that bear some relation to the ones we develop.
Continuous scale decomposition of the GFF and methods of stochastic analysis also
have been used
in the context of the Gaussian multiplicative chaos (see \cite{MR3274356,MR3339158} and references),
in recent constructions of the $\Phi^4_d$ model,
\cite{1805.10814,2004.01513,2003.12535},
and of the boundary sine-Gordon model \cite{1903.01394}.
Despite the shared use of the continuous decomposition of the GFF, these constructions have
a flavour quite different from ours (they use the decomposition to construct the moment
generating function of the field rather than a pathwise coupling of the field itself).
For Gaussian log-correlated fields,
couplings more in the sprit of ours have been constructed
in \cite{MR4047992}; see also the discussion below Theorem~\ref{thm:coupling-intro}.
The above mentioned results for gradient interface models with uniformly convex potential
also make use of a coupling \cite{MR2855536}
(of a somewhat different flavour though).

\subsection{Outline}

In Section~\ref{sec:vt} we give detailed estimates on the renormalised potential $v_t$
of the sine-Gordon field, extending the estimates from \cite{1907.12308,MR914427}.
Using this renormalised potential, in Section~\ref{sec:coupling},
we then construct the sine-Gordon field as the solution to the SDE \eqref{e:intro-sde}.
In Section~\ref{sec:maximum}, we use this coupling together with the
methods developed to determine the distribution of the maximum of the GFF
from \cite{MR3433630} to prove Theorem~\ref{thm:convergence-to-Gumbel}.

\subsection{Notation}

We use the standard Landau big-$O$ and little-$o$ notation,
and emphasise by writing, for example,
$O_p$ that the implied constant depends on a parameter $p$.
We also write $A \lesssim B$ to denote that $A\leq O_\beta(B)$, and $A \simeq B$ if $A \lesssim B$
and $B\lesssim A$.

\section{Convergence of the renormalised potential}
\label{sec:vt}

\subsection{Renormalised potential}
\label{sec:vteps}

Let $\Omega = \T^2$ be the continuum unit torus,
and let $\Omega_\epsilon = \Omega \cap \epsilon\Z^2$ be its lattice approximation where we assume from now on that $1/\epsilon$ is an integer.
Let $X_\epsilon = \{ \varphi : \Omega_\epsilon \to \R \}$. 
For emphasis, we will sometimes denote the side length of the torus by $L=1$.
Given $z\in \R$, 
we define the microscopic potential $v_0^\epsilon$ of the sine-Gordon field
for $\varphi = (\varphi(x))_{x\in\Omega_\epsilon} \in X_\epsilon$
by
\begin{equation}
  v_0^\epsilon(\varphi)
  = \epsilon^{2} \sum_{x\in \Omega_\epsilon} 2z\epsilon^{-\beta/4\pi} \cos(\sqrt{\beta}\varphi(x)).
\end{equation}
For $t>0$, we define the renormalised potential $v_t^\epsilon: X_\epsilon \to \R$ as follows.
Let
\begin{equation} \label{e:dotcteps}
  \dot c_t^\epsilon = e^{t \Delta^\epsilon} e^{-m^2 t},
   \qquad
   c_t^\epsilon = \int_0^t \dot c_s^\epsilon \, ds,
 \end{equation}
where $\Delta^\epsilon$ is the discretised Laplacian acting on $X_\epsilon$ as defined below \eqref{eq:DefinitionSGMeasure}.
Then for any $t>s>0$,
\begin{equation} \label{e:vdef-conv}
  e^{-v_t^\epsilon(\varphi)}
  = \EE_{c_t^\epsilon} \pB{ e^{-v_0^\epsilon(\varphi+\zeta)} }
  = \EE_{c_t^\epsilon-c_s^\epsilon} \pB{ e^{-v_s^\epsilon(\varphi+\zeta)} },
\end{equation}
where $\EE_{c}$ denotes the expectation of the Gaussian measure with covariance $c$. 
The second equality holds since the convolution of two Gaussian measures with covariances
$c_1$ and $c_2$ is Gaussian with covariance $c_1+c_2$.
Equivalently to \eqref{e:vdef-conv}, $v_t^\epsilon$ is the unique
solution to the Polchinski equation:
\begin{equation}
  \label{eq:PolchinskiSDE-v}
  \partial_t v_t^\epsilon
  = \frac12 \Delta_{\dot c_t^\epsilon} v_t^\epsilon - \frac12 (\nabla v_t^\epsilon)_{\dot c_t^\epsilon}^2
  = \frac12 \epsilon^4\sum_{x,y \in \Omega_\epsilon} \dot c_t^\epsilon(x,y)
  \qa{\ddp{^2v_t^\epsilon}{\varphi(x)\partial\varphi(y)}
    -
    \ddp{v_t^\epsilon}{\varphi(x)}
    \ddp{v_t^\epsilon}{\varphi(y)}
    }
\end{equation}
where the matrix $\dot c_t^\epsilon(x,y)$ associated to $\dot c_t^\epsilon$ is normalised such that
\begin{equation} \label{e:dotcteps-sum}
  \dot c_t^\epsilon f(x) = \epsilon^2 \sum_{y\in \Omega_\epsilon} \dot c_t^\epsilon(x,y) f(y).
\end{equation}
Later, we will also use the notation
\begin{equation} \label{e:innerprod}
  \avg{f,g} \equiv \avg{f,g}_{\Omega_\epsilon} \equiv \epsilon^2 \sum_{x \in \Omega_\epsilon} f(x)\bar g(x),
  \qquad
  \int_{\Omega_\epsilon} dx \equiv \epsilon^2 \sum_{\Omega_\epsilon}.
\end{equation}
Both representations of $v_t^\epsilon$, that as a solution to the Polchinski equation
\eqref{eq:PolchinskiSDE-v},
and that in terms of Gaussian convolution \eqref{e:vdef-conv}, are useful.
For their equivalence, see \cite{1907.12308,MR914427}.

\subsection{Statement of estimates}

To state the estimates for the renormalised potential, we define
\begin{gather} \label{e:ZL-def}
  L_t = \sqrt{t} \wedge 1/m,
  \qquad
  Z_t = z L_t^{-\beta/4\pi},
  \qquad
  \lZ_t = L_t^2Z_t = z L_t^{2-\beta/4\pi}
  ,\qquad
  \theta_t = e^{-\frac12 m^2t},
\end{gather}
and also set $Z_t^\epsilon = Z_{t\vee \epsilon^2}$ and $\lZ_t^\epsilon = \lZ_{t\vee \epsilon^2}$.
We also recall that $L=1$ is the side length of the torus.
Moreover, while the results in Section~\ref{sec:intro} are stated for the mass parameter $m=1$, we will allow $m>0$ throughout this section.
By rescaling one could then also recover the general $L$ case.

The following estimates can be extracted from
\cite[Section~3]{1907.12308} (which follows the method of \cite{MR914427}).
Compared to \cite[Section~3]{1907.12308}, we have rescaled $t\to t/\epsilon^2$ and $x\to x/\epsilon$
in preparation for the limit $\epsilon\downarrow 0$ which we will take later.
This rescaling is summarised in Section~\ref{sec:rescaling} below.

\begin{theorem} \label{thm:v-bd}
  Let $\beta<6\pi$, $m^2>0$, $z\in \R$, and $\epsilon>0$.
  There is $t_0 = t_0(\beta,z,m) > 0$ independent of $\epsilon>0$ such that
  for all $\varphi \in X_\epsilon$, $t \geq 0$,
  \begin{align}
    \label{e:nablaveps-bd}
    L_t^2 \norm{\dot c_t^\epsilon \nabla v_t^\epsilon(\varphi)}_{L^\infty(\Omega_\epsilon)}
    &\leq
    O_\beta(\theta_t|\lZ_t^\epsilon|)+  {\mathbf 1}_{t>t_0} O_{\beta,m}(\theta_t|\lZ_t^\epsilon|),
    \\
    \label{e:Hessveps-bd}
    L_t^2 \norm{\dot c_t^\epsilon \nabla v_t^\epsilon(\varphi)
    -\dot c_t^\epsilon \nabla v_t^\epsilon(\varphi')}_{L^\infty(\Omega_\epsilon)}
    &\leq
    (O_\beta(\theta_t|\lZ_t^\epsilon|)
      + {\mathbf 1}_{t>t_0} O_{\beta,m,z}(\theta_t))\norm{\varphi-\varphi'}_{L^\infty(\Omega_\epsilon)}
      .
  \end{align}
\end{theorem}

Since we will not directly apply Theorem~\ref{thm:v-bd} and extend the
estimates in Theorem~\ref{thm:v-limit} below,
we do not give a precise reference at this point.

Note that, 
from the macroscopic point of view, the microscopic potential
$v_0^\epsilon$ blows up as $\epsilon\downarrow 0$.
On the other hand, when normalised with respect to the
relevant scale with a factor $L_t^2$ as in the above theorem,
the gradients of the renormalised potential $v_t^\epsilon$ are bounded 
uniformly in $\epsilon$
and in fact tend to $0$ since
$\lZ_t \to 0$ as $t\to 0$.

The main result of this section is Theorem~\ref{thm:v-limit} below.
It extends the above estimates and also shows
the convergence of $v^\epsilon_t$ as $\epsilon \downarrow 0$.
To this end, we first note that the lattice covariance decomposition $\dot c_t^\epsilon$
from \eqref{e:dotcteps} and \eqref{e:dotcteps-sum}
converges to its continuous counterpart, given by
\begin{equation} \label{e:dotct0}
  \dot c_t^{0}(x,y)
  =
  e^{t\Delta}(x,y) e^{-m^2 t} 
  =
  \sum_{n\in\Z^2} \frac{e^{-|x-y+L n|^2/4t-m^2 t}}{4\pi t}
  ,
  \qquad L=1,
\end{equation}
where $\Delta$ is the Laplacian on $\T^2$.
Moreover, we note that $v_t^\epsilon$ is a function from $X_\epsilon$ to $\R$,
so for each $\varphi \in X_\epsilon$,
$\dot c_t^\epsilon \nabla v_t^\epsilon(\varphi)$ is an element of
$X_\epsilon  = \{\Omega_\epsilon \to \R\}$
and we denote this function by
\begin{equation}
  \dot c_t^\epsilon \nabla v_t^\epsilon(\varphi, x) = [\dot c_t^\epsilon \nabla v_t^\epsilon(\varphi)](x), \qquad x\in \Omega_\epsilon, \;\varphi \in X_\epsilon.
\end{equation}
Here $\nabla$ denotes the gradient with respect to the field $\varphi \in X_\epsilon=\R^{\Omega_\epsilon}$.
We will also need (discrete) gradients in the variable $x$ and denote these by
$\partial = \partial_\epsilon$.
More precisely, for $e$ one of the $2d$ unit directions in $\Z^d$,
we set $\partial_{\epsilon}^e f(x) = \epsilon^{-1}(f(x+\epsilon e)-f(x))$,
and for $k \in \N$ we then denote by
$\partial_\epsilon^k f$ the matrix-valued function consisting of all iterated discrete gradients
$\partial^{e_1}_\epsilon \cdots \partial^{e_k}_\epsilon f$ where $e_1, \dots, e_k$
are unit directions in $\Z^d$.
For $\epsilon=0$, we similarly denote true spatial derivatives by $\partial$.

Finally, for $x\in \Omega$, we will denote by $x^\epsilon$ the point in $\Omega_\epsilon$ that is closest to $x$, i.e.,
the unique element of $\Omega_\epsilon$ such that $x\in x^\epsilon + (-\epsilon/2,\epsilon/2]^2$.

\begin{theorem} \label{thm:v-limit}
  Let $\beta<6\pi$, $m^2>0$, $z\in \R$.
  For all $t>0$, there exist bounded functions
  $\dot c_t^0 \nabla v_t^0 : C(\Omega) \to C^\infty(\Omega)$ such that
  $\dot c_t^\epsilon \nabla v^\epsilon_t$ converges to $\dot c_t^0\nabla v_t^0$
  in the sense that, for any  
  $\varphi \in C(\Omega)$ and any $\varphi^\epsilon \in X_\epsilon$ such that $\sup_{x \in \Omega} |\varphi^\epsilon(x^\epsilon)-\varphi(x)| \to 0$,
  \begin{equation} \label{e:cnablav-eps}
    L_t^{2} \norm{\dot c_t^0\nabla v_t^0(\varphi)-\dot c_t^\epsilon\nabla v_t^\epsilon(\varphi^\epsilon)}_{L^\infty(\Omega_\epsilon)}
    \to 0
    \qquad (\epsilon \to 0)
    .
  \end{equation}
  The following estimates hold (for $\epsilon>0$ and in the limit $\epsilon=0$), uniformly in $\varphi,\varphi'$, for any $k\in \N$:
  \begin{align}
    \label{e:cnablav-bd}
    L_t^{2+k} \norm{\partial^k\dot c_t\nabla v_t(\varphi)}_{L^\infty(\Omega)}
    &\leq O_{\beta,k}(\theta_t |\lZ_t|)+{\bf 1}_{t>t_0}O_{\beta,m,k}(\theta_t |\lZ_{t}|),
    \\
    \label{e:cnablav-Lip}
    L_t^{2+k} \norm{\partial^k \dot c_t\nabla v_t(\varphi)-\partial^k \dot c_t\nabla v_t(\varphi')}_{L^\infty(\Omega)}
    &\leq 
      \p{O_{\beta,k}(\theta_t|\lZ_t|)+ {\bf 1}_{t>t_0} O_{\beta,m,z,k}(\theta_t) }
      \norm{\varphi-\varphi'}_{L^\infty(\Omega)},
  \end{align}
  where $t_0=t_0(\beta,z,m)>0$ is a positive constant independent of $\epsilon$.
\end{theorem}

Note that we consider $\dot c_t\nabla v_t$ as one object
rather than as a composition of $\dot c_t$ with $\nabla v_t$
(which we do not define in the continuum limit).
Also, as in this example for $\dot c_t\nabla v_t$, we will often omit the index $\epsilon$,
in which case we make the convention that it refers to both cases, $\epsilon>0$
and the limiting case $\epsilon=0$, simultaneously.

The remainder of this section is devoted to the proof of Theorem~\ref{thm:v-limit}.
The theorem provides the main input for the coupling of the sine-Gordon field
with the GFF in Section~\ref{sec:coupling}.

\subsection{Yukawa gas representation of renormalised potential}
\label{sec:yukawagas}

To prove Theorem~\ref{thm:v-limit},
we first recall the construction of the renormalised potential $v_t^\epsilon$
from \cite[Section~3]{1907.12308} and \cite{MR914427}.
As discussed above, compared to \cite[Section~3]{1907.12308},
we use the continuum rescaling $t\to t/\epsilon^2$ that is more convenient in our context;
the trivial relation between these scalings is summarised in Section~\ref{sec:rescaling} below.

For $\epsilon>0$ and $\xi_i = (x_i,\sigma_i) \in \Omega_\epsilon \times \{\pm 1\}$,
we define $\tilde v_t^\epsilon(\xi_1,\dots,\xi_n) = \tilde v_t^{n,\epsilon}(\xi_1,\dots,\xi_n)$ by
\begin{align}
  \label{e:polchinski-fourier-duhamel-rescaled1}
  \tilde v_t^{\epsilon}(\xi_1) &= e^{-\frac12 \beta c^\epsilon_t(0)} \epsilon^{-\beta/4\pi} z 
                                 = e^{-\frac12 \beta (c^\epsilon_t(0) + (\log \epsilon^2)/4\pi)} z
  \\
    \label{e:polchinski-fourier-duhamel-rescaled}
  \tilde v_{t}^{\epsilon}(\xi_1,\dots,\xi_n) &=
  \frac12 \int_0^{t} ds \, \sum_{I_1 \dot\cup I_2= [n]}
  \sum_{i\in I_1,j\in I_2}
  \dot u_s^\epsilon(\xi_i,\xi_j) 
  \tilde v_{s}^\epsilon(\xi_{I_1}) \tilde v_{s}^\epsilon(\xi_{I_2}) 
  e^{-(w^\epsilon_{t}-w^\epsilon_{s})(\xi_1,\dots,\xi_n)},
\end{align}
where the second equation is for $n \geq 2$ (see \cite[(3.32)]{1907.12308}),
$\xi_I = (\xi_i)_{i\in I}$ is identified with an element of $(\Omega_\epsilon\times\{\pm 1\})^{|I|}$
(the order does not matter), $\dot \cup$ denotes the disjoint union, $[n]=\{1,\ldots,n\}$, and
\begin{equation} \label{e:ueps-def}
  \dot u^\epsilon_t(\xi_1,\xi_2) = \beta \sigma_1\sigma_2 \dot c_t^\epsilon(x_1,x_2),
  \qquad
  (w^\epsilon_t-w^\epsilon_s)(\xi_1,\dots,\xi_n) = \frac{\beta}{2} \sum_{i,j=1}^n \int_s^t \sigma_i\sigma_2 \dot c_r^\epsilon(x_i,x_j) \, dr.
\end{equation}
In particular, $\tilde v_t^{n,\epsilon}$ is determined inductively by $\tilde v_s^{k,\epsilon}$ with $k<n$ and $s<t$,
and the above equations are well defined for all $n$ and $t$ whenever $\epsilon>0$.

As we will see further below,
for $\epsilon>0$, $\beta<6\pi$ and $t<t_0= t_0(z,\beta,m)$ independent of $\epsilon>0$,
the following Fourier series in fact converges absolutely uniformly in $\varphi \in X_\epsilon$:
\begin{equation} \label{e:v-fourier1}
  v_t^\epsilon(\varphi) = \sum_{n=0}^\infty v_t^{n,\epsilon}(\varphi),
\end{equation}
where
\begin{equation} \label{e:v-fourier2}
  v_t^{n,\epsilon}(\varphi) = \frac{1}{n!} \int_{(\Omega_\epsilon \times \{\pm 1\})^n} d\xi_1 \dots d\xi_n \, \tilde v_t^{n,\epsilon}(\xi_1, \dots, \xi_n) e^{i\sqrt{\beta}\sum_{i=1}^n \sigma_i\varphi(x_i)}
\end{equation}
and the discrete integral over $(\Omega_\epsilon\times \{\pm 1\})^n$ is defined as a sum analogously to \eqref{e:innerprod}.
Moreover, \eqref{e:v-fourier1}--\eqref{e:v-fourier2} then gives the unique solution to the Polchinski equation \eqref{eq:PolchinskiSDE-v};
see \cite[Section~3.3]{1907.12308}.

By differentiation in $\varphi$ (denoted $\nabla$) and then discrete differentiation in $x$
(denoted $\partial$), we further obtain that, for any $k \in \N$,
\begin{multline}
  \label{e:nablavtn-fourier}
  \partial_\epsilon^k \dot c_t^\epsilon\nabla v_t^{n,\epsilon}(\varphi,x)
  \\
  = \frac{i \sqrt{\beta}}{n!}
  \int_{(\Omega_\epsilon \times \{\pm 1\})^n} d\xi_1 \dots d\xi_n
  \pa{\sum_{i=1}^n \partial_\epsilon^k \dot c_t^\epsilon(x,x_i)\sigma_i} \tilde v_t^{n,\epsilon}(\xi_1,\dots,\xi_n) e^{i\sqrt{\beta}\sum_{i=1}^n \sigma_i \varphi(x_i)}.
\end{multline}
Denoting by $\He$ the Hessian matrix in $\varphi$, for any $g: \Omega_\epsilon \to \R$, also
\begin{multline} \label{e:Hessvtn-fourier}
  [\partial_\epsilon^k\dot c_t^\epsilon\He v_t^{n,\epsilon}(\varphi)g](x)
  \\
  = \frac{-\beta}{n!}
  \int_{(\Omega_\epsilon \times \{\pm 1\})^n} d\xi_1 \dots d\xi_n
  \pa{\sum_{i=1}^n \partial_\epsilon^k\dot c_t^\epsilon(x,x_i)\sigma_i}
  \pa{\sum_{i=1}^n g(x_i)\sigma_i}
  \tilde v_t^{n,\epsilon}(\xi_1,\dots,\xi_n) e^{i\sqrt{\beta}\sum_{i=1}^n \sigma_i\varphi(x_i)},
\end{multline}
and analogous expressions for higher derivatives hold as well.

We will define continuum versions $\tilde v_t^{0}(\xi_1,\dots, \xi_n)$ by passing to the limit $\epsilon \downarrow 0$
in \eqref{e:polchinski-fourier-duhamel-rescaled1}--\eqref{e:polchinski-fourier-duhamel-rescaled}.
To define these limits, for $t \geq  s > 0$, we first define $u_t^0$ and $w_t^0-w_s^0$  as in \eqref{e:ueps-def},
only replacing $c_t^\epsilon$ by $c_t^0$ defined in \eqref{e:dotct0}.
We emphasise that (with slight abuse of notation) we only define the difference $w_t^0-w_s^0$ and
not $w_t^0$ and $w_s^0$ individually.
We will show in Lemma~\ref{lem:C-limit} below that there exist constants $\gamma_t = \gamma_t(m)$ such that
\begin{equation}
  e^{-\frac12 \beta (c^\epsilon_t(0) + (\log \epsilon^2)/4\pi)} \to \gamma_t^\beta L_t^{-\beta/4\pi}.
\end{equation}
For all distinct $\xi_1, \xi_2, \ldots \in \Omega \times \{\pm 1\}$, we then define inductively
\begin{align}
  \label{e:polchinski-fourier-duhamel-0-1}
  \tilde v_t^{0}(\xi_1) &= L_t^{-\beta/4\pi} \gamma_t^\beta z
  \\
    \label{e:polchinski-fourier-duhamel-0}
  \tilde v_{t}^{0}(\xi_1,\dots,\xi_n) &=
  \frac12 \int_0^{t} ds \, \sum_{I_1 \dot\cup I_2= [n]}
  \sum_{i\in I_1,j\in I_2}
  \dot u_s^0(\xi_i,\xi_j) 
  \tilde v_{s}^0(\xi_{I_1}) \tilde v_{s}^0(\xi_{I_2}) 
  e^{-(w^0_{t}-w^0_{s})(\xi_1,\dots,\xi_n)}.
\end{align}
To see that this is well-defined we note that, for any distinct $\xi_i$,
one has $\dot u_s^0(\xi_i,\xi_j) = \sigma_i\sigma_j \dot c_s^0(x_i,x_j) = O_{x_i,x_j}(s^N)$ as $s\to 0$ for any $N$
by the heat kernel estimate \eqref{e:dotct0}, and also $e^{-(w_t-w_s)(\xi_1,\dots,\xi_n)} \leq 1$
since $\dot c_t$ is positive definite.
From this, it follows easily by induction that the integrals on the right-hand side of
\eqref{e:polchinski-fourier-duhamel-0} converge absolutely for all $\xi_i$ distinct.

Finally, when
\begin{equation}
\label{eq:integral-1-convergence}
  \frac{1}{n!}
  \int_{(\Omega \times \{\pm 1\})^n} d\xi_1 \dots d\xi_n \, |\tilde v_t^{n,0}(\xi_1, \dots, \xi_n)| < \infty
\end{equation}
respectively
\begin{equation}
\label{eq:integral-2-convergence}
  \frac{1}{n!} \int_{(\Omega \times \{\pm 1\})^{n}} d\xi_1\cdots d\xi_n
  \absa{\sum_{i=1}^n \partial^k\dot c_t^0(x,x_i)\sigma_i \tilde v_t^{n,0}(\xi_1,\dots,\xi_n)} < \infty
\end{equation}
we  define $v_t^{n,0}: C(\Omega) \to \R$ respectively $\partial^k c_t^0\nabla v_t^{n,0}: C(\Omega) \to C(\Omega)$
as in \eqref{e:v-fourier1}--\eqref{e:nablavtn-fourier} with $\epsilon=0$. 
We will provide conditions for \eqref{eq:integral-1-convergence} and \eqref{eq:integral-2-convergence} below.

\subsection{Convergence of covariances and Gaussian fields}
\label{sec:C-limit}

The following lemma provides the conver\-gence of the covariance terms and heat kernel terms in
\eqref{e:polchinski-fourier-duhamel-rescaled1}--\eqref{e:polchinski-fourier-duhamel-rescaled}
to those in
\eqref{e:polchinski-fourier-duhamel-0-1}--\eqref{e:polchinski-fourier-duhamel-0}.
For $x\in \Omega$, recall that we denote by $x^\epsilon$ the unique element of $\Omega_\epsilon$
such that $x\in x^\epsilon + (-\epsilon/2,\epsilon/2]^2$,
and correspondingly, for $\xi =(x,\sigma) \in \Omega \times \{\pm 1\}$, we write $\xi^\epsilon = (x^\epsilon,\sigma)$.
For $f : \Omega_\epsilon^n \to \R$ we further denote by $\exteps f$ the piecewise constant
extention of $f$ to $\Omega^n$, i.e., $\exteps f(x_1,\dots, x_n) = f(x_1^\epsilon, \dots, x_n^\epsilon)$, and we denote the extension of $f: (\Omega_\epsilon \times \{\pm 1\})^n \to \R$
analogously.

\begin{lemma} \label{lem:C-limit}
  There exists a function $\gamma_t=\gamma_t(m)$ of $(t,m) \in [0,\infty) \times (0,\infty)$
  with $\gamma_t = \gamma_0 + O(m^2t)$ as $t\to 0$ such that
  \begin{equation}
    \label{e:c0-limit}
    e^{-\frac12 c_t^{\epsilon}(0,0)}\epsilon^{-1/4\pi} \to  \gamma_t L_t^{-1/4\pi} .
  \end{equation}
  For all $x \neq y \in \Omega$ and $t>0$, the integral $c_t^0(x,y) = \int_0^t \dot c_s^0(x,y) \, ds$ exists and
  uniformly on compact subsets of $x\neq y$,
  \begin{equation} \label{e:cxy-limit}
    c_t^\epsilon(x^\epsilon,y^\epsilon) \to c_t^0(x,y).
  \end{equation}
  For all $t \geq s \geq \epsilon^2$, uniformly in $\xi = (\xi_1,\dots,\xi_n) \in (\Omega \times \{\pm 1\})^n$,  
  \begin{equation} \label{e:ew-limit}
    e^{-(w_t^\epsilon(\xi^\epsilon)-w_s^\epsilon(\xi^\epsilon))} \to e^{-(w_t^0(\xi)-w_s^0(\xi))} .
  \end{equation}
  For all $t \geq \epsilon^2$, and all $k \in \N$,
  \begin{equation}
      L_t^k \sup_x \norm{\partial_\epsilon^k \dot c_t^\epsilon(x,\cdot)}_{L^1(\Omega_\epsilon)}  
      \leq O_k(\theta_t),
      \qquad
      \label{e:c-limit}
    \sup_{x}
    \norm{\exteps \dot c_t^\epsilon(x,\cdot)-\dot c_t^0(x,\cdot)}_{L^1(\Omega)}
    \leq O(\frac{\epsilon^2}{t}) \theta_t.
  \end{equation}
\end{lemma}

The proof of the lemma essentially follows from the convergence of the lattice
heat kernel to its continuum counterpart.
We have collected the required heat kernel statements in the appendix,
in Lemmas~\ref{lem:pt}--\ref{lem:pttorus}, to which we will refer in the 
proof of the above lemma below.

\begin{proof}[Proof of \eqref{e:c0-limit}]
To emphasise the effect of periodic boundary conditions and the mass term,
recall that we write $L$ for the side length of the torus. 
Then, denoting the torus heat kernel by $p^{\epsilon,L}_t(x)$, $x\in \Omega_\epsilon$,
as in Appendix~\ref{app:pt},
our goal is to estimate
$c_t^{\epsilon,L}(0) = \int_0^{t} p^{\epsilon,L}_{s}(0) e^{-m^2 s} \, ds$.
We will first estimate its infinite volume version, i.e., with $L=\infty$, given by
\begin{align}
  c_t^{\epsilon,\infty}(0)
  &= \int_0^{t} p^\epsilon_{s}(0) e^{-m^2 s} \, ds
    \nnb
  &= \int_0^{t\wedge 1/m^2} p^\epsilon_{s}(0) \, ds
    + \int_0^{t\wedge 1/m^2} p^\epsilon_s(0)(e^{-m^2s}-1) \, ds
    + \int_{t\wedge 1/m^2}^t p^\epsilon_s(0)e^{-m^2s} \, ds
    ,
    \label{e:cthreeterms}
\end{align}
where $p_t^\epsilon(x)$, $x\in \epsilon\Z^2$ is the heat kernel on $\epsilon\Z^2$, see Appendix~\ref{app:pt}.
The first term on the right-hand side is estimated as follows.
We denote the unit lattice heat kernel by $\tilde p_s$ so that $p_s^\epsilon(x) = \epsilon^{-2} \tilde p_{s/\epsilon^2}(x/\epsilon)$ for $x \in \epsilon\Z^2$.
By Lemma~\ref{lem:pt}, $\tilde p_s(0) = 1/(4\pi s) + O(1/s^2)$, and thus
\begin{equation}
  \int_0^t \tilde p_s(0) \, ds
  =
    \int_0^1 \tilde p_s(0) \, ds
    +
    \int_1^t \frac{ds}{4 \pi s}
    +
    \int_1^t \pa{ \tilde p_s(0)-\frac{1}{4\pi s}} \, ds
\end{equation}
and
\begin{equation} \label{e:intpt-asymp-bis}
  \int_0^t \tilde p_s(0) \, ds
  = \frac{1}{4\pi} \log t  + \tilde c + O(\frac{1}{t}),
  \qquad
  \tilde c = \int_0^1 \tilde p_s(0) \, ds +
  \int_1^\infty \pa{ \tilde p_s(0)-\frac{1}{4\pi s}} \, ds.
\end{equation}
Equivalently, rescaling $(s,x) \to (s/\epsilon^2,x/\epsilon)$, 
\begin{align}
\label{e:intpt-asymp}
  \int_0^t p_s^\epsilon(0) \, ds + \frac{1}{4\pi} \log \epsilon^{2}
  &= \frac{1}{4\pi} \log t  + \tilde c + O(\frac{\epsilon^2}{t}).
\end{align}
For the second term on the right-hand side of \eqref{e:cthreeterms},
since $p^0_t(0)=1/(4\pi t)$ and $p^0_t(0)-p^\epsilon_t(0) = O(\epsilon^2/t^2 \wedge 1/t)$,
we have
\begin{align}
  \int_0^{t} p_{s}^\epsilon(0) (1-e^{-m^2s}) \, ds 
  &=
    \int_0^{t} \frac{1}{4\pi s} (1-e^{-m^2s}) \, ds 
    + O(\epsilon^2m^2)\int_{\epsilon^2}^{1/m^2} \frac{ds}{s} + O(m^2\epsilon^2),
\end{align}
where the middle term is $O(\epsilon^2m^2 |\log (\epsilon^2m^2)|)$.
The third term on the right-hand side of \eqref{e:cthreeterms} is
\begin{align}
\int_{1/m^2}^t p_{s}^\epsilon(0) e^{-m^2s} \, ds
  &=
    \int_{1/m^2}^t \frac{1}{4\pi s} e^{-m^2s} \, ds
    + O(\epsilon^2m^2)
\end{align}
where we used
$\int_{1/m^2}^\infty \frac{\epsilon^2}{s^2} e^{-m^2s} \, ds = O(\epsilon^2m^2)$.
Define
\begin{equation}
  \tilde c(m,t) = \tilde c +
  \int_0^{t\wedge 1/m^2} \frac{1}{4\pi s} (e^{-m^2s}-1) \, ds
  +
  \int_{t\wedge 1/m^2}^t \frac{1}{4\pi s} e^{-m^2s} \, ds
\end{equation}
and note that
\begin{equation}
  \tilde c(m,t) = \tilde c + O(m^2t) \quad (t\to 0),\qquad
  \tilde c(m,t) = \tilde c(m,\infty) + O(e^{-m^2t}) \quad (t\to\infty).
\end{equation}
In summary, with $L_t = \sqrt{t} \wedge 1/m$, we have shown that
\begin{equation}
  c_t^{\epsilon,\infty}(0)+\frac{1}{4\pi}\log \epsilon^2 = \frac{1}{2\pi} \log L_t + \tilde c(m,t)
  + O(\frac{\epsilon^2}{t} + \epsilon^2 m^2 |\log \epsilon^2m^2|)
\end{equation}
and therefore
with $\gamma_t(m) = e^{-\frac12 \tilde c(m,t)}$, we have shown that
\begin{equation}
  \epsilon^{-1/4\pi} e^{-c^{\epsilon,\infty}_t(0)/2}
  = e^{-\frac12 (c^{\epsilon,\infty}_t(0) + (\log \epsilon^2)/4\pi)}
  = L_t^{-1/4\pi}  (\gamma_t(m)+O(\epsilon^2/t+m^2\epsilon^2|\log m^2\epsilon^2|)).
\end{equation}

Finally, we return to the torus version of $c_t$.
Using \eqref{e:pttorusbounds}--\eqref{e:pttoruslimit}, similar element\-ary computa\-tions show
that there is $\kappa>0$ such that
\begin{align}
  |c_t^{\epsilon,\infty}(0)-c_t^{\epsilon,L}(0)|
  &\lesssim 1
    \\
  |(c_t^{0,\infty}(0)-c_t^{0,L}(0))-(c_t^{\epsilon,\infty}(0)-c_t^{\epsilon,L}(0))|
  & \lesssim \epsilon^\kappa.
\end{align}
Thus with 
$\tilde c(m,L,t) = \tilde c(m,t) - (c_t^{0,\infty}(0)-c_t^{0,L}(0))$
and $\gamma_t(m,L)= e^{-\frac12 \tilde c(m,L,t)}$,
we have shown
\begin{align}
  \epsilon^{-1/4\pi} e^{-c_t^{\epsilon,L}(0)/2} = L_t^{-1/4\pi}(\gamma_t(m,L) + O(\frac{\epsilon^2}{t})+O_{m,L}(\epsilon^{\kappa})).
\end{align}
The claim is the special case $L=m=1$.
\end{proof}

\begin{proof}[Proof of \eqref{e:cxy-limit}]
  It follows directly from the definition of $c_t(x,y)$ that, for any $\kappa \in (0,1)$,
  \begin{multline}
    |c_t^0(x,y)-c_t^\epsilon(x^\epsilon,y^\epsilon)|
    \\
    \leq
    \int_0^{\kappa^2|x-y|^2} (p_s^0(x,y)+p_s^\epsilon(x^\epsilon,y^\epsilon)) \, ds
    + \int_{\kappa^2|x-y|^2}^\infty |p_s^0(x,y)-p_s^\epsilon(x^\epsilon,y^\epsilon)| e^{-m^2 s} \, ds.
  \end{multline}
  By \eqref{e:ptbounds} respectively \eqref{e:ptlimit}, the two terms on the right-hand side are bounded by multiples of
  \begin{equation}
    \int_0^{\kappa^2|x-y|^2} e^{-c|x-y|/\sqrt{t}} \, \frac{dt}{t}
    \lesssim e^{-c/(2\kappa)},
    \qquad
    \epsilon^2 \int_{\kappa^2|x-y|^2}^\infty t^{-2} \, dt
    \lesssim
    \frac{\epsilon^2}{\kappa^2 |x-y|^2},
  \end{equation}
  and choosing say $\kappa = c/(4\log (|x-y|/\epsilon))$ gives the claim.
\end{proof}

\begin{proof}[Proof of \eqref{e:ew-limit}]
In the definition of $(w_t-w_s)(\xi_1,\dots,\xi_n)$
there are $n^2$ pairings between the charges $\xi_1,\dots, \xi_n$.
We ignore the signs of the charges and bound each pairing simply
using \eqref{e:pttorusbounds}--\eqref{e:pttoruslimit}:
For $t\geq s \geq \epsilon^2$, this gives
\begin{equation}
  \int_s^t |\dot u_r^\epsilon(\xi_i^\epsilon,\xi_j^\epsilon) - \dot u_r^0(\xi_i,\xi_j)| \, dr
  \lesssim \int_s^t (\frac{\epsilon^2}{r})^{1-\kappa} \, \pa{\frac{1}{r}+\frac{1}{L^2}}\, \theta_r\, dr
  \lesssim (\frac{\epsilon^2}{s})^{1-\kappa} 
  ,
\end{equation}
uniformly in $\xi_i$ and $\xi_j$.
Summing over all $O(n^2)$ pairs of charges, therefore
\begin{equation}
  \label{e:w-limit}
  \sup_{\xi_1,\dots,\xi_n} |w_t^\epsilon(\xi^\epsilon)-w_s^\epsilon(\xi^\epsilon) - (w_t^0(\xi)-w_s^0(\xi))|
  \lesssim n^2 (\frac{\epsilon^2}{s})^{1-\kappa}
\end{equation}
and \eqref{e:ew-limit} follows.
\end{proof}

\begin{proof}[Proof of \eqref{e:c-limit}]  
The first bound in \eqref{e:c-limit} is immediate from \eqref{e:pttorusbounds}
and \eqref{e:ptbounds}.
For the convergence bound in \eqref{e:c-limit}, we can take the sum of
\eqref{e:pttorus} and \eqref{e:ptlimit} over $x \in\Omega_\epsilon$
to see that, for $t\geq \epsilon^2$,
\begin{equation}
  \epsilon^2\sum_{x\in \Omega_\epsilon} |p^{\epsilon,L}_t(x) - p_t^{0,L}(x)|
  \leq \epsilon^2\sum_{x\in \epsilon\Z^2} |p^\epsilon_t(x) - p_t^0(x)|
  \lesssim \frac{\epsilon^2}{t} + \frac{\epsilon^{2k}}{t^{k}}
  \lesssim \frac{\epsilon^2}{t}.
\end{equation}
For $t \leq \epsilon^2$ the left-hand side is trivially bounded by $O(1) \leq O(\epsilon^2/t)$.
\end{proof}

We also record the following lemma providing some basic estimates for the regularised
GFF. 
Here recall that for $\epsilon>0$ we use $\EE_{c^\epsilon}$ to denote the expectation of the Gaussian measure on $X_\epsilon$ with covariance
matrix $c^{\epsilon}$.
Moreover, when $c^0$ is a smooth covariance kernel on $\Omega$, we denote by $\EE_{c^0}$ the Gaussian measure
supported on $C^\infty(\Omega)$ with covariance $c^0$.

\begin{lemma} \label{lem:Gauss-conv}
Let $t_0>0$. Then (with all constants depending on $t_0$ and $m$ but uniform in $\epsilon \geq 0$)
\begin{equation}\label{e:Gaussbd}
  \sup_{1 \geq \epsilon \geq 0}\sup_{t \geq t_0}\EE_{c_t^\epsilon-c_{t_0}^\epsilon}\pa{e^{O(\norm{\partial_\epsilon\zeta}_{L^\infty(\Omega_\epsilon)})}} \lesssim 1
  .
\end{equation}
Assume that $F_\epsilon\colon X_\epsilon \to \R$ and $F_0 \colon C(\Omega)\to \R$
satisfy the uniform bounds
$|F_\epsilon(\zeta)| \leq e^{O( \| \partial_\epsilon \zeta\|_{L^\infty(\Omega_\epsilon)} ) }$
and moreover $|F_0(\zeta)| \leq e^{O( \| \partial \zeta\|_{L^\infty(\Omega)} )}$
and $F_\epsilon(\varphi^\epsilon) \to F_0(\varphi)$ as $\epsilon\to 0$
if $\max_{x \in\Omega} |\varphi^\epsilon(x^\epsilon)-\varphi(x)| \to 0$.
Then
\begin{equation}\label{e:conv-lattice-expectations}
\EE_{c_t^\epsilon-c_{t_0}^\epsilon} F_\epsilon - \EE_{c_t^0-c_{t_0}^0} F_0 \to 0.
\end{equation}
Similarly, if
$F_\epsilon\colon X_\epsilon \to X_\epsilon$ and  $F_0\colon C(\Omega) \to C(\Omega)$ satisfy
$\|F_\epsilon(\zeta)\|_{L^\infty(\Omega_\epsilon)} \leq e^{O( \| \partial_\epsilon \zeta\|_{L^\infty(\Omega_\epsilon)} ) }$
and $\|F_0(\zeta)\|_{L^\infty(\Omega)} \leq e^{O( \| \partial \zeta\|_{L^\infty(\Omega)} ) }$
and moreover
$\|F_\epsilon(\varphi^\epsilon) - F_0(\varphi)\|_{L^\infty(\Omega_\epsilon)} \to 0$ as $\epsilon \to 0$
if $\max_{x \in\Omega} |\varphi^\epsilon(x^\epsilon)-\varphi(x)| \to 0$,
then
\begin{equation}\label{e:conv-lattice-expectations-Linfty}
  \|\EE_{c_t^\epsilon-c_{t_0}^\epsilon} F_\epsilon - \EE_{c_t^0-c_{t_0}^0} F_0\|_{L^\infty(\Omega_\epsilon)} \to 0.
\end{equation}
\end{lemma}
 
For the proof, the following coupling of the 
Gaussian measures with expectations $\EE_{c^0}$ and $\EE_{c^\epsilon}$
where
$c^\epsilon = c_t^\epsilon- c_{t_0}^\epsilon$ and $c^0 = c_t^0- c_{t_0}^0$
is convenient.
Let $\Omega^* = 2\pi \Z^2$ and $\Omega_\epsilon^* = \{k \in 2\pi \Z^2: -\pi/\epsilon < k_i \leq \pi/\epsilon \}$
be the Fourier duals of $\Omega$ and $\Omega_\epsilon$.
Let $(X(k))_{k\in 2\pi \Z^2}$ be a collection of independent complex standard Gaussian random variables
subject to $X(k) =\overline{X(-k)}$ for $k\neq 0$ and $X(0)$ is a real standard Gaussian random variable  (all defined on a common probability space).
Let $\hat q^0(k)$ denote the Fourier coefficients of  $q^0\colon \Omega\to \R$ where $[q^0 * q^0](x-y) = c^0(x,y)$ and $*$ denoting the convolution of two functions.
Then the random variable $\Phi$ defined by
\begin{equation}\label{e:gaussian-c0}
\Phi(x) = \sum_{k \in \Omega^*} \hat q^0(k) e^{ik\cdot x} X(k)
\end{equation}
takes almost surely values in $C^\infty(\Omega)$ and is Gaussian with mean zero and covariance $c^0$, i.e.,
\begin{equation}
  \E[\avg{f,\Phi} \avg{g,\Phi}]
  = \avg{c^0f,g}
\end{equation}
holds for all $f,g\in L^2(\Omega)$. 
Moreover, since $c^0$ is a positive real symmetric function, 
$\hat q^0(k) = \sqrt{\hat c^0(k)}$.
Similarly, let $q^\epsilon\colon \Omega_\epsilon \to \R$ be such that $[q^\epsilon * q^\epsilon](x-y) = c^\epsilon(x,y)$,
where $*$ now denotes convolution on the lattice,
and let $\hat q^\epsilon (k)$ be the Fourier coefficients of $q^\epsilon$, 
i.e.\ for $x\in \Omega_\epsilon$,
\begin{equation}\label{e:phi-eps}
q^\epsilon(x) = \sum_{k \in \Omega_\epsilon^*} \hat q^\epsilon (k) e^{ik\cdot x}.
\end{equation}
Then the random function $\Phi^\epsilon$ defined by
\begin{equation}\label{e:gaussian-c-eps}
\Phi^\epsilon(x) = \sum_{k \in \Omega_\epsilon^*} \hat q^\epsilon(k) e^{ik\cdot x} X(k)
\end{equation}
restricted to $X_\epsilon$
is multivariate Gaussian with covariance matrix $c^\epsilon$. 

\begin{proof}
Let $- \hat \Delta^0(k) = |k|^2$ and $- \hat \Delta^\epsilon(k)= \epsilon^{-2}\sum_{i=1}^d (2-2\cos(\epsilon k_i))$
be the Fourier multipliers of the continuum and lattice Laplacians, respectively,
and let $\Omega_\epsilon^* = \{k \in 2\pi \Z^2: -\pi/\epsilon < k_i \leq \pi/\epsilon \}$
be the Fourier dual of $\Omega_\epsilon$. 
For $k \in \Omega_\epsilon^*$, then
\begin{equation} \label{e:fourier-multipliers-difference}
  0 \leq -    \hat \Delta^0(k) + \hat \Delta^\epsilon(k)
  = \sum_{i=1}^d (k_i^2-\epsilon^{-2}(2-2\cos(\epsilon k_i)))
  \leq |k|^2 h(\epsilon k)
\end{equation}
where $h(x) = \max_{i=1,2} (1 -x_i^{-2}(2-2\cos(x_i)))$ satisfies $h(x) \in [0,1-\kappa]$ with $\kappa = 4/\pi^2$ for $|x|\leq \pi$
and $h(x) =O(|x|^2)$.

Denote $c^\epsilon = c_t^\epsilon - c_{t_0}^\epsilon$ and $c^0 = c_t^0 - c_{t_0}^0$.
Since,  for $\epsilon \geq 0$,
\begin{equation}
0 \leq \hat c^\epsilon(k) = \hat c_t^\epsilon(k) - \hat c_{t_0}^\epsilon(k) 
=  \int_{t_0}^t  e^{s \hat\Delta^\epsilon(k) - sm^2} ds ,
\end{equation}
we see from \eqref{e:fourier-multipliers-difference} that
\begin{align}\label{e:fourier-coef-bound-1}
0 \leq \hat c^0(k) &\leq \frac{1}{|k|^2+ m^2} e^{-t_0 (|k|^2 + m^2)}, &\qquad 0\leq \hat q^0(k) &\leq \frac{1}{\sqrt{|k|^2+ m^2}} e^{-(t_0/2) (|k|^2 + m^2)},
  \\
  \label{e:fourier-coef-bound-1eps}
0 \leq \hat c^\epsilon(k) &\leq \frac{1}{\kappa|k|^2+ m^2} e^{-t_0 (\kappa|k|^2 + m^2)}, &\qquad 0\leq \hat q^\epsilon(k) &\leq \frac{1}{\sqrt{\kappa |k|^2+ m^2}} e^{-(t_0/2) (\kappa |k|^2 + m^2)} .
\end{align}
Similar, but somewhat more tedious computations give 
\begin{equation}\label{e:fourier-coeff-difference-1}
  0 \leq \hat q^\epsilon(k) - \hat q^0(k) \lesssim e^{-(t_0/2) (\kappa |k|^2 + m^2)} h(\epsilon k)^{1/2}. 
\end{equation}
%
To prove \eqref{e:Gaussbd} we note that $\EE_{c^\epsilon}\|\zeta\|_{H^\alpha(\Omega_\epsilon)}^2  \leq O_\alpha(1)$ for any $\alpha \in \R$, where
\begin{equation}
\|\zeta\|_{H^\alpha(\Omega_\epsilon)}^2=
  \sum_{k \in \Omega_\epsilon^*} (1+ |k|^2)^\alpha |\hat \zeta(k) |^2.
\end{equation}
Indeed, with $\Phi^\epsilon$ as in \eqref{e:phi-eps}, 
\begin{equation}
\EE_{c^\epsilon}\|\zeta\|_{H^\alpha(\Omega_\epsilon)}^2 
= \sum_{k \in \Omega_\epsilon^*} (1+ |k|^2)^\alpha \E|\hat q^\epsilon(k)X(k) |^2
= \sum_{k \in \Omega_\epsilon^*} (1+ |k|^2)^\alpha |\hat q^\epsilon(k)|^2 \leq O_\alpha(1)
\end{equation}
where we used \eqref{e:fourier-coef-bound-1}.
By the Sobolev inequality $\|\partial_\epsilon f\|_{L^\infty(\Omega_\epsilon)} \leq C_\alpha \|f\|_{H^\alpha(\Omega_\epsilon)}$
which holds for any $\alpha> d/2+1$ (uniformly in $\epsilon \geq 0$)
and Gaussian concentration (for example, \cite[Theorem~2.1.1]{MR2319516}), hence
\begin{equation} \label{e:Gaussbd-pf}
  \EE_{c^\epsilon}
  e^{O(\norm{\partial_\epsilon \zeta}_{L^\infty(\Omega_\epsilon)})}
  \lesssim
  \EE_{c^\epsilon}
  e^{O(\norm{\zeta}_{H^\alpha(\Omega_\epsilon)})}
  \lesssim 1.
\end{equation}
In fact, under the coupling \eqref{e:gaussian-c-eps}, one has
\begin{equation} \label{e:Gaussbd-coupling}
  \E\pa{\sup_{\epsilon>0} e^{O(\|\partial_\epsilon \Phi^\epsilon\|_{L^\infty(\Omega_\epsilon)})}} \lesssim 1.
\end{equation}
This can be seen from $\|\partial_\epsilon \Phi^\epsilon\|_{L^\infty(\Omega_\epsilon)} \leq C_\alpha \|\Phi^\epsilon\|_{H^\alpha(\Omega_\epsilon)}$
and then bounding $\|\Phi^\epsilon\|_{H^\alpha(\Omega_\epsilon)}^2\lesssim \sum_{k \in \Omega_\epsilon^*} (1+|k|^2)^\alpha \hat q^\epsilon(k)^2 |X(k)|^2$
which is bounded uniformly in $\epsilon$ by \eqref{e:fourier-coef-bound-1eps}.

To prove \eqref{e:conv-lattice-expectations}  and \eqref{e:conv-lattice-expectations-Linfty},
we will use that almost surely under the coupling introduced above the statement of the lemma,
\begin{equation} \label{e:conv-sup-phi-phi-eps}
\sup_{x\in \Omega} | \Phi^\epsilon (x^\epsilon)  - \Phi(x) | \to 0
\end{equation}
as $\epsilon \to 0$. 
Indeed,
\begin{multline}
| \Phi^\epsilon(x^\epsilon)  - \Phi(x) | \leq   
| \Phi^\epsilon(x^\epsilon)  - \Phi(x^\epsilon) | 
+ | \Phi (x^\epsilon)  - \Phi(x) | 
\\
\leq  \sum_{k\in\Omega_\epsilon^*} |\hat q^\epsilon (k) - \hat q^0(k)| |X(k)|  
+ \sum_{k \in \Omega^*\setminus \Omega_\epsilon^*} |\hat q^0(k)| |X(k)|  + 
\sum_{k\in \Omega^*} |\hat q^0(k) X(k)| |k| |x^\epsilon-x|
\end{multline}
and the second and the third sum converge to $0$ almost surely since
\begin{equation}
\sum_{k\in \Omega^*} |\hat q^0(k) X(k)| <\infty \text{ a.s\ } \text{ and } \sum_{k\in \Omega^*} |\hat q^0(k) ||k| |X(k)| <\infty \text{ a.s\ }.
\end{equation}
These hold by \eqref{e:fourier-coef-bound-1} since
\begin{equation}
  \E\big[\sum_{k\in \Omega^*} (1+|k|)|\hat q^0(k) X(k)| \big] \lesssim \sum_{k\in \Omega^*} (1+|k|)|\hat q^0(k)|<\infty.
\end{equation}
To see that also the first sum converges to $0$ almost surely, we use the estimate \eqref{e:fourier-coeff-difference-1} 
together with the Borel-Cantelli lemma. 
Let $E_k= \{ |X(k)| > |k| \}$. Then $\P(E_k) \leq e^{-|k|^2/2}$ and hence $\sum_{k\in \Omega^*} \P(E_k) < \infty$.
By the Borel-Cantelli lemma, $\P(E_k \text{ infinitely often} ) = 0$. Let $M$ be the smallest natural number such that $E_k^c$ occurs for all $k>M$. Note that on the event $\{E_k \text{ finitely often} \}$ we have $M< \infty$ and hence
\begin{align}
\sum_{k\in\Omega_\epsilon^*} |\hat q^\epsilon (k) - \hat q^0(k)| |X(k)|  
&\lesssim  \sum_{k\in\Omega_\epsilon^*} e^{-{(t_0/2)(\kappa|k|^2+m^2)}} h(\epsilon k) |X(k)|
\nnb 
&\leq \sum_{|k|\leq M} e^{-{(t_0/2)(\kappa|k|^2+m^2)}} h(\epsilon k) |X(k)| + \sum_{|k|> M}  e^{-(t_0/2)(\kappa|k|^2+m^2)} h(\epsilon k) |k|.
\label{e:sums-bc}
\end{align}
Taking $\epsilon \to 0$, the first sum converges to $0$ as it has only finitely many terms. The same holds for the second sum by dominated convergence with respect to the counting measure on $\Omega^*$.

Now the statement \eqref{e:conv-lattice-expectations} follows from
\begin{equation}\label{e:dct-lattice-expectations}
\EE_{c_t^\epsilon-c_{t_0}^\epsilon} F_\epsilon = \E F_\epsilon (\Phi^\epsilon) \to  \E F_0 (\Phi) = \EE_{c_t^0-c_{t_0}^0} F_0
\end{equation}
where we used the dominated convergence theorem. Note that by \eqref{e:conv-sup-phi-phi-eps} and the assumption on $F_\epsilon$ and $F_0$, we have
$F_\epsilon (\Phi^\epsilon) \to F_0 (\Phi)$ almost surely
and moreover
\begin{equation}
  |F_\epsilon (\Phi^\epsilon)| \lesssim e^{O(\| \partial_\epsilon \Phi^\epsilon \|_{L^\infty(\Omega_\epsilon)})}
  \lesssim e^{O(\sup_{\epsilon}\| \partial_\epsilon \Phi^\epsilon \|_{L^\infty(\Omega)})}
\end{equation}
where the right-hand side is integrable by \eqref{e:Gaussbd-coupling}.

To prove \eqref{e:conv-lattice-expectations-Linfty} we use the embedding $\exteps$ and obtain
\begin{equation}
\| \exteps \EE_{c_t^\epsilon-c_{t_0}^\epsilon} F_\epsilon - \EE_{c_t^0-c_{t_0}^0} F_0 \|_{L^\infty(\Omega)} 
\leq  \E \| \exteps F_\epsilon (\Phi^\epsilon) -  F_0(\Phi) \|_{L^\infty(\Omega)} \to 0,
\end{equation}
where the convergence follows from an extension of the dominated convergence theorem to Banach space valued functions. Note that in this case, we have
\begin{equation}
\exteps F_\epsilon (\Phi^\epsilon) \to F_0(\Phi) \text{~a.s.\ in~} L^\infty(\Omega)
\end{equation}
by \eqref{e:conv-sup-phi-phi-eps} and the properties of $F_\epsilon$ and $F_0$, 
and moreover
\begin{equation}
  \| \exteps F_\epsilon (\Phi^\epsilon) \|_{L^\infty(\Omega)} = \|F_\epsilon (\Phi^\epsilon)\|_{L^\infty(\Omega_\epsilon)} \lesssim e^{O(\| \partial_\epsilon \Phi^\epsilon \|_{L^\infty(\Omega_\epsilon)})}
  \lesssim e^{O( \sup_{\epsilon} \| \partial_\epsilon \Phi^\epsilon \|_{L^\infty(\Omega_\epsilon)})}
\end{equation}
where the right-hand side is integrable by \eqref{e:Gaussbd-coupling}.
This shows that the assumption for the dominated convergence theorem are satisfied.
\end{proof}

\subsection{Proof of Theorem~\ref{thm:v-limit} for \texorpdfstring{$\beta<4\pi$}{beta<4pi}}
\label{sec:v-limit-4pi}

We first prove Theorem~\ref{thm:v-limit} in  the simpler case $\beta<4\pi$.
To this end, recall the definitions of $\tilde v^{n,\epsilon}$ from
\eqref{e:polchinski-fourier-duhamel-rescaled1}--\eqref{e:polchinski-fourier-duhamel-rescaled}
and those of  $\tilde v^{n,0}$ from
\eqref{e:polchinski-fourier-duhamel-0-1}--\eqref{e:polchinski-fourier-duhamel-0},
as well as the notation $\exteps f$ for the piecewise constant extension of a function $f$ on $\Omega_\epsilon$
introduced above Lemma~\ref{lem:C-limit}.

\begin{lemma} \label{lem:vtbd}
Let $\beta<4\pi$. Then for all $\epsilon \geq 0$, $n \geq 2$, and $t \geq 0$,
\begin{equation} \label{e:vtbd}
  L_t^2 \sup_{\xi} \norm{\tilde v_{t}^{n,\epsilon}(\xi, \cdot)}_{L^1((\Omega_\epsilon\times\{\pm 1\})^{n-1})}
  \leq n^{n-2} (C_\beta|\lZ_t^\epsilon|)^{n}.
\end{equation}
Moreover, for any $t>0$, as $\epsilon \to 0$,
\begin{equation} \label{e:vdiff-bd}
  L_t^2 \sup_{\xi} \norm{\exteps \tilde v_{t}^{n,\epsilon}(\xi,\cdot)- \tilde v_t^{n,0}(\xi,\cdot)}_{L^1((\Omega\times\{\pm 1\})^{n-1})}
  \to 0.
\end{equation}
\end{lemma}

\newcommand{\Lpmeps}[1]{L^1((\Omega_{\epsilon} \times \{\pm 1\})^{#1})}
\newcommand{\Lpm}[1]{L^1((\Omega \times \{\pm 1\})^{#1})}
\begin{proof}
The bound \eqref{e:vtbd} is proved in \cite[Proposition~3.5 and (3.49)]{1907.12308}.
Recall that we use a different normalisation here than in \cite{1907.12308};
see Section~\ref{sec:rescaling} below for translation.
Moreover, the argument there is stated for $\epsilon>0$, but all estimates are uniform
in $\epsilon$ and hold without changes also when $\epsilon=0$.
More precisely, the following is shown in the proof of
\cite[Proposition~3.5]{1907.12308}.
Writing \eqref{e:polchinski-fourier-duhamel-rescaled} as
\begin{equation}
  \tilde v^{n,\epsilon}_t(\xi_1, \dots, \xi_n) = \int_0^t r^{n,\epsilon}_{s,t}(\xi_1, \dots, \xi_n) \, ds,
\end{equation}
where 
\begin{equation} \label{e:rdef}
  r_{s,t}^{n,\epsilon}(\xi_1,\dots,\xi_n)
  =
  \sum_{I_1 \dot\cup I_2= [n]}
  \sum_{i\in I_1,j\in I_2}
  \beta  \sigma_i\sigma_j \dot c_s^\epsilon(x_i,x_j)
  \tilde v_{s}^\epsilon(\xi_{I_1}) \tilde v_{s}^\epsilon(\xi_{I_2}) 
  e^{-(w^\epsilon_{t}-w^\epsilon_{s})(\xi_1,\dots,\xi_n)},
\end{equation}
there are $R^n_{s,t} >0$ such that
\begin{equation}\label{e:rbd1}
  \sup_{\epsilon \geq 0} \sup_{\xi_1}\norm{\exteps r^{n,\epsilon}_{s,t}(\xi_1,\cdot)}_{\Lpm{n-1}} \leq R^n_{s,t}
\end{equation}
and
\begin{equation} \label{e:rbd2}
  L_t^2 \int_0^t R^{n}_{s,t} \, ds \leq n^{n-2} (C_\beta |\lZ_t|)^n.
\end{equation}
We now deduce the convergence \eqref{e:vdiff-bd} as follows.
Assume by induction that \eqref{e:vdiff-bd} holds for all $t>0$ and $n \leq k$.
For $k=1$, this follows from \eqref{e:c0-limit}
and the definitions of $v_t^{1,\epsilon}$ in \eqref{e:polchinski-fourier-duhamel-rescaled1}
and of $v_t^{1,0}$ in \eqref{e:polchinski-fourier-duhamel-0-1}.
To advance the induction, note that the definition of $r_{s,t}^{n,\epsilon}$ above depends on $\tilde v^{m,\epsilon}_s$ only with $m<n$.
Therefore, using the inductive assumption for the $\tilde v_s$ terms in \eqref{e:rdef},
and Lemma~\ref{lem:C-limit} for $e^{-(w_t-w_s)}$ and $\dot c_s$ in \eqref{e:rdef},
it follows that, for any $0<s<t$, as $\epsilon \to 0$,
\begin{equation}
  \sup_{\xi_1}\norm{\exteps r^{k+1,\epsilon}_{s,t}(\xi_1,\cdot)- r^{k+1,0}_{s,t}(\xi_1,\cdot)}_{\Lpm{k}} \to 0.
\end{equation}
By dominated convergence for the $s$-integral, using \eqref{e:rbd2}, it then follows that, as $\epsilon \to 0$,
\begin{equation}
  \norm{\exteps \tilde v^{k+1,\epsilon}_{s,t}(\xi_1,\cdot)- \tilde v^{k+1,0}_{s,t}(\xi_1,\cdot)}_{\Lpm{k}} \to 0,
\end{equation}
and the induction is advanced.
\end{proof}

\begin{lemma} \label{lem:nablavt-limit}
  Let $\beta < 4\pi$, $z\in \R$.
  Then for all $n \geq 2$ and $k\in \N$, uniformly in $\varphi$ and $\epsilon \geq 0$,
  \begin{align} 
        \label{e:vtn-bd}
    |v_t^{n,\epsilon}(\varphi)|
    &\leq (C_{\beta}|\lZ_{t\wedge \epsilon}|)^n,
    \\
    \label{e:nablavtn-bd}
    L_t^{2+k} \norm{\partial_\epsilon^k\dot c_t^\epsilon\nabla v_t^{n,\epsilon}(\varphi)}_{L^\infty(\Omega_\epsilon)}
    &\leq C_k(C_{\beta}|\lZ_{t\wedge \epsilon}|)^n,
    \\
    \label{e:Hessvtn-bd}
    L_t^{2+k} \norm{\partial_\epsilon^k\dot c_t^\epsilon\He v_t^{n,\epsilon}(\varphi)g}_{L^\infty(\Omega_\epsilon)}
    &\leq C_k \norm{g}_{L^\infty(\Omega_\epsilon)} (C_{\beta}|\lZ_{t\wedge \epsilon}|)^n
      .
  \end{align}
  Moreover, for any $t>0$ and $\max_{x\in\Omega_\epsilon} |\varphi^\epsilon(x^\epsilon) -\varphi(x)| \to 0$ as $\epsilon \to 0$,
  \begin{align} 
    \label{e:vt-limit}
    |v_t^{n,\epsilon}(\varphi^\epsilon) - v_t^{n,0}(\varphi)|
    &\to 0
      ,
    \\
    \label{e:nablavt-limit}
        L_t^2 \norm{\dot c_t^\epsilon\nabla v_t^{n,\epsilon}(\varphi^\epsilon)-\dot c_t^0\nabla v_t^{n,0}(\varphi)}_{L^\infty(\Omega_\epsilon)}
    &\to 0
    .
    \end{align}
\end{lemma}

\begin{proof}
  The bounds \eqref{e:vtn-bd}--\eqref{e:Hessvtn-bd} follow easily by substituting \eqref{e:vtbd} into
  \eqref{e:v-fourier1}--\eqref{e:Hessvtn-fourier},
  exactly   as in \cite[Section~3]{1907.12308}.
  We consider this argument in more detail for \eqref{e:nablavtn-bd}.
  To this end, let
  \begin{equation}
    r_t^{\epsilon}(x,\xi_1, \dots, \xi_n)
    = \frac{i \sqrt{\beta}}{n!}
    \pa{\sum_{i=1}^n \partial_\epsilon^k \dot c_t^\epsilon(x,x_i)\sigma_i} \tilde v_t^{n,\epsilon}(\xi_1,\dots,\xi_n)
  \end{equation}
  so that \eqref{e:nablavtn-fourier} can be written as
  \begin{equation}
    \partial_\epsilon^k \dot c_t^\epsilon\nabla v_t^{n,\epsilon}(\varphi,x)
    =
    \int_{(\Omega_\epsilon \times \{\pm 1\})^n} \, d\xi_1 \cdots d\xi_n \,
    r_t^{\epsilon}(x,\xi_1,\dots,\xi_n) e^{i\sum_{i=1}^n  \sigma_i \varphi(x_i)}.
  \end{equation}
  Using that $n^n/n! \leq e^{n}$,
  it then follows from \eqref{e:c-limit} and \eqref{e:vtbd} that
  \begin{multline} \label{e:nablavtn-rbd}
    L_t^{2+k} \sup_x \norm{r_t^{\epsilon}(x,\cdot)}_{L^{1}(\Omega_\epsilon \times \{\pm 1\})}
    \\
    \lesssim \frac{n}{n!} L_t^k \sup_x \norm{\partial_\epsilon^k c_t^\epsilon(x,\cdot)}_{\Lpmeps{}}
      L_t^2 \sup_{\xi_1} \norm{\tilde v_t^{n,\epsilon}(\xi_1,\cdot)}_{\Lpmeps{n}}
    \lesssim C_k (C_{\beta}|\lZ_t|)^n.
  \end{multline}
  To show the convergence \eqref{e:nablavt-limit}, recall the notation $\xi^\epsilon$
  from Section~\ref{sec:C-limit}.
  Then (where now  $k=0$)
  \begin{equation}
    \dot c_t^\epsilon\nabla v_t^{n,\epsilon}(\varphi, x)
    =
    \int_{(\Omega \times \{\pm 1\})^n} \, d\xi_1 \cdots d\xi_n \,
    r_t^{\epsilon}(x,\xi_1^\epsilon,\dots,\xi_n^\epsilon) e^{i\sum_{i=1}^n \sigma_i \varphi(x_i^\epsilon)},
  \end{equation}
  and $\sup_x |\dot c_t^\epsilon\nabla v_t^{n,\epsilon}(\varphi, x)-      \dot c_t^0\nabla v_t^{n,0}(\varphi, x)|$
  is bounded by the sum of the following two terms:
  \begin{align}
    \label{e:nablavtn-pf1}
    &\sup_x \int_{(\Omega \times \{\pm 1\})^n} \, d\xi_1 \cdots d\xi_n \,
    \absB{r_t^{\epsilon}(x,\xi_1^\epsilon,\dots,\xi_n^\epsilon)
      -r_t^{0}(x,\xi_1,\dots,\xi_n)}
    ,
    \\
    \label{e:nablavtn-pf2}
    &\sup_x \int_{(\Omega \times \{\pm 1\})^n} \, d\xi_1 \cdots d\xi_n \,
    |r_t^{\epsilon}(x,\xi_1^\epsilon,\dots,\xi_n^\epsilon)|
    \absB{
      e^{i\sum_{i=1}^n \sigma_i \varphi(x_i^\epsilon)}-e^{i\sum_{i=1}^n \sigma_i \varphi(x_i)}
    }
    .
  \end{align}
  The first term \eqref{e:nablavtn-pf1} converges to $0$ as $\epsilon\to 0$
  by \eqref{e:c-limit} and \eqref{e:vdiff-bd}.
  For the second term \eqref{e:nablavtn-pf2},
  we use the assumption $\sup_{x\in\Omega}|\varphi^\epsilon(x^\epsilon) - \varphi(x)|\leq a_\epsilon$ with $a_\epsilon \to 0$
  which implies
  \begin{equation}
    \absB{
      e^{i\sum_{i=1}^n  \sigma_i \varphi(x_i^\epsilon)}-e^{i\sum_{i=1}^n\sigma_i  \varphi(x_i)}
    }
    \leq O(n a_\epsilon).
  \end{equation}
  Hence
  \eqref{e:nablavtn-pf2} is bounded by
  $O(na_\epsilon) \sup_x \norm{r^\epsilon(x,\cdot)}_{\Lpmeps{n}}$
  and the claim follows from \eqref{e:nablavtn-rbd}. 
\end{proof}

\begin{proof}[Proof of Theorem~\ref{thm:v-limit} for $\beta<4\pi$] 
    In order to treat the two cases $\epsilon>0$ and $\epsilon=0$ simultaneously, we will typically omit the superscript $\epsilon$,
    and we also make the convention that for $\epsilon>0$ the space $C(\Omega_\epsilon)$ simply refers to the finite dimensional space $X_\epsilon$.
  We define $t_0 = t_0(\beta,z,m)$ as the largest $t_0>0$ with
  $C_\beta|\lZ_{t_0}| \leq \alpha_0$ for a sufficiently small constant $\alpha_0$;
  this $t_0$ is independent of $\epsilon$.
  We then start with showing the bounds in the case $t \leq t_0$.
  In this case, by \eqref{e:nablavtn-bd}, the sum on the right-hand side of
  \begin{equation}
    \dot c_t \nabla v_t(\varphi) = \sum_{n=0}^{\infty} \dot c_t \nabla v_t^{n}(\varphi)
  \end{equation}
  converges in $C(\Omega)$, uniformly in $\varphi\in C(\Omega)$,
  and for any $n_0$, we have
\begin{equation}
  \dot c_t \nabla v_t(\varphi) = \sum_{n=0}^{n_0}\dot c_t\nabla v_t^{n}(\varphi) + O(L_t^{-2} |\lZ_t|^{n_0})
\end{equation}
where again the error  $O(L_t^{-2}|\lZ_t|^{n_0})$ is bounded in $C(\Omega)$ uniformly in $\varphi$.
In particular, with $n_0\to \infty$, the bound  \eqref{e:cnablav-bd} with $k=0$
follows for $t\leq t_0$.
The argument for $k \in \N$ is analogous,
and the bound \eqref{e:cnablav-Lip} also follows analogously from \eqref{e:Hessvtn-bd}.
To show the convergence \eqref{e:cnablav-eps}, for $t\leq t_0$,
we apply \eqref{e:nablavt-limit} to see
from the above that
\begin{align}\label{e:conv-c-nabla-v}
  L_t^2 \norm{\dot c_t^\epsilon \nabla v_t^\epsilon(\varphi^{\epsilon}) - \dot c_t^0\nabla v_t^0(\varphi)}_{L^\infty(\Omega_\epsilon)}
  &\leq \sum_{n=0}^{n_0} L_t^2 \norm{\dot c_t^\epsilon\nabla v_t^{n,\epsilon}(\varphi^{\epsilon})-\dot c_t\nabla v_t^{n,0}(\varphi)}_{L^\infty(\Omega_\epsilon)} + O(|\lZ_t|^{n_0})
  \nnb
  &\leq
    o_{n_0}(1)
    + O(|\lZ_t|^{n_0}) \to 0
\end{align}
by taking $\epsilon \to 0$ and then $n_0\to\infty$.
This establishes the convergence \eqref{e:cnablav-eps} for $t\leq t_0$.

We next extend the bounds to $t>t_0$.
Since $t_0$ is of order $1$,
$\|\dot c_{t_0} \nabla v_{t_0}(\varphi)\|_{L^\infty(\Omega)}$
is uniformly bounded,
and the same argument shows that $|v_{t_0}(\varphi)|$ is likewise uniformly bounded,
and that $v_{t_0}^\epsilon(\varphi^\epsilon) \to v_{t_0}^0(\varphi)$  when $\varphi^\epsilon \to \varphi$ as in the statement of the theorem.
From this, a simple argument suffices. Indeed,
$v_t$ can be obtained from $v_{t_0}$ using the convolution representation
\eqref{e:vdef-conv} as
\begin{align} \label{e:vt-conv}
  e^{-v_t(\varphi)} &= \EE_{c_{t}-c_{t_0}} (e^{-v_{t_0}(\varphi+\zeta)}),
  \\ \label{e:nablavt-conv}
  \partial^k\dot c_t \nabla v_t(\varphi)
  &= [\PP_{t_0,t} (\partial^k \dot c_{t} \nabla v_{t_0})](\varphi),
\end{align}
where
\begin{equation}
  \PP_{t_0,t}F(\varphi) = e^{+v_t(\varphi)} \EE_{c_{t}-c_{t_0}} (e^{-v_{t_0}(\varphi+\zeta)} F(\varphi+\zeta)).
\end{equation}
This representation holds for any $\epsilon>0$ and we use it as our definition
in the limiting case $\epsilon= 0$.
To see that this is well-defined note that, by Jensen's inequality,
\begin{equation} \label{e:vt-Jensen}
  v_t(\varphi)
  = -\log \EE_{c_t-c_{t_0}}(e^{-v_{t_0}(\varphi+\zeta)})
  \leq \EE_{c_t-c_{t_0}}(v_{t_0}(\varphi+\zeta))
  \leq O_{t_0,t}(1),
\end{equation}
and thus $e^{+v_t}$ is bounded.
By taking maxima over $\varphi$ and $x$,
\eqref{e:nablavt-conv} also immediately implies that
\begin{equation}
  \sup_\varphi\|\partial^k \dot c_t \nabla v_t(\varphi)\|_{L^\infty(\Omega)}
  \leq \sup_\varphi\|\partial^k \dot c_t \nabla v_{t_0}(\varphi)\|_{L^\infty(\Omega)}
  \lesssim \theta_{t-t_0} \sup_\varphi\|\partial^k \dot c_{t_0} \nabla v_{t_0}(\varphi)\|_{L^\infty(\Omega)}
\end{equation}
which gives \eqref{e:cnablav-bd} for $t\geq t_0$ since $t_0$ is of order $1$ meaning that
$L_t \simeq L_{t_0}$ and $\lZ_t \simeq \lZ_{t_0}$ (with constants depending on $m,L$).
The bound \eqref{e:cnablav-Lip} also follows similarly; see \cite[Section~3.8]{1907.12308}.

To see the convergence \eqref{e:cnablav-eps} for $t > t_0$, we apply \eqref{e:conv-lattice-expectations}.
Indeed, by the above analysis for $t\leq t_0$, when $\sup_{x\in\Omega}|\varphi^\epsilon(x^\epsilon)-\varphi(x)|\to 0$, we have
$v_{t_0}^\epsilon(\varphi^\epsilon) \to v_{t_0}^0(\varphi)$ and
$\dot c_{t}^\epsilon \nabla v^\epsilon_{t_0}(\varphi^\epsilon,x^\epsilon) \to \dot c_{t}^0 \nabla v_{t_0}^0(\varphi,x)$ uniformly in $x$.
By \eqref{e:conv-lattice-expectations}, therefore 
\begin{align}
  e^{-v_t^\epsilon(\varphi)} = \EE_{c_{t}^\epsilon-c_{t_0}^\epsilon} (e^{-v_{t_0}^\epsilon(\varphi+\zeta)})
  &\to
  \EE_{c_{t}^0-c_{t_0}^0} (e^{-v_{t_0}^0(\varphi+\zeta)})
  =e^{-v_t^0(\varphi)},
  \\
  \EE_{c_{t}^\epsilon-c_{t_0}^\epsilon}(\dot c_t^\epsilon \nabla v_{t_0}^\epsilon(\varphi+\zeta)e^{-v_{t_0}^\epsilon(\varphi+\zeta)})
  &\to
  \EE_{c_{t}^0-c_{t_0}^0}(\dot c_t^0 \nabla v_{t_0}^0(\varphi+\zeta)e^{-v_{t_0}^0(\varphi+\zeta)}) .
\end{align}
Since $e^{-v_{t}^\epsilon(\varphi)}$ is bounded below uniformly in $\epsilon \geq 0$
by \eqref{e:vt-Jensen}, the convergence of
\begin{align}
  \dot c_t \nabla v_t
  = e^{+v_{t}(\varphi)} \EE_{c_t-c_{t_0}} (\partial^k \dot c_{t} \nabla v_{t_0}(\varphi+\zeta) e^{-v_{t_0}(\varphi+\zeta)})
\end{align}
then follows.
\end{proof}

\subsection{Proof of Theorem~\ref{thm:v-limit} for \texorpdfstring{$\beta<6\pi$}{beta<6pi}}
\label{sec:vt-6pi}

Finally, we extend the proof of Theorem~\ref{thm:v-limit} to the more subtle regime $\beta<6\pi$.
In this case, the $n=2$ term requires a more careful treatment
as it turns out that, for $\beta \geq 4\pi$,
\begin{equation}
  \norm{\tilde v_t^{2,0}(\xi,\cdot)}_{L^1(\Omega \times \{\pm 1\})} = \infty.
\end{equation}
Indeed, explicitly by \eqref{e:polchinski-fourier-duhamel-rescaled1}--\eqref{e:polchinski-fourier-duhamel-rescaled},
\begin{align}
  \tilde v_{t}(\xi_1,\xi_2)
  &=
    \int_0^{t} ds \,
    \dot u_s(\xi_1,\xi_2) 
    \tilde v_{s}(\xi_{1}) \tilde v_{s}(\xi_{2}) 
    e^{- \frac12 \int_s^t \dot u_r(\xi_1,\xi_1) - \frac12 \int_s^t \dot u_r(\xi_2,\xi_2) -\int_s^t dr \, \dot u_r(\xi_1,\xi_2)},
  \nnb
  &=
    \tilde v_{t}(\xi_{1}) \tilde v_{t}(\xi_{2}) 
    \int_0^{t} ds \,
    \pa{\ddp{}{s} e^{-\int_s^t dr \, \dot u_r(\xi_1,\xi_2)}}
  =
    \tilde v_{t}(\xi_{1}) \tilde v_{t}(\xi_{2}) 
    (1-e^{-\beta \sigma_1\sigma_2 c_t(x_1,x_2)}).
    \label{e:s-explicit}
\end{align}

In the neutral case $\sigma_1+\sigma_2=0$, the term $e^{+\beta c_t(x_1,x_2)}$
diverges as $|x_1-x_2|^{-\beta/2\pi}$ as $|x_1-x_2|\to 0$,
when $\epsilon=0$,
and is therefore not integrable for $\beta \geq 4\pi$.
The remedy for this is as follows.
First, the \emph{charged} part $\tilde v^2(\xi_1,\xi_2) {\bf 1}_{\sigma_1+\sigma_2 \neq 0}$ remains integrable, and we will see that
it satisfies the same estimates as for $\beta<4\pi$.
Second, while (as observed above) the \emph{neutral} part $\tilde v^2(\xi_1,\xi_2) {\bf 1}_{\sigma_1+\sigma_2=0}$
is not integrable,
the following weaker estimates hold and we will see that these are sufficient to prove the theorem.
We write
\begin{align}
  \delta c_t(x_1,x_2,x_3) &= c_t(x_1,x_2)-c_t(x_2,x_3),
  \\
  \partial^k \delta c_t(x_1,x_2,x_3) &= \partial_{x_1}^k c_t(x_1,x_2)-\partial_{x_1}^k c_t(x_1,x_3),
\end{align}
where all quantities can depend on $\epsilon \geq 0$ (as always).

\begin{lemma} \label{lem:vtbd-6pi-1}
Let $\beta <6\pi$.
For $n=2$ and all $t>0$,
the charged part $\tilde v^2_t(\xi_1,\xi_2)\mathbf{1}_{\sigma_1+\sigma_2 \neq 0}$
satisfies the bounds \eqref{e:vtbd}--\eqref{e:vdiff-bd},
while the neutral part $\tilde v^2_t(\xi_1,\xi_2)\mathbf{1}_{\sigma_1+\sigma_2=0} =: \tilde s_t(x_1,x_2)$
satisfies the following replacement: for any $k\in \N$ and $\epsilon \geq 0$,
\begin{align}
  \label{e:v2-bd1}
  L_t^{2+k} \sup_x \norm{\partial^k_\epsilon \delta c_t^\epsilon(x,\cdot)\tilde s_t}_{L^1(\Omega_\epsilon^2)}
  &\leq C_{\beta,k} \lZ_t^2\theta_t^2,
  \\
  \label{e:v2-bd2}
  L_t^2 \sup_x \norm{\frac{|x-\cdot|}{L_t} \tilde s_t^\epsilon(x,\cdot)}_{L^1(\Omega_\epsilon)}
  &\leq C_{\beta,k} \lZ_t^2.
\end{align}
Moreover, for any $t>0$, as $\epsilon \to 0$,
\begin{align} \label{e:v2-diffeps1}
  L_t^2 \sup_x \norm{\exteps \delta c_t^\epsilon(x,\cdot) \exteps s_t^{\epsilon}- \delta c_t^0(x,\cdot)\tilde s_t^{0}}_{L^1(\Omega^2)}
      &\to 0
      ,
  \\
  \label{e:v2-diffeps2}
  L_t^2 \sup_x \norm{\frac{|x-\cdot|}{L_t}(\exteps s_t^{\epsilon}(x, \cdot) - \tilde s_t^{0}(x,\cdot))}_{L^1(\Omega)}
      &\to 0
  .
\end{align}
\end{lemma}

\begin{proof}
  The estimates all essentially follow from the formula \eqref{e:s-explicit}
  and heat kernel estimates.
  In more detail, for the charged part, the estimate \eqref{e:vtbd} is derived in \cite[(3.65)]{1907.12308}.
  For the neutral part, the bound \eqref{e:v2-bd1} is \cite[(3.76)]{1907.12308} when $k=0$
  and, as in \eqref{e:nablavtn-rbd}, the generalisation to $k\in \N$ is 
  exactly the same since $L_t^k\partial^k c_t$ satisfies
  the same estimates   as $c_t$ by \eqref{e:c-limit}.
  The proof of \eqref{e:v2-bd2}
  can similarly be seen from the representation
  \eqref{e:s-explicit} together with \cite[Lemma~A.4]{1907.12308}.

  For the convergence statements, we can again start from the formula \eqref{e:s-explicit}.
  By \eqref{e:cxy-limit},
  \begin{equation}
    \sup_{|x_1-x_2|\geq \kappa} \max_{\sigma_1,\sigma_2 \in \{\pm 1\}} |\tilde v^{2,\epsilon}_t(\xi_1^\epsilon,\xi_2^\epsilon) - \tilde v^{2,0}_t(\xi_1,\xi_2)| \to 0
    \qquad (\epsilon \to 0),
  \end{equation}
  for any $\kappa>0$.
  In the charged case, \eqref{e:vdiff-bd} follows because then
  \begin{equation}
    \sup_{\xi_1}
    \int_{\Omega\times\{\pm 1\}}
    \pB{ |\tilde v_t^{2,\epsilon}(\xi_1^\epsilon,\xi_2^\epsilon)|+|\tilde v_t^{2,0}(\xi_1,\xi_2)| } \, {\bf 1}_{|x_1-x_2| \leq \kappa} \, d\xi_2 \to 0
    \qquad (\kappa \to 0).
  \end{equation}
  The statements for the neutral case, \eqref{e:v2-diffeps1}, and \eqref{e:v2-diffeps2}, follow similarly.
\end{proof}

For $n\geq 3$, the recursive definition of $\tilde v(\xi_1,\dots, \xi_n)$ in
\eqref{e:polchinski-fourier-duhamel-rescaled}
and \eqref{e:polchinski-fourier-duhamel-0}
depends on the neutral part of $\tilde v^{2}$ only in the combination
\begin{equation}
  \delta c_t(x_1,x_2,x_3) \tilde v_t^2(\xi_2,\xi_3)\mathbf{1}_{\xi_2+\xi_3=0}
  =
  \delta c_t(x_1,x_2,x_3) \tilde s_t(x_2,x_3)
  .
\end{equation}
From this, one obtains the following estimates
generalising Lemmas~\ref{lem:vtbd}--\ref{lem:nablavt-limit}.

\begin{lemma} \label{lem:vtbd-6pi-2}
Let $\beta <6\pi$.
Then \eqref{e:vtbd}--\eqref{e:vdiff-bd} hold for $n\geq 3$.
\end{lemma}

\begin{proof}
  For $n\geq 3$ and $\beta<6\pi$,
  the bounds in \eqref{e:vtbd} with $k=0$ are shown in \cite[Proposition~3.7]{1907.12308},
  and  again their extension with general $k\in \N$ is exactly the same since $L_t^k\partial^k c_t$ satisfies
  the same estimates   as $c_t$. 
  Moreover, the $\epsilon \to 0$ convergence
  \eqref{e:vdiff-bd} for $n\geq 3$ 
  holds by similarly arguments as in
  in the proof of Lemma~\ref{lem:vtbd} for $\beta<4\pi$.
\end{proof}

The upshot of the above considerations
is that, for all $n\geq 3$, we can define and analyse $v^{n,0}_t(\varphi)$ in exactly the same way
for $\beta<6\pi$ as for $\beta<4\pi$.
On the other hand, while the definition of $v^{2,0}_t(\varphi)$ does not make sense when $\beta \geq 4\pi$,
it turns out that $v^{2,0}_t(\varphi)$ is actually well-defined up to an divergent additive constant.
More precisely, the $\epsilon\to 0$ limits of
the differences $v_t^{2,\epsilon}(\varphi+\zeta)-v_t^{2,\epsilon}(\varphi)$ and
of the $\varphi$-derivative $\dot c_t^\epsilon \nabla v_t^{2,\epsilon}(\varphi,x)$ continue to exist.
Abusing notation,
we will denote this limit by $v^{2,0}_t(\varphi+\zeta)-v^{2,0}_t(\varphi)$,
for example,
with the implicit understanding that only the difference of $v^{2,0}_t$ is defined.
The following lemma then provides the required replacement of Lemma~\ref{lem:nablavt-limit}.

\begin{lemma} \label{lem:v2}
  Let $\beta<6\pi$.
  For $n \geq 3$ and $t>0$, the bounds \eqref{e:nablavtn-bd}--\eqref{e:nablavt-limit} continue to hold. 
  For $n=2$ and $t>0$, the bounds \eqref{e:nablavtn-bd}--\eqref{e:Hessvtn-bd} and \eqref{e:nablavt-limit}
  also continue to hold,
  and the following replacements for \eqref{e:vtn-bd} and \eqref{e:vt-limit} hold:
  for any $\varphi$ and $\zeta$,
  \begin{equation}
    \label{e:v2-phi-bd2}
    |v^{2}_{t}(\varphi+\zeta)- v^{2}_{t}(\varphi)|
    \lesssim |\lZ_t|^2 (1+ L_t\norm{\partial \zeta}_{L^\infty(\Omega)}),
  \end{equation}
  and for any $t>0$,
  and $\max_{x\in\Omega}|\varphi^\epsilon(x^\epsilon)-\varphi(x)| \to 0$
  and $\max_{x\in\Omega}|\zeta^\epsilon(x^\epsilon)-\zeta(x)| \to 0$,
  as $\epsilon \to 0$,
  \begin{equation}
    \label{e:v2-phi-diffeps2}
    |[v^{2,0}_t(\varphi+\zeta)-v^{2,0}_t(\varphi)]-
    [v_t^{2,\epsilon}(\varphi+\zeta)-v_t^{2,\epsilon}(\varphi)]|
    \to 0
      .
  \end{equation}
\end{lemma}

\begin{proof}
  For $n\geq 3$ and the charged part of $n=2$,
  since \eqref{e:vtbd}--\eqref{e:vdiff-bd} continue to hold in these cases for $\beta<6\pi$ by
  Lemma~\ref{lem:vtbd-6pi-1}--\ref{lem:vtbd-6pi-2},
  the argument is the same as that for $\beta<4\pi$ in Lemma~\ref{lem:nablavt-limit}.
It thus suffices to consider the neutral part of $v_{t}^{2}(\varphi+\zeta)-v_{t}^2(\varphi)$
which is given by
\begin{equation}
  \epsilon^{4}\sum_{x_1,x_2 \in \Omega_\epsilon} \tilde v^2_{t}((x_1,+1),(x_2,-1))(\cos(\varphi(x_1)-\varphi(x_2)+\zeta(x_1)-\zeta(x_2))-\cos(\varphi(x_1)-\varphi(x_2))).
\end{equation}
Using that $|\cos(a+b)-\cos(a)| \leq |b|$
and \eqref{e:v2-bd2}, this is bounded uniformly in $\varphi$ by
\begin{equation}
  L_t^2 \sup_x \norm{\frac{|x-\cdot|}{L_t} \tilde v^2_{t}((x,+1),(\cdot,-1))}_{L^1(\Omega_\epsilon)} 
  \qbb{ \sup_{x_1,x_2} \frac{|\zeta(x_1)-\zeta(x_2)|}{|x_1-x_2|/L_t}}
  \lesssim |\lZ_t|^2 L_t \norm{\partial \zeta}_{L^\infty(\Omega_\epsilon)}.
\end{equation}
This proves \eqref{e:v2-phi-bd2}.
The convergence \eqref{e:v2-phi-diffeps2} is proved analogously using
\eqref{e:v2-diffeps2}.
\end{proof}

\begin{proof}[Proof of Theorem~\ref{thm:v-limit} for $\beta<6\pi$]
For $t\leq t_0$,
the proof is again identical to that for $\beta<4\pi$
since \eqref{e:nablavtn-bd}--\eqref{e:Hessvtn-bd} and \eqref{e:nablavt-limit}
continue to hold for $n\geq 2$ by Lemma~\ref{lem:v2}.

Thus let $t \geq t_0$.
As discussed at the beginning of Section~\ref{sec:vt-6pi},
$v_{t_0}(\varphi)$ is now divergent as $\epsilon\to 0$, but
differences $v_{t_0}(\varphi+\zeta)-v_{t_0}(\varphi)$ continue to make sense
in the limit. Indeed, 
by \eqref{e:v2-phi-bd2} for the $n=2$ term and \eqref{e:vtn-bd} for the $n\geq 3$ terms,
for any fixed $t_0>0$,
\begin{equation}
  e^{-v_{t_0}(\varphi+\zeta)+v_{t_0}(\varphi)} \lesssim e^{O(\|\partial \varphi\|_{L^\infty(\Omega)})}.
\end{equation}
Moreover, for any fixed $t_0>0$ and $t>t_0$,
the estimated for $v_{t_0}$ used above,
Jensen's inequality,
and $\EE_{c_t-c_{t_0}} \|\partial \zeta\|_{L^\infty(\Omega)} \leq O_{t_0}(1)$ which holds by Lemma~\ref{lem:Gauss-conv},
imply
\begin{align} \label{e:vt-Jensen-6pi}
  v_t(\varphi)-v_{t_0}(\varphi)
  &= -\log \EE_{c_t-c_{t_0}}(e^{-v_{t_0}(\varphi+\zeta) + v_{t_0}(\varphi)})
  \nnb
  &\leq \EE_{c_t-c_{t_0}}(v_{t_0}(\varphi+\zeta)-v_{t_0}(\varphi))
  \leq O_{t_0}(1).
\end{align}
Hence $e^{+v_t(\varphi)-v_{t_0}(\varphi)}$ is bounded.
We then again start from \eqref{e:nablavt-conv}, now interpreted for $\epsilon=0$ as
\begin{equation}
  \PP_{t_0,t}F(\varphi) = e^{+v_t(\varphi)-v_{t_0}(\varphi)}
  \EE_{c_{t}-c_{t_0}} (e^{-v_{t_0}(\varphi+\zeta)+v_{t_0}(\varphi)} F(\varphi+\zeta)).
\end{equation}
From this, the proofs of the bounds \eqref{e:cnablav-bd}--\eqref{e:cnablav-Lip} using the
already established case $t=t_0$ are identical
to those in case $\beta<4\pi$, and the convergence \eqref{e:cnablav-eps} is also similar
from \eqref{e:conv-lattice-expectations}.
\end{proof}

\subsection{Summary of rescaling}
\label{sec:rescaling}

To compare with the estimates from \cite[Section~3]{1907.12308},
we here summarise how the definitions change under the rescaling
$t \to t/\epsilon^2$ and $x \to x/\epsilon$
that we use here.

To distinguish both normalisations,
we will here denote objects in unit lattice normalisation by capital letters
and their versions in continuum normalisation by lower case latters.
For $\xi=(x,\sigma) \in \Lambda \times \{\pm 1\}$, where $\Lambda$ is a subset of
the unit lattice $\Z^2$ as in \cite[Section~3]{1907.12308}, below we write
$\xi/\epsilon = (x/\epsilon, \sigma) \in \Omega \times \{\pm 1\}$, and then note the correspondence
\begin{align}
  c_t^\epsilon(x_1,x_2) &= C_{t/\epsilon^2}^\epsilon(x_1/\epsilon,x_2/\epsilon),
  \\
  \dot c_t^\epsilon(x_1,x_2) &= \epsilon^{-2}\dot C_{t/\epsilon^2}^\epsilon(x_1/\epsilon,x_2/\epsilon),
  \\
  u_t(\xi_1,\xi_2) &= U_{t/\epsilon^2}(\xi_1/\epsilon, \xi_2/\epsilon),
  \\
  \dot u_t(\xi_1,\xi_2) &= \epsilon^{-2} \dot U_{t/\epsilon^2}(\xi_1/\epsilon, \xi_2/\epsilon),
  \\
  \tilde v_t(\xi_1, \dots, \xi_n)
  &= \epsilon^{-2n} \tilde V_{t/\epsilon^2}(\xi_1/\epsilon, \dots, \xi_n/\epsilon),
  \\
  w_t(\xi_1,\dots,\xi_n) &= W_{t/\epsilon^2}(\xi_1/\epsilon,\dots,\xi_n/\epsilon),
\end{align}
where the functions on the right-hand sides are as in \cite[Section~3]{1907.12308},
with the only difference that we write $U_t$ instead of $u_t$
because we reserve lower case letters for continuum objects.
Moreover,
  \begin{gather}
  L_t = \epsilon \ell_{t/\epsilon^2} = \sqrt{t} \wedge 1/m,
  \qquad
  Z_t^\epsilon = \epsilon^{-2} z_{t/\epsilon^2}
  = z \epsilon^{-\beta/4\pi} e^{-\beta c_t^\epsilon(0)/2}
  \approx Z_t
  ,\\
  \lZ_t^\epsilon = L_t^2Z_t^\epsilon \asymp \ell_{t/\epsilon^2}^2z_{t/\epsilon^2}^\epsilon =\lz_{t/\epsilon^2}^\epsilon
  ,\qquad
  \theta_t = \vartheta_{t/\epsilon^2}^\epsilon = e^{-\frac12 m^2t}.
\end{gather}
Note that under the identification of functions on $\Lambda^n$
with functions on $\Omega^n$ that are constant on squares of side length $\epsilon$,
the norms $\|\cdot\|$ used in \cite[Section~3]{1907.12308}, i.e.,
\begin{align}
  \|F\| = \max_{\xi_1} \sum_{\xi_2,\dots,\xi_n} |F(\xi_1, \dots, \xi_n)|
  ,
\end{align}
for functions $F : (\Lambda \times \{\pm 1\})^n \to \R$ can be written as
\begin{align} \label{e:norm-rescaling}
  \epsilon^{2(n-1)}\|F\|
  &= \sup_{\xi_1} \int_{(\Omega_\epsilon \times \{\pm 1\})^{n-1}} d\xi_2 \dots d\xi_{n-1}
    |F(\xi_1, \dots, \xi_n)|
    \nnb
  &= \sup_{\xi_1} \norm{F(\xi_1,\cdot)}_{\Lpmeps{n-1}}.
\end{align}

\section{Coupling of the sine-Gordon field with the GFF}
\label{sec:coupling}

Using the estimates for the renormalised potential given in Section~\ref{sec:vt},
we now couple the sine-Gordon field with the GFF.
The main results of this section, Theorems~\ref{thm:coupling}--\ref{thm:coupling-bd} below,
immediately imply Theorem~\ref{thm:coupling-intro}
but are stronger in an important way exploited in Section~\ref{sec:maximum}.

\subsection{Decomposed Gaussian free field}
\label{sec:gffdecomp}

For the construction of the decomposed lattice GFF on $\Omega_\epsilon$, let
$(W^\epsilon(x))_{x\in\Omega_\epsilon}$
be independent Brownian motions $W^\epsilon(x)=(W^\epsilon_t(x))_{t\geq 0}$
with quadratic variations $t/\epsilon^d=t/\epsilon^2$.
In other words, $W^\epsilon$ is a standard Brownian motion with values in
$X_\epsilon$ equipped with the inner product \eqref{e:innerprod}.
In terms of
the (discrete) heat operator
$\dot c_{t}^\epsilon = e^{t\Delta^\epsilon - m^2t}$
from \eqref{e:dotcteps},
we then define the decomposed lattice GFF by
\begin{equation}
  \label{e:GFFepsdef}
  \Phi_t^{\GFF_\epsilon} = \int_t^\infty q^\epsilon_{u}\, dW_{u}^\epsilon,
  \qquad \text{where }
  q_t^\epsilon
  = \dot c_{t/2}^\epsilon.
\end{equation}
In particular, $\Phi_0^{\GFF_\epsilon}$ is then a realisation of
the massive Gaussian free field on $\Omega_\epsilon$,
i.e., a Gaussian field with covariance
$\int_0^\infty \dot c_t^{\epsilon}\, dt = (-\Delta^\epsilon+m^2)^{-1}$.

The decomposed continuum GFF is constructed analogously.
Let $W$ now be a cylindrical Brownian motion in $L^2(\Omega)$ defined on some probability space.
Thus, $W_t = \sum_{k \in 2\pi \Z^2} e^{ik\cdot (\cdot)} \hat W_t(k)$
where the $\hat W(k)$ are independent complex standard Brownian motions subject to $\hat W(k)=\overline{\hat W(-k)}$
for $k \neq 0$ and $\hat W(0)$ is a real standard Brownian motion,
and where the sum over $k \in 2\pi\Z^2$ converges in $C([0,\infty),H^{-d/2-\kappa}(\Omega))$.

The decomposed continuum GFF $(\Phi^{\GFF_0}_t)_{t\geq 0}$ is then defined
analogously as in \eqref{e:GFFepsdef}, by replacing
the lattice heat kernel $\dot c^\epsilon$ by
its continuum counterpart $\dot c^0$ defined in \eqref{e:dotct0}, 
i.e.,
\begin{equation} \label{e:GFF0def}
  \Phi_t^{\GFF_0} = \int_t^\infty q_u^0 dW_{u} \equiv
  \sum_{k\in2\pi \Z^2}  e^{ik\cdot(\cdot)}\int_t^\infty \hat q_u^0(k) d\hat W_u(k),
  \qquad \text{where }
  q_t^0
  = \dot c_{t/2}^0.
\end{equation}
The process $\Phi^{\GFF_0}$ takes values in $C([0,\infty),H^{-\kappa}(\Omega))$ for any $\kappa>0$
almost surely since $\int_0^\infty \sum_k (1+|k|^{2})^{-\kappa}|\hat q_u^0(k)|^2 du < \infty$,
and, for any $t_0>0$,
also $\Phi^{\GFF_0} \in C([t_0,\infty),C^\infty(\Omega))$ almost surely.

To take the limit $\epsilon \to 0$, it will be
convenient to couple  the Brownian motions $W^\epsilon(x)$ for all $\epsilon>0$ to
$W$ in the following standard way.
Recall that any $f\colon \Omega_\epsilon \to \R$ has the Fourier representation
$f(x)=\sum_{k\in\Omega_\epsilon^*} \hat f(k)e^{ik\cdot x}$ where
$\Omega_\epsilon^* = \{k \in 2\pi \Z^2: -\pi/\epsilon < k_i \leq \pi/\epsilon \}$.
For $x\in \Omega_\epsilon$, we set $W_t^\epsilon(x) = \Pi_\epsilon W_t(x)$
with $\Pi_\epsilon$ the restriction to the Fourier coefficients in $\Omega_\epsilon^*$,
i.e.,
\begin{equation}
  W^\epsilon_t(x) = \Pi_\epsilon W_t(x) = \sum_{k \in \Omega_\epsilon^*} e^{ik\cdot x} \hat W_t(k),
  \qquad x\in \Omega_\epsilon.
\end{equation}
Then $(W^\epsilon(x))_{x\in\Omega_\epsilon}$ are independent Brownian motions
indexed by $\Omega_\epsilon$ with quadratic variation $t/\epsilon^{2}$.
Indeed, clearly $W^\epsilon$ is a Gaussian process, and its covariance is
\begin{align}
  \E[W_t^\epsilon(x) W_s^\epsilon(y)]
  &=\sum_{k\in \Omega_\epsilon^*}\sum_{k'\in \Omega_\epsilon^*} \E[\hat W_t(k)\hat W_s(-k')]e^{i(k\cdot x-k'\cdot y)}
    \nnb
  &=
\sum_{k\in \Omega_\epsilon^*} (s\wedge t) e^{ik\cdot (x-y)}=\epsilon^{-2}(s\wedge t)\mathbf 1_{x=y}.
\end{align}
On the last line, we used that the Fourier series of $\mathbf 1_{x_0}\colon \Omega_\epsilon \to \R$ is given by
\begin{equation}
  \mathbf 1_{x_0}(x)= \sum_{k\in \Omega_\epsilon^*} \avg{\mathbf 1_{x_0}, e^{ik\cdot(\cdot)}}_{\Omega_\epsilon}e^{ik\cdot x}= \epsilon^2 \sum_{k \in \Omega_\epsilon^*} e^{ik(x-x_0)}
  .
\end{equation}
Note also that the lattice field $\Phi_t^{\GFF_\epsilon}$
admits a Fourier representation analogous to \eqref{e:GFF0def}
using the Brownian motions $W^\epsilon(x)$:
\begin{equation} \label{e:GFFepsFourier}
  \Phi_t^{\GFF_\epsilon} =
  \sum_{k\in \Omega_\epsilon^*}  e^{ik\cdot(\cdot)}\int_t^\infty \hat q_u^\epsilon(k) d\hat W_u(k),
  \qquad \text{where }
  q_t^\epsilon
  = \dot c_{t/2}^\epsilon.
\end{equation}

We will use the above coupling from now on.
The associated forward and backward filtrations $(\cF_t)$ and $(\cF^t)$
are defined by completing the $\sigma$-algebras generated by
the past $\{W_s-W_0: s \leq t\}$  respectively the future $\{W_s-W_t: s\geq t\}$.
We emphasise that the processes $\Phi^{\GFF} = (\Phi^{\GFF}_t)_{t>0}$
(with either $\epsilon >0$ or $\epsilon=0$) are adapted
to the \emph{backward} filtration $(\cF^t)$.

\subsection{Decomposed sine-Gordon field}

The following four theorems are the main results of Section~\ref{sec:coupling}.
These all refer to the decomposed GFF from \eqref{e:GFFepsdef}--\eqref{e:GFF0def}
and the renormalised potential $v_t^\epsilon$ from \eqref{eq:PolchinskiSDE-v}
when $\epsilon>0$ and Theorem~\ref{thm:v-limit} when $\epsilon=0$.

The first theorem shows the well-posedness of the (backward) SDEs that we
subsequently show to construct the sine-Gordon field.
In its statement and afterwards, $C_0$ denotes the space of continuous functions vanishing at infinity.

\begin{theorem} \label{thm:coupling}
  For $\epsilon>0$, there is a unique $\cF^t$-adapted process $\Phi^{\SG_\epsilon} \in C([0,\infty), X_\epsilon)$ such that
  \begin{equation}
    \label{e:Phi-SG-eps-coupling}
    \Phi_{t}^{\SG_\epsilon}
    = - \int_t^\infty \dot c^\epsilon_u \nabla v_{u}^\epsilon(\Phi_u^{\SG_\epsilon}) \, du
      + \Phi_t^{\GFF_\epsilon}.
    \end{equation}
    Analogously, there is a unique $\cF^t$-adapted process $\Phi^{\SG_0}$
    with $\Phi^{\SG_0}-\Phi^{\GFF_0}\in C_0([0,\infty), C(\Omega))$ such that
  \begin{equation}
    \label{e:Phi-SG-0-coupling}
    \Phi_{t}^{\SG_0} = - \int_t^\infty \dot c^0_u \nabla v_{u}^0(\Phi_u^{\SG_0}) \, du
                              + \Phi_t^{\GFF_0}.
  \end{equation}
  In particular, for any $t>0$, $\Phi^{\GFF}_0-\Phi_t^{\GFF}$ is independent of $\Phi_t^{\SG}$
  (in either case).
\end{theorem}

The next theorem shows that $\Phi^{\SG_{\epsilon}}_0$ is really the sine-Gordon field
\eqref{eq:DefinitionSGMeasure} when $\epsilon>0$.

\begin{theorem} \label{thm:coupling-sg}
  Let $\epsilon>0$. Then
  $\Phi_0^{\SG_\epsilon}$ is distributed as the sine-Gordon field on $\Omega_\epsilon$
  defined in \eqref{eq:DefinitionSGMeasure}.
\end{theorem}

For $\epsilon=0$, the sine-Gordon field has no direct definition and is instead defined as
the $\epsilon\to 0$ limit of the regularised fields.
Thus, by showing that $\Phi^{\SG_\epsilon} \to \Phi^{\SG_0}$, the next theorem justifies
that $\Phi^{\SG_0}$ is the continuum sine-Gordon field.

We denote the difference of the GFF and the sine-Gordon field under the
coupling of Theorem~\ref{thm:coupling} by
\begin{equation}
  \Phi_t^\Delta := \Phi_t^{\SG}-\Phi_t^{\GFF}
  = - \int_t^\infty \dot c_u \nabla v_{u}(\Phi_u^{\SG}) \, du.
\end{equation}

\begin{theorem} \label{thm:Phi-limit}
  Under the couplings of Theorems~\ref{thm:coupling} 
  (with the same Brownian motion for lattice and continuum versions
  as in Section~\ref{sec:gffdecomp}),
  for any $t_0>0$, 
  \begin{equation}
    \label{e:Phi-Delta-limit}
    \E \sup_{t\geq t_0}\norm{\Phi_t^{\Delta_\epsilon}- \Phi^{\Delta_0}_t}_{L^\infty(\Omega_\epsilon)} \to 0
    \quad \text{as $\epsilon \downarrow 0$.}
  \end{equation}
  In particular, for any $t\geq 0$,
  the lattice field $\Phi^{\SG_\epsilon}_t$
  converges weakly to 
  $\Phi^{\SG_0}_t$ in $H^{-\kappa}(\Omega)$ as $\epsilon \downarrow 0$,
  when $\kappa>0$,
  where we have identified $\Phi^{\SG_\epsilon}_t$ with the element of $C^\infty(\Omega)$
  with the same Fourier coefficients for $k\in \Omega_\epsilon^*$ and
  vanishing Fourier coefficients for $k\not\in\Omega_\epsilon^*$.
\end{theorem}

Finally,
with the estimates for $\dot c_t\nabla v_t$ from Theorem~\ref{thm:v-limit},
we obtain the following strong estimates for the 
difference of the GFF and the sine-Gordon field.

\begin{theorem} \label{thm:coupling-bd}
  The following estimates hold for $\epsilon>0$ or $\epsilon=0$ with all
  constants deterministic and independent of $\epsilon$.
  For any $t\geq 0$, the difference field $\Phi^\Delta$ satisfies the bound
  \begin{equation} \label{e:PhiDelta-bd0t}
    \max_x |\Phi^{\Delta}_0(x)-\Phi^{\Delta}_t(x)|
    \leq
    O_\beta(|\lZ_t|),
  \end{equation}
  as well as the following H\"older continuity estimates:
  \begin{align}
    \max_{x} |\Phi_t^{\Delta}(x)|+ \max_{x} |\partial \Phi_t^{\Delta}(x)| +        
    \max_{x,y}\frac{|\partial \Phi_t^\Delta(x)-\partial \Phi_t^\Delta(y)|}{|x-y|^{1-\beta/4\pi}}
    &\leq O_\beta(|z|),
    && (0<\beta<4\pi),
    \nnb
    \label{e:PhiDelta-bd}
    \max_{x} |\Phi_t^{\Delta}(x)|+ \max_{x,y}
    \frac{|\Phi_t^\Delta(x)-\Phi_t^\Delta(y)|}{|x-y| (1+|\log|x-y||)}
    &\leq O_\beta(|z|) ,
    && (\beta=4\pi),
    \\
    \max_{x} |\Phi_t^{\Delta}(x)|+ \max_{x,y}
    \frac{|\Phi_t^\Delta(x)-\Phi_t^\Delta(y)|}{|x-y|^{2-\beta/4\pi}}
    &\leq O_\beta(|z|) ,
    && (4\pi\leq \beta<6\pi).
       \nonumber
  \end{align}
  In addition, for any $t>0$
  and $k \in \N$, the following cruder bounds on all derivatives hold:
  \begin{equation} \label{e:delPhiDelta-bd}
    L_t^k\|\partial^k\Phi^\Delta_t\|_{L^\infty(\Omega)}
    \leq O_{\beta,k}(|z|),
  \end{equation}
  and in particular $\Phi^\Delta_t \in C^\infty(\Omega)$ for any $t>0$.
\end{theorem}

In the remainder of Section~\ref{sec:coupling}, we prove the above four theorems.
Theorem~\ref{thm:coupling-intro} is then an immediate consequence.

\begin{proof}[Proof of Theorem~\ref{thm:coupling-intro}]
  By Theorem~\ref{thm:Phi-limit}, the lattice approximation to the sine-Gordon field
  conver\-ges to the solution of \eqref{e:Phi-SG-0-coupling}. The stated estimates are
  immediate from Theorem~\ref{thm:coupling-bd}.
\end{proof}

\subsection{Well-posedness of SDEs: proof of Theorem~\ref{thm:coupling}}

We first show that the SDEs
\eqref{e:Phi-SG-eps-coupling} and \eqref{e:Phi-SG-0-coupling}
that construct the sine-Gordon field are well-posed.
This is essentially the standard argument for SDEs
and uses the Lipschitz continuity of $\dot c_s \nabla v_s$ given by Theorems~\ref{thm:v-bd} and \ref{thm:v-limit}.

\begin{proof}[Proof of Theorem~\ref{thm:coupling}]
The same argument applies to $\epsilon>0$ and to $\epsilon=0$,
but for concreteness, we will discuss the case $\epsilon=0$.
For $D,E \in C([t,\infty), C(\Omega))$, let
\begin{align} \label{e:F-Picard-def}
  F_t^0(D,E) = -\int_t^\infty \dot c_s^0\nabla v_s^0(E_s+D_s) \, ds.
\end{align}
By \eqref{e:cnablav-Lip}, the following estimate holds uniformly in $E \in C([t,\infty),C(\Omega))$:
\begin{align} \label{e:F-Lip}
  \norm{F_t^0(D,E)-F_t^0(\tilde D,E)}_{L^\infty(\Omega)}
  \leq \int_t^\infty O(\theta_s|Z_s|) \norm{D_s-\tilde D_s}_{L^\infty(\Omega)} \, ds.
\end{align}
Using that $0<\beta<8\pi$, it follows from \eqref{e:ZL-def} that
\begin{align} \label{e:intZt}
  \int_{0}^t \theta_s |Z_s| \, ds
  &\leq |z| \int_{0}^{t} L_s^{-\beta/4\pi} \, ds
  \leq O_\beta(|z|L_t^{2-\beta/4\pi})
  = O_\beta(|\lZ_t|), \qquad (t \leq 1/m^2),
  \\
  \label{e:intZ}
  \int_{0}^\infty \theta_s |Z_s| \, ds
  &\leq O_\beta(|\lZ_{1/m^2}|)
  +  |Z_{1/m^2} | \int_{1/m^2}^\infty \theta_s \, ds 
  \leq O_{\beta}(|z|m^{-2+\beta/4\pi}) = O_{\beta,m}(|z|),
\end{align}
where on the second line we used that $Z_t$ is decreasing in $t$ and that $\theta_t = e^{-\frac12 m^2t}$.

Suppose first that there are two solutions $\Phi^{\SG}$ and $\tilde\Phi^{\SG}$ to \eqref{e:Phi-SG-0-coupling}
satisfying $D:=\Phi^{\SG}-\Phi^{\GFF} \in C_0([t_0,\infty),L^\infty(\Omega))$
and $\tilde D:=\tilde\Phi^{\SG}-\Phi^{\GFF} \in C_0([t_0,\infty),L^\infty(\Omega))$.
Then by \eqref{e:F-Lip},
\begin{align}
  \norm{D_t-\tilde D_t}_{L^\infty(\Omega)}
  &= \norm{F_t^0(D,\Phi^{\GFF})-F_t^0(\tilde D,\Phi^{\GFF})}_{L^{\infty}(\Omega)}
    \nnb
  &\leq \int_t^\infty O_{\beta,m}(\theta_s|Z_s|) \norm{D_s-\tilde D_s}_{L^\infty(\Omega)} \, ds
    .
\end{align}
Thus $f(t) =   \norm{D_t-\tilde D_t}_{L^\infty(\Omega)}$ is bounded with $f(t)\to 0$ as $t\to\infty$ and
it satisfies
\begin{equation}
  f(t) \leq a + \int_t^\infty O_{\beta,m}(\theta_s|Z_s|) f(s) \, ds, \qquad \text{with $a=0$.}
\end{equation}
Since $\int_t^\infty O(\theta_s |Z_s|) \, ds < \infty$ by \eqref{e:intZ}
a version of Gronwall's inequality (which is derived in the same way as the standard version)
implies that 
\begin{equation}
  f(t) \leq a \exp\pa{\int_t^\infty O_\beta(\theta_s |Z_s|) \, ds} = 0.
\end{equation}
Hence $D=\tilde D$.

That a solution as above exists also follows by the standard argument for SDEs, i.e., by Picard iteration.
Let $\|D\|_t = \sup_{s\geq t} \norm{D_s}_{L^\infty(\Omega)}$.
For any $t>0$ and $E\in C([t,\infty), C(\Omega))$ fixed,
set $D^0=0$ and $D^{n+1}=F(D^n,E)$.
Using \eqref{e:intZ} then
\begin{align}
  \|D^1\|_t = \|F^0(0)\|_t &\leq \int_t^\infty ds\, O_{\beta,m}(\theta_s|Z_s|)= O_{\beta,m}(|z|),
  \\
  \|D^{n+1}-D^n\|_t
  &\leq \int_t^\infty ds \, O_{\beta,m}(\theta_{s} |Z_{s}|) \|D^{n}-D^{n-1}\|_{s},
\end{align}
and from the elementary identity 
\begin{equation}
  \int_t^\infty ds \, g(s) \pa{\int_s^\infty ds' \, g(s')}^{k-1} =  \frac{1}{k} \pa{\int_t^\infty ds\, g(s)}^{k}
\end{equation}
applied with $g(s) = O_{\beta,m}(\theta_s|Z_s|)$, we conclude that
\begin{equation}
  \|D^{n+1}-D^n\|_t
  \leq \frac{1}{n!} \pa{O_{\beta,m}(|z|)}^n 
    .
\end{equation}

Since the right-hand side is summable,
we have that $D^n \to D$ for some $D = D^*(E) \in C_0([t,\infty),C(\Omega))$ in $\|\cdot\|_t$,
and the limit satisfies $F_s^0(D^*(E),E)=D^*(E)_s$ for $s\geq t$.
By uniqueness, $D^*(E)_t$ is consistent in $t$ and we can thus define
$D^*(E) \in C_0((0,\infty), L^\infty(\Omega))$.
In summary, $\Phi^{\SG}_t = D^*(\Phi^\GFF)_t + \Phi^{\GFF}_t$ is the desired solution for $t>0$.
Further noticing that, for $t>0$,
\begin{equation}
  D^*(\Phi^\GFF)_t = -\int_t^\infty \dot c_s \nabla v_s(\Phi_s^\SG) \, ds,
\end{equation}
we may extend the solution to $t=0$ by continuity. Indeed, by \eqref{e:intZt},
as $t\to 0$,
\begin{equation}
  \norma{ \int_0^t \dot c_s \nabla v_s(\Phi_s^\SG) \, ds}_{L^\infty(\Omega)}
  \leq
  \int_0^t O_\beta(\theta_s|Z_s|) \, ds \leq O_\beta(|\lZ_t|) \to 0,
\end{equation}
and thus the limit $\lim_{t\downarrow 0} D^*(\Phi^\GFF)_t$ exists in $C(\Omega)$.
\end{proof}

\subsection{Coupling of fields: proof of Theorem~\ref{thm:coupling-sg}}
\label{ss:CouplingFields}

We next show that \eqref{e:Phi-SG-eps-coupling} defines a
realisation of the sine-Gordon field on $\Omega_\epsilon$, 
thus proving Theorem~\ref{thm:coupling-sg}.

\begin{proof}[Proof of Theorem~\ref{thm:coupling-sg}]
Throughout this proof, $\epsilon>0$ is fixed and dropped from the notation.
We will show that the sine-Gordon field  \eqref{eq:DefinitionSGMeasure}
can be constructed from the Polchinski semigroup as
follows.
Let $\nu_t$ be the renormalised measure defined by
$\E_{\nu_t}f = e^{v_t(0)} \EE_{c_t}(e^{-v_0(\zeta)}f(\zeta))$
where $\EE_{c_t}$ is the expectation of the Gaussian measure with covariance $c_t$
and $v_t$ is as in Section~\ref{sec:vteps}.
In \cite[Section~2]{1907.12308}, it is shown that
for every bounded and smooth function $f: X_\epsilon \to \R$,
\begin{equation} \label{e:nu0nutf}
  \E_{\nu_0}f
  = \E_{\nu_t} \PP_{0,t} f,
\end{equation}
where
\begin{equation}
  \PP_{s,t} f(\phi) = e^{v_t(\phi)} \EE_{c_t-c_s}(e^{-v_s(\phi+\zeta)} f(\phi+\zeta)).
\end{equation}
By \cite[Proposition~2.1]{1907.12308}, this semigroup is characterised by its infinitesimal generator
\begin{equation}
  \LL_t = \frac12 \Delta_{\dot c_t} - (\nabla v_t,\nabla)_{\dot c_t},
\end{equation}
in the sense that
\begin{equation}
  \ddp{}{s} \PP_{s,t} f = -\PP_{s,t} \LL_s f,
  \qquad
  \ddp{}{t} \PP_{s,t} f = +\LL_t \PP_{s,t} f.
\end{equation}
Therefore, as in \cite[Remark~2.2]{1907.12308}, the semigroup $\PP_{s,t}$ has the following stochastic representation.
(As in Section~\ref{sec:vt} note that we have here rescaled $t$ by $1/\epsilon^2$
compared to \cite{1907.12308} both in the definition of the renormalised measure
and the Polchinski semigroup; see Section~\ref{sec:rescaling} for a translation of
these normalisations.)
Let $(W_t)$ be an $X_\epsilon$-valued Brownian motion normalised so that
$(W_t(x))$ is a Brownian motion with quadratic variation $t/\epsilon^2$
for each $x\in \Omega_\epsilon$, with completed forward and backward filtrations $(\cF_t)$
and $(\cF^t)$.
For any $T>0$, the time-reversed process $\tilde W^{T}_t = W_T-W_{T-t}$ is then again a Brownian motion.
Let $\tilde\Phi^T$ be the unique strong solution to the following (forward) SDE
with Lipschitz coefficients:
\begin{equation}
  d\tilde\Phi_t^T = -\dot c_{T-t} \nabla v_{T-t}(\tilde\Phi_t^T) \, dt + q_{T-t} \, d\tilde W_t^T,
  \qquad (0 \leq t \leq T).
\end{equation}
By It\^o's formula, we see that
\begin{equation}
  \ddp{}{t} \E_{\tilde\Phi_0^T=\phi} f(\tilde\Phi_t^T)
  =
  \E_{\tilde\Phi_0^T=\phi} (L_{T-t}f(\tilde\Phi_t^T)).
\end{equation}
Therefore
$t\mapsto \E_{\tilde\Phi_0^T=\phi} f(\tilde\Phi_t^T)$
defines an (inhomogeneous) Markov semigroup with generator $L_{T-t}$ and uniqueness
of such semigroups implies $\PP_{T-t,T} f(\phi) = \E_{\tilde \Phi_0^T=\phi} f(\tilde \Phi_t^T)$.
Indeed, denoting the last right-hand side
by $\tilde\PP_{T-t,T}f(\phi)$, for any bounded smooth function $f$, one has
\begin{align}
  \tilde \PP_{T-t,T}f - \PP_{T-t,T}f
  &= \int_0^t \ddp{}{s}(\tilde\PP_{T-s,T} \PP_{T-t,T-s}f) \, ds
    \nnb
  &= \int_0^t (\tilde\PP_{T-s,T} (L_{T-s}-L_{T-s})\PP_{T-t,T-s}f) \, ds = 0.
\end{align}
In particular,
\begin{equation}
  \PP_{0,T} f(\phi)
  = \E_{\tilde\Phi_0^T=\phi} f(\tilde\Phi_T^T)
  .
\end{equation}
We now reverse the time direction.
We thus set $\Phi_t^T = \tilde\Phi_{T-t}^T$ so that,
with the change of variable $s\mapsto T-s$ on the second line of the following display,
\begin{align}  \label{e:varphiT-SDE}
  \Phi_t^T
  &= \Phi_T^T - \int_0^{T-t} \dot c_{T-s} \nabla v_{T-s}(\tilde\Phi_s^T) \, ds + \int_0^{T-t} q_{T-s} \, d\tilde W_s^T
    \nnb
  &= \Phi_T^T - \int_t^T \dot c_{s} \nabla v_{s}(\Phi_s^T) \, ds + \int_t^T q_{s} \, dW_s
    .
\end{align}
Then $(\Phi^T_t)$ is adapted to the descending filtration $\cF^t$, and
\begin{equation}
  \PP_{0,T}f (\phi)
  = \E_{\Phi_T^T=\phi} f(\Phi_0^T).
\end{equation}
In particular, by \eqref{e:nu0nutf}, if $\Phi_T^T$ is distributed according to $\nu_T$
then $\Phi_0^T$ is distributed according to the sine-Gordon measure $\nu_0 = \nu^{\SG_\epsilon}$.

It therefore suffices to show that, as $T\to\infty$, the solution \eqref{e:varphiT-SDE}
with $\Phi_T^T$ distributed according to $\nu_T$ converges to
the solution $(\Phi_t)_{t\geq 0}$ to \eqref{e:Phi-SG-eps-coupling}
constructed using the same Brownian motion $W$.
This will essentially follow from the fact that $\nu_T \to \delta_0$.
Let $\Phi_\infty=0$. We start from
\begin{equation}
  \Phi_t - \Phi_t^T
  = (\Phi_\infty-\Phi_T^T)
  - \int_t^T \qa{\dot c_{s} \nabla v_{s}(\Phi_s)-\dot c_{s} \nabla v_{s}(\Phi_s^T)} \, ds  + \int_T^\infty \dot c_{s} \nabla v_{s}(\Phi_s) \, ds + \int_T^\infty q_{s} \, dW_s.
\end{equation}
Since $\epsilon>0$ is fixed, we may use any norm on $X_\epsilon$ and denote it by $\norm{\cdot}$.
The first, third, and fourth terms on the right-hand side above are independent of $t$ and
we claim that they converge to $0$ in probability as $T \to\infty$.
For the first term this follows from the weak convergence of the
measure $\nu_T$ to $\delta_0$, e.g.,
in the sense of \cite[(1.6)]{1907.12308}.
By \eqref{e:nablaveps-bd}, the third term is bounded by
\begin{equation}
  \int_T^\infty \norm{\dot c_s \nabla v_s(\Phi_s)} \, ds
  \leq \int_T^\infty O_{\beta,m}(\theta_sZ_s) \, ds \to 0.
\end{equation}
The fourth term is a Gaussian field on $\Omega_\epsilon$ with covariance matrix
$c_\infty - c_T \to 0$
as $T\to\infty$. Since $\Omega_\epsilon$ is finite,
it is a trivial consequence that 
this Gaussian field convergences to $0$.

In summary, we have shown that there is $R_T$ such that $\norm{R_T} \to 0$ in probability, and
\begin{equation}
  \Phi_t - \Phi_t^T
  = - \int_t^T \qa{\dot c_{s} \nabla v_{s}(\Phi_s)-\dot c_{s} \nabla v_{s}(\Phi_s^T)} \, ds  + R_T.
\end{equation}
By the Lipschitz continuity of $\dot c_s \nabla v_s$,
see \eqref{e:nablaveps-bd},
with $M_t=O_\beta(\theta_t |Z_t|) + O_{\beta,m,z}(\theta_t {\mathbf 1}_{t\geq t_0})$,
\begin{equation}
  \int_t^T \norm{\dot c_{s} \nabla v_{s}(\Phi_s)-\dot c_{s} \nabla v_{s}(\Phi_s^T)} \, ds
  \leq \int_t^T  M_s \norm{\Phi_s-\Phi_s^T}  \, ds
  .
\end{equation}
Thus we have shown that $D_t = \Phi_{t}-\Phi^T_{t}$ satisfies
\begin{align}
  \norm{D_t}
  \leq \norm{R_T} +  \int_{t}^T  M_s \norm{D_s}  \, ds.
\end{align}
Since $\int_0^\infty M_s \, ds < \infty$,
the same version of Gronwall's inequality as in the proof of Theorem~\ref{thm:coupling}
implies that
\begin{equation}
  \norm{D_t}
  \leq \norm{R_T}\exp\pa{\int_t^T M_s \, ds}
  \leq \norm{R_T}\exp\pa{\int_0^\infty M_s \, ds}
  \lesssim \norm{R_T}.
\end{equation}
This bound is uniform in $t$ and
hence $\sup_{t \in [0,T]} \norm{D_t} \to 0$ in probability as $T\to\infty$.
\end{proof}

\subsection{Lattice convergence: proof of Theorem~\ref{thm:Phi-limit}}

We now prove that $\Phi^{\SG_\epsilon} \to \Phi^{\SG_0}$ as $\epsilon \to 0$.
To this end, we will need the following crude estimates for the convergence of the decomposed GFF.
Let $\extepsI: L^2(\Omega_\epsilon)\hookrightarrow L^2(\Omega)$
be the standard isometric embedding, i.e.,
given $f: \Omega_\epsilon \to \R$, the function $\extepsI f \in C^\infty(\Omega)$
is defined to have
the same Fourier coefficients as $f$ for $k \in \Omega_\epsilon^*$
and to have vanishing Fourier coefficients for $k\in\Omega^*\setminus \Omega_\epsilon^*$.

\begin{lemma} \label{lem:PhiGFF-limit}
  For any $t_0>0$, and $1\leq p<\infty$,
  \begin{align}
    \label{e:PhitGFFbd}
    \E\qa{ \sup_{t\geq t_0}\norm{\partial \Phi_t^{\GFF}}_{L^\infty(\Omega)}^p} &\leq O_{p,t_0}(1)
\\
    \E\qa{ \sup_{t\geq t_0}\norm{\Phi^{\GFF_\epsilon}_t -\Phi^{\GFF_{0}}_t}_{L^\infty(\Omega_\epsilon)}^p} &\to 0 \quad \text{as $\epsilon \downarrow 0$.}
    \label{eq:sup-lattice-restriction}
  \end{align}
  Moreover, with $\extepsI$ the isometric embedding $L^2(\Omega_\epsilon)\hookrightarrow L^2(\Omega)$ as above, for any $t\geq 0$ and $\kappa>0$,
  \begin{equation} \label{e:PhiGFFlimit}
    \E\qa{ \|\extepsI \Phi_t^{\GFF_\epsilon}-\Phi_t^{\GFF_{0}}\|_{H^{-\kappa}(\Omega)}^2 } \to 0.
  \end{equation}
\end{lemma}

\begin{proof}
Similarly to the proof of Lemma \ref{lem:Gauss-conv},
we will use the Fourier coefficients of $\Phi_t^{\GFF_\epsilon}$ and $\Phi_t^{\GFF_0}$ to prove the statements. Let $q_t^\epsilon$ and $q_t^0$ be as in \eqref{e:GFFepsdef} and \eqref{e:GFF0def}. Then as in \eqref{e:fourier-coef-bound-1}--\eqref{e:fourier-coef-bound-1eps},
\begin{equation}\label{e:fourier-coef-bound-bis}
  0\leq \hat q_t^0(k) = e^{-(t/2) (|k|^2 + m^2)} ,\quad 
                      0\leq \hat q_t^\epsilon(k) \leq e^{-(t/2) (\frac14 |k|^2 + m^2)},
\end{equation}
and, similarly, with $h$ as below \eqref{e:fourier-multipliers-difference},
\begin{align}
0 \leq \hat q_t^\epsilon - \hat q_t^0
&= e^{(t/2) (\hat \Delta^\epsilon (k) - m^2)}(1-e^{(t/2)(-\hat \Delta^0(k) + \hat \Delta^\epsilon(k))} )
\nnb
&\leq e^{-(t/2)(\frac14 |k|^2 + m^2))}(t/2) |k|^2 h(\epsilon k).
\label{e:GFFfourier-difference}
\end{align}
By the Sobolev inequality
$\|f\|_{L^\infty(\Omega_\epsilon)} \leq C_\alpha \|f\|_{H^\alpha(\Omega_\epsilon)}$,
$\alpha>d/2$,
which holds for $\epsilon= 0$ and $\epsilon>0$,
it suffices to prove     \eqref{e:PhitGFFbd}--\eqref{eq:sup-lattice-restriction}
with the norms replaced by $H^\alpha$ norms for $\alpha$ large enough.

  For any $\alpha\in \R$ and $t_0>0$,
  the field $(\Phi^{\GFF}_t)_{t\geq t_0}$ is an $H^\alpha(\Omega)$-valued
  backward martingale with respect to the backward filtration $\cF^t=\sigma(W_u \colon u\geq t)$
  and quadratic variation
  \begin{equation}
\int_t^\infty \norm{q_u}_{H^\alpha(\Omega)}^2 du 
= \sum_{k \in \Omega^*} (1+|k|^2)^\alpha \int_t^\infty |\hat q_u(k)|^2 du
\leq O_{t_0,\alpha}(1),
  \end{equation}
where we used 
\eqref{e:fourier-coef-bound-1}--\eqref{e:fourier-coef-bound-1eps}
to obtain the uniform bound on the right-hand side.

  Thus Burkholders's inequality (for Hilbert space valued martingales, see, e.g.,
  \cite{MR3463679}) implies
  \begin{equation}
    \E\qbb{\sup_{t\geq t_0}\norm{\Phi_t^{\GFF}}_{H^\alpha(\Omega)}^p}
    \leq O_{t_0,\alpha,m,p}(1).
  \end{equation}
  To prove \eqref{eq:sup-lattice-restriction} we use a similar argument together with the Fourier representations \eqref{e:GFF0def} and \eqref{e:GFFepsFourier}.
  Observe that
  \begin{equation}
  \norm{\Phi_t^{\GFF_\epsilon} - \Phi_t^{\GFF_0} }_{L^\infty(\Omega_\epsilon)}^p
  \lesssim 
   \norm{ \Phi_t^{\GFF_\epsilon} - \Pi_\epsilon \Phi_t^{\GFF_0} }_{L^\infty(\Omega_\epsilon)}^p
    +
      \norm{ \Pi_\epsilon \Phi_t^{\GFF_0}  - \Phi_t^{\GFF_0} }_{L^\infty(\Omega)}^p.
  \end{equation}
  Similarly as above the processes $(\Phi_t^{\GFF_\epsilon} -  \Pi_\epsilon\Phi_t^{\GFF_0})_{t\geq t_0}$ and $(\Pi_\epsilon \Phi_t^{\GFF_0} -  \Phi_t^{\GFF_0})_{t\geq t_0}$ are backward martingales taking values in $H^\alpha(\Omega_\epsilon)$ and $H^\alpha(\Omega)$ with quadratic variations
  \begin{align}
  [\Phi_t^{\GFF_\epsilon} -  \Pi_\epsilon\Phi_t^{\GFF_0}]_t
  &=\sum_{k\in \Omega_\epsilon^*} (1+|k|^2)^\alpha \int_t^\infty | \hat q_u^\epsilon (k) - \hat q_u^0(k)|^2 \, du 
  \\
  [\Pi_\epsilon \Phi_t^{\GFF_0} -  \Phi_t^{\GFF_0}]_t &= \sum_{k\in \Omega^* \setminus \Omega_\epsilon^*} (1+ |k|^2)^\alpha \int_t^\infty |\hat q_u^0(k)|^2 \, du.
  \end{align} 
Since both converge to $0$ as $\epsilon \to 0$, the claim follows again by Burkholder's inequality for Hilbert space valued martingales.

Finally, we prove \eqref{e:PhiGFFlimit}.
By the definition of the Sobolev norm, we have
\begin{multline}
\|\extepsI \Phi_t^{\GFF_\epsilon} -\Phi_t^{\GFF_{0}}\|_{H^{-\kappa}(\Omega)}^2 
= \sum_{k\in \Omega_\epsilon^*} (1+ |k|^2)^{-\kappa} \big|\int_t^\infty (\hat q_u^\epsilon(k) - \hat q_u^0(k)) d\hat W_u(k)   \big|^2 
\\
+ 
\sum_{k \in \Omega^*\setminus \Omega_\epsilon^*}  (1+ |k|^2)^{-\kappa} \big|\int_t^\infty \hat q_u^0(k) d\hat W_u(k)   \big|^2.
\end{multline}
Taking expectation the corresponding second sum on the right-hand side is finite for $\kappa>0$ and hence converges to $0$ as $\epsilon \to 0$.

For the first sum on the right-hand side, we may of course assume that $\kappa \leq 4$.
Then $h(\epsilon k)^2 \leq O(\epsilon |k|)^\kappa$  and the expectation of the first sum is bounded by
\begin{align}
\sum_{k\in \Omega_\epsilon^*} (1+ |k|^2)^{-\kappa} \int_0^\infty |\hat q_u^\epsilon(k) - \hat q_u^0(k)|^2 du   
&\leq  \frac{1}{4}  \sum_{k\in \Omega_\epsilon^*} (1+ |k|^2)^{-\kappa} |k|^4 h^2(\epsilon k) \int_0^\infty  e^{-t(\frac14 |k|^2 + m^2)}t^2 \, dt
\nnb
  &\lesssim \epsilon^{\kappa} \sum_{k\in \Omega_\epsilon^*} (1+ |k|^2)^{-\kappa} |k|^{4+\kappa} \frac{1}{(|k|^2 + m^2)^3} 
    \lesssim \epsilon^\kappa
\end{align}
as needed.
\end{proof}

\begin{proof}[Proof of Theorem~\ref{thm:Phi-limit}]
We will first show \eqref{e:Phi-Delta-limit}, i.e., that for any fixed $t_0>0$,
\begin{equation} \label{e:Phi-Delta-limit-bis}
  \sup_{t\geq t_0}\norm{\Phi_t^{\Delta_\epsilon} - \Phi_t^{\Delta}}_{L^\infty(\Omega_\epsilon)}
  \to 0\qquad
  \text{in $L^1(\P)$ as $\epsilon \downarrow 0$.}
\end{equation}
To this end, define $F^0$ as in \eqref{e:F-Picard-def}
and $F^\epsilon$ analogously with $\dot c_t^0\nabla v_t^0$ replaced by $\dot c_t^\epsilon \nabla v_t^\epsilon$.
Let $D = \Phi^\Delta$ be the corresponding fixed points with $E = \Phi^{\GFF}$, i.e.,
\begin{equation}
  \Phi^{\Delta_\epsilon} = F^\epsilon(\Phi^{\Delta_\epsilon},\Phi^{\GFF_\epsilon}),
  \qquad \Phi^{\Delta_0} = F^0(\Phi^{\Delta_0},\Phi^{\GFF_0}).
\end{equation}
Then (again identifying $\Phi_t^{\Delta_0}$ with its restriction to $\Omega_\epsilon$)
\begin{align} \label{e:Phi-limit-twoterms}
  \Phi^{\Delta_0}-  \Phi^{\Delta_\epsilon}
  &=F^0(\Phi^{\Delta_0}, \Phi^{\GFF_0})-F^\epsilon(\Phi^{\Delta_\epsilon}, \Phi^{\GFF_\epsilon})
    \nnb
  &= [F^\epsilon(\Phi^{\Delta_0}, \Phi^{\GFF_0})-F^\epsilon(\Phi^{\Delta_\epsilon}, \Phi^{\GFF_\epsilon})]
     +[F^0(\Phi^{\Delta_0}, \Phi^{\GFF_0})-F^{\epsilon}(\Phi^{\Delta_0},\Phi^{\GFF_0})].
\end{align}
The first term in \eqref{e:Phi-limit-twoterms} is bounded as in \eqref{e:F-Lip} by
\begin{multline}
  \norm{F_t^\epsilon(\Phi^{\Delta_0},\Phi^{\GFF_0})-F_t^\epsilon(\Phi^{\Delta_\epsilon}, \Phi^{\GFF_\epsilon})}_{L^\infty(\Omega_\epsilon)}
  \\
  \leq \int_t^\infty O(\theta_s|Z_s|) \norm{\Phi^{\Delta_0}_s-\Phi^{\Delta_\epsilon}_s}_{L^\infty(\Omega_\epsilon)} \, ds
  +
  \int_t^\infty O(\theta_s|Z_s|) \norm{\Phi_s^{\GFF_0}-\Phi_s^{\GFF_\epsilon}}_{L^\infty(\Omega_\epsilon)} \, ds
  .
\end{multline}
By Lemma~\ref{lem:PhiGFF-limit}, the second term on the right-hand side converges to $0$
  in $L^1(\P)$.
To bound the second term in \eqref{e:Phi-limit-twoterms},
write
  \begin{multline}
    \norm{F_t^0(\Phi^{\Delta_0}, \Phi^{\GFF_0})-F_t^{\epsilon}(\Phi^{\Delta_0},\Phi^{\GFF_0})}_{L^\infty(\Omega_\epsilon)}
    \\
    \leq \int_t^\infty \norm{
      \dot c_s^0\nabla v_s^0(\Phi_s^{\Delta_0}+\Phi_s^{\GFF_0})
      -
      \dot c_s^\epsilon\nabla v_s^\epsilon(\Phi_s^{\Delta_0}+\Phi_s^{\GFF_0})}_{L^\infty(\Omega_\epsilon)} \, ds.
  \end{multline}
  By \eqref{e:PhitGFFbd}, $\Phi^{\GFF_0}_t$ is smooth for all $t \geq t_0>0$,
  almost surely,
  and by Theorem~\ref{thm:coupling-bd}
  (which is presented after the current theorem but whose proof is independent of it),
  $\Phi^{\Delta}_t$ is also smooth for all $t>0$.
  Therefore, by \eqref{e:cnablav-eps}--\eqref{e:cnablav-bd} and dominated convergence,
  \begin{equation}
    \E \sup_{t \geq t_0} \norm{F_t^0(\Phi^{\Delta_0}, \Phi^{\GFF_0})-F_t^{\epsilon}(\Phi^{\Delta_0},\Phi^{\GFF_0})}_{L^\infty(\Omega_\epsilon)} \to 0 \qquad (\epsilon \to 0).
  \end{equation}
  In summary, for every $t>0$, there is a random variable $R_t^\epsilon$ converging to $0$ as $\epsilon \to 0$
  in $L^1(\P)$ such that
\begin{equation}
  \norm{\Phi^{\Delta_0}_t-\Phi^{\Delta_\epsilon}_t}_{L^\infty(\Omega_\epsilon)}
  \leq R_t^\epsilon
  + \int_t^\infty O(\theta_sZ_s) \norm{\Phi^{\Delta_0}_s-\Phi^{\Delta_\epsilon}_s}_{L^\infty(\Omega_\epsilon)}  \, ds.
\end{equation}
The same version of Gronwall's inequality as in the proof of 
Theorem~\ref{thm:coupling} thus implies \eqref{e:Phi-Delta-limit-bis}.

We now conclude the proof of the convergence $\extepsI \Phi^{\SG_\epsilon} \to \Phi^{\SG_0}$ in $H^{-\kappa}(\Omega)$.
Since $\extepsI \Phi^{\GFF_\epsilon} \to \Phi^{\GFF_0}$ in $H^{-\kappa}(\Omega)$ by Lemma~\ref{lem:PhiGFF-limit},
it is more than sufficient to show that  the following
right-hand side converges to $0$ in $L^\infty(\Omega)$ as $\epsilon \to 0$:
\begin{equation}
  (\Phi^{\SG_0}_0-\extepsI \Phi^{\SG_\epsilon}_0)
  -
  (\Phi^{\GFF_0}_0-\extepsI \Phi^{\GFF_\epsilon}_0)
  =
  (\Phi^{\Delta_0}_0-\Phi^{\Delta_0}_t)
  + \extepsI(\Phi^{\Delta_\epsilon}_t-\Phi^{\Delta_\epsilon}_0)
  +
    (\Phi^{\Delta_0}_t-\extepsI \Phi^{\Delta_\epsilon}_t).
\end{equation}
By \eqref{e:PhiDelta-bd0t},
the first and second terms on the right-hand side are bounded as
$\norm{\Phi_t^{\Delta}-\Phi_0^{\Delta}}_{L^\infty(\Omega)} = O(\lZ_t)$ which converges to $0$ as $t\to 0$.
For the third term,
since $\Phi^{\Delta_0}_t$ satisfies the H\"older continuity estimate \eqref{e:PhiDelta-bd0t},
for $\kappa<2-\beta/4\pi$ we have
\begin{equation}
  \norm{\Phi^{\Delta_0}_t-\extepsI \Phi^{\Delta_\epsilon}_t}_{L^\infty(\Omega)}
  \leq
  \norm{\Phi^{\Delta_0}_t - \Phi^{\Delta_\epsilon}_t}_{L^\infty(\Omega_\epsilon)} + O(\epsilon^{\kappa})
  .
\end{equation}
This converges to $0$ in $L^1(\P)$ as $\epsilon \to 0$ for any $t>0$
by \eqref{e:Phi-Delta-limit}.
Thus taking first $\epsilon \to 0$ and then $t\to 0$ gives the required convergence.
\end{proof}

\subsection{Estimates: proof of Theorem~\ref{thm:coupling-bd}}

To prove Theorem~\ref{thm:coupling-bd}, we start from
\begin{equation} \label{e:PhiDeltaint}
  \Phi^\Delta_t = \Phi_t^{\SG}-\Phi_t^{\GFF} = -\int_t^\infty \dot c_s\nabla v_s(\Phi_s^{\SG}) \, ds
  .
\end{equation}
As previously, we often omit the subscript $\epsilon$ which can be positive or zero.
By \eqref{e:cnablav-bd}, the following estimates for $\dot c_t\nabla v_t(\varphi)$
hold with constants uniform in $\varphi \in C(\Omega)$ and $x \in \Omega$:
\begin{align}
  \label{e:cnablav-bdx}
  L_t^{2} \abs{\partial^k\dot c_t\nabla v_t(\varphi,x)}
  &= O_{\beta,k}(\theta_t\frac{|\lZ_t|}{L_t^k}),
  \\
  \label{e:cnablav-contx}
  L_t^2 \abs{\partial^k\dot c_t\nabla v_t(\varphi,x) - \partial^k \dot c_t\nabla v_t(\varphi,x')}
  &= O_{\beta,k}(\theta_t\frac{|\lZ_t|}{L_t^k}) \pa{1\wedge \frac{|x-x'|}{L_t}}.
\end{align}
Theorem~\ref{thm:coupling-bd} now follows by integrating these bounds as follows.


\begin{proof}[Proof of \eqref{e:PhiDelta-bd0t}]
By \eqref{e:PhiDeltaint} and \eqref{e:cnablav-bdx}, for $0\leq \beta<8\pi$,
\begin{equation}
  |\Phi^\Delta_0(x)-\Phi^\Delta_t(x)| 
  \lesssim \int_0^t \theta_s |\lZ_s| \, \frac{ds}{L_s^2}
  \lesssim |z| \int_0^t s^{1-\beta/8\pi} \, \theta_s\, \frac{ds}{s}
    \lesssim |\lZ_t|.\qedhere
\end{equation}
\end{proof}

\begin{proof}[Proof of \eqref{e:PhiDelta-bd} for $\beta<4\pi$]
  By \eqref{e:PhiDeltaint} and   \eqref{e:cnablav-bdx}, for $\beta<4\pi$,
  \begin{equation}
    |\partial\Phi^\Delta_0(x)|\lesssim 
    \int_{0}^\infty \frac{|\lZ_t|}{L_t} \, \theta_t\, \frac{dt}{L_t^2}
  \lesssim 
  |z| \, \int_{0}^\infty t^{1/2-\beta/8\pi} \, \theta_t\, \frac{dt}{t}
  \lesssim |z|\,.
\end{equation}
Moreover, by \eqref{e:PhiDeltaint} and \eqref{e:cnablav-contx},
\begin{equation}
  |\partial\Phi^\Delta_0(x)-\partial\Phi^\Delta_0(x')|
  \lesssim \int_0^\infty \theta_t\frac{|\lZ_t|}{L_t} \pa{1\wedge \frac{|x-x'|}{L_t}} \frac{dt}{L_t^2}
    .
\end{equation}
We split the integral into the two contribution from $t \leq |x-x'|^2$ and $t\geq |x-x'|^2$.
For $\beta<4\pi$, the contribution from $t\leq |x-x'|^2$ is bounded by
\begin{equation}
  \int_0^{|x-x'|^2} \frac{|\lZ_t|}{L_t} \, \frac{dt}{L_t^2}
  \lesssim
  |z| \int_0^{|x-x'|^2} t^{1/2-\beta/8\pi} \, \frac{dt}{t}
  \lesssim  |z|\, |x-x'|^{1-\beta/4\pi},
\end{equation}
and the contribution from $t\geq |x-x'|^2$ likewise gives
\begin{equation}
  \int_{|x-x'|^2}^\infty \frac{|\lZ_t|}{L_t} \frac{|x|}{L_t} \, \theta_t\, \frac{dt}{t}
  \lesssim
  |z| \, |x-x'| \int_{|x|^2}^\infty t^{-\beta/8\pi} \, \theta_t\, \frac{dt}{t}
  \lesssim |z|\, |x-x'|^{1-\beta/4\pi}
  .\qedhere
\end{equation}
\end{proof}

\begin{proof}[Proof of \eqref{e:PhiDelta-bd} for $4\pi \leq \beta < 6\pi$]
By \eqref{e:PhiDeltaint} and \eqref{e:cnablav-contx},
\begin{equation}
  |\Phi^\Delta_0(x)-\Phi^\Delta_0(x')|
  \lesssim \int_0^\infty \theta_t|\lZ_t| \pa{1\wedge \frac{|x-x'|}{L_t}} \frac{dt}{L_t^2}
    .
\end{equation}
The integral is bounded in the same way as in the previous proof
by splitting the integral into the two contribution
from $t \leq |x-x'|^2$ and $t\geq |x-x'|^2$.
The contribution from $t\leq |x-x'|^2$ is bounded by
\begin{equation}
  \int_0^{|x-x'|^2} |\lZ_t| \, \frac{dt}{L_t^2}
  \lesssim
  |z| \int_0^{|x-x'|^2} t^{1-\beta/8\pi} \, \frac{dt}{t}
  \lesssim  |z|\, |x-x'|^{2-\beta/4\pi}.
\end{equation}
For the contribution due to $t > |x-x'|^2$, for $\beta>4\pi$, we also have
\begin{equation}
  \int_{|x-x'|^2}^\infty |\lZ_t| \frac{|x-x'|}{L_t} \, \theta_t\, \frac{dt}{t}
  \lesssim
  |z| \, |x-x'| \int_{|x|^2}^\infty t^{1/2-\beta/8\pi} \, \theta_t\, \frac{dt}{t}
  \lesssim |z|\, |x-x'|^{2-\beta/4\pi}
  .
\end{equation}
For $\beta = 4\pi$, similarly,
\begin{equation}
  \int_{|x-x'|^2}^\infty |Z_t| \frac{|x-x'|}{L_t} \, \theta_t\, \frac{dt}{t}
  =
  |z| \, |x-x'| \int_{|x-x'|^2}^\infty \theta_t\, \frac{dt}{t}
  \lesssim |z|\, |x-x'| \, |\log |x-x'||
  .
\end{equation}
The claim follows by taking the sum over the $t\leq |x-x'|^2$ and $t>|x-x'|^2$ contributions.
\end{proof}

\begin{proof}[Proof of \eqref{e:delPhiDelta-bd}]
By \eqref{e:PhiDeltaint} and \eqref{e:cnablav-bdx}, for $0\leq \beta<8\pi$,
\begin{equation}
  \|\partial^k\Phi^\Delta_t\|_{L^\infty(\Omega)}
  \lesssim
  \int_t^\infty   \theta_s\frac{|\lZ_s|}{L_s^k}\, \frac{ds}{L_s^2}
  \leq
  \frac{1}{L_t^k}\int_t^\infty   \theta_s|\lZ_s|\, \frac{ds}{L_s^2}
  \lesssim \frac{|z|}{L_t^k}.\qedhere
\end{equation}
\end{proof}

\newcommand{\tPhisSG}{\tilde \Phi_s^\SG}
\newcommand{\PhisKSG}{\Phi_s}
\newcommand{\PhisKGFF}{\eta}  
\newcommand{\maxPhisKGFF}{\eta^*}
\newcommand{\maxPhisKSG}{\Phi_s^*}
\newcommand{\restmaxPhisKSG}{\Psi^*} 
\newcommand{\coarsesKGauss}{\tilde X_{s,K}^c}
\newcommand{\coarseSG}{Z_{s,K}^c}  
\newcommand{\coarseGauss}{X_{s,K}^c}  
\newcommand{\limcoarseSG}{Z_{s,K}^{c,0}}  
\newcommand{\limcoarseSGatMax}{S_{s,K}^i}  
\newcommand{\scaling}{\sqrt{8\pi}}
\newcommand{\scalinglog}{\frac{2}{\sqrt{2\pi}}}
\newcommand{\fineGFF}{X_K^f} 
\newcommand{\distgrid}{\delta} 
\newcommand{\bnuKdel}{\nu_{K,\epsilon}}
\newcommand{\nuKdel}{\nu_K}
\newcommand{\DiscTorusGrid}{\Omega_\epsilon^{\delta}}
\newcommand{\sqrtpi}{} 

\section{Convergence in law of the maximum} 
\label{sec:maximum}

In this section, we show the convergence of the maximum of the regularised
sine-Gordon field, Theorem~\ref{thm:convergence-to-Gumbel}.
To this end, we start from the multiscale coupling of the sine-Gordon field with the GFF
given by Theorem~\ref{thm:coupling}.
In the first step, we show that the non-Gaussian part of the sine-Gordon field
may be removed at small scales.
We then connect our coupling to the decomposition
of the GFF used in \cite{MR3433630}, by decomposing the Gaussian part of our
field once more into
a `fine' and a `coarse' field analogous to those in \cite{MR3433630}.
In fact, we will choose exactly the same Gaussian fine field as in \cite{MR3433630}
so that the key results from that reference immediately transfer to our setting.
Together with regularity estimates for our Gaussian coarse field
analogous to those in \cite{MR3433630}
and the regularity estimates from Theorem~\ref{thm:coupling-bd} for the non-Gaussian
part of the sine-Gordon field,
the convergence then follows along similar lines as in \cite{MR3433630}.

Throughout this section, we always work on a probability space
  on which the fields $\Phi^{\SG_\epsilon}_t$ and $\Phi^{\GFF_\epsilon}_t$ are constructed simultaneously for all $\epsilon \geq 0$ and $t\geq 0$
  as in Section~\ref{sec:gffdecomp}.
As before, we will often drop the index $\epsilon$ from the notation in
$\Phi^{\SG_\epsilon}$ and $\Phi^{\GFF_\epsilon}$.
While all estimates are uniform in $\epsilon$,
throughout this section, 
we will always assume that $\epsilon>0$ is strictly positive (unless emphasised otherwise)
 because we are studying the limiting behaviour of the maximum which is
 singular as $\epsilon\to 0$.
The $\epsilon$-dependent part of the mean of the maximum will be denoted
\begin{equation}
  \label{e:definition-m}
  m_\epsilon = \frac{1}{\sqrt{2\pi}}(2\log \frac{1}{\epsilon} - \frac34 \log \log \frac{1}{\epsilon}).
\end{equation}

Some statements below refer to the L\'evy distance $d$ on the set of probability measures on $\R$.
This is a metric for the topology of weak convergence,
defined for any two probability measures $\nu_1,\nu_2$ on $\R$ by
\begin{equation} \label{e:Levy}
d(\nu_1,\nu_2)= \min \{\kappa >0 \colon \nu_1 (B) \leq \nu_2(B^\kappa)+\kappa \text{~for all open sets~} B \}
\end{equation}
where $B^\kappa=\{y\in \R \colon \dist (y,B) < \kappa\}$.
We will use the convention that when a random variable appears in the argument of $d$, we refer to its distribution on $\R$.
Note that if two random variables $X$ and $Y$ can be coupled with $|X-Y|\leq \kappa$ with probability $1-\kappa$ then $d(X,Y) \leq \kappa$.

\subsection{Regularisation of the non-Gaussian part}
\label{ssec:fine-coarse}

From Theorem \ref{thm:coupling} and \eqref{e:GFFepsdef} it follows that
\begin{equation}
\Phi_0^{\SG} = \Phi_s^{\SG} + (\Phi_0^{\GFF}-\Phi_s^{\GFF}) + R_s
\label{e:phi-SG-decomp}
\end{equation}
where 
\begin{equation}
R_s = - \int_0^s \dot c_t \nabla v_t(\Phi^{\SG}_t) \, dt.
\end{equation}
Note that the first two terms in \eqref{e:phi-SG-decomp} are independent,
by Theorem~\ref{thm:coupling}, and
that $\max_{\Omega_\epsilon}|R_s|$ is of order $|\lZ_s|$ by \eqref{e:PhiDelta-bd0t}
and thus tends to $0$ as $s\to 0$ (recall from \eqref{e:ZL-def} that $|\lZ_s| \leq |z|s^{1-\beta/8\pi}$).
Denoting the main contribution of the field by
\begin{equation}
\tPhisSG \equiv \Phi_s^{\SG} + (\Phi_0^{\GFF}-\Phi_s^{\GFF}),
\label{e:Phit-s}
\end{equation}
we can rewrite \eqref{e:phi-SG-decomp} as
\begin{equation}
\Phi_0^\SG=\tPhisSG + R_s.
\label{e:phi-SG-phi-tilde}
\end{equation}

By the following lemma, it suffices to prove Theorem~\ref{thm:convergence-to-Gumbel} with $\Phi^{\SG}$ replaced by $\tPhisSG$ with $s>0$.

\begin{lemma} \label{lem:reduction-to-phi-tilde}
  Assume that the limiting law $\tilde \mu_s$ of $\max_{\Omega_\epsilon} \sqrtpi \tPhisSG -m_\epsilon$ as $\epsilon\to 0$ exists for every $s>0$,
  and that there are positive random variables $\ZDM_s$ (on the above common probability space) such that
  \begin{equation}    \label{e:mu-s-infty-bis}
    \tilde \mu_s((-\infty,x])=\E[e^{-\alpha^* \ZDM_s e^{-\scaling x}}],
  \end{equation}
  for some constant $\alpha^*>0$.
  Then the law of $\max_{\Omega_\epsilon} \sqrtpi \Phi_0^\SG-m_\epsilon$ converges weakly to
  some probability measure $\mu_0$ as $\epsilon \to 0$ and $\tilde \mu_s \rightharpoonup \mu_0$ weakly as $s\to 0$.
  Moreover,  there is a positive random variable
  $\ZDM^\SG$ such that
  \begin{equation}  \label{e:mu0-ZSG}
    \mu_0((-\infty,x])=\E[e^{-\alpha^* \ZDM^{\SG} e^{-\scaling x}}].
  \end{equation}
\end{lemma}

\begin{proof}
By \eqref{e:phi-SG-phi-tilde}, on the underlying coupled probability space we have
\begin{equation}
\max_{\Omega_\epsilon}\sqrtpi \tPhisSG -m_\epsilon=\max_{\Omega_\epsilon}\sqrtpi \Phi_0^{\SG}-m_\epsilon + O(\max_{\Omega_\epsilon} |R_s|)
\label{e:phi-sg-phi-tilde-s-max}
\end{equation}
with $\max_{\Omega_\epsilon} |R_s| = O(|\lZ_s|) \to 0$ as $s\to 0$.
Since $(\max_{\Omega_\epsilon}\sqrtpi \tPhisSG -m_\epsilon)_\epsilon$ is tight by assumption,
it follows that also $(\max_{\Omega_\epsilon}\sqrtpi \Phi_0^{\SG}-m_\epsilon)_\epsilon$ is tight.
By the latter tightness, there is a subsequence $\epsilon \to 0$ such that the law of $(\max_{\Omega_{\epsilon}}\sqrtpi\Phi_0^{\SG} -m_{\epsilon})_{\epsilon}$ converges to $\mu_0$. Thus along this subsequence,
\begin{align}
d(\tilde\mu_s,\mu_0)
&\leq \limsup_{\epsilon = \epsilon_k \to 0} \Big[ d(\max_{\Omega_\epsilon}\sqrtpi\tPhisSG -m_\epsilon,\tilde\mu_s)
+ d(\max_{\Omega_\epsilon}\sqrtpi\Phi_0^{\SG}-m_\epsilon,\mu_0) \nnb
  &\qquad \qquad+ d(\max_{\Omega_\epsilon}\sqrtpi\Phi_0^{\SG}-m_\epsilon,\max_{\Omega_\epsilon}\sqrtpi\tPhisSG-m_\epsilon) \Big]
    \leq O(|\lZ_s|),
\end{align}
where we recall that $d$ denotes the L\'evy distance.
In particular, by taking $s\to 0$, it follows that $\mu_0$ is unique,
and therefore now along any sequence $s \to 0$,
\begin{equation}
\label{e:laplace-transforms-s-to-0}
\tilde \mu_s((-\infty,x]) \to \mu_0((-\infty,x]).
\end{equation}

Next, we prove that $(\ZDM_s)_s$ is tight. Denote by $\mu_0^\GFF$ the limiting law of the centred maximum of $\Phi_0^\GFF$
which exists by \cite{MR3433630,MR3729618}.
Since $\tPhisSG -\Phi_0^\GFF = \Phi_s^\Delta$ is uniformly bounded, it follows from \eqref{e:mu-s-infty-bis} that
\begin{equation}
\label{e:mu-GFF-estimate}
\mu_0^\GFF((-\infty,x-C]) \leq  \E[e^{-\alpha^* \ZDM_s e^{-\scaling x}}] 
\end{equation}
for some constant $C>0$ which is independent of $s$. Assume $(\ZDM_s)_s$ is not tight. Then,
\begin{equation}
\exists \kappa >0 \colon \forall M>0 \colon \exists s_M \colon \P(\ZDM_{s_M} >M) >\kappa.
\end{equation}
It follows that 
\begin{equation}
\E[e^{-\alpha^* \ZDM_{s_M} e^{-\scaling x}}]
\leq e^{-\alpha^* M e^{-\scaling x}} + \P(\ZDM_{s_M} \leq M) \leq  e^{-\alpha^* M e^{-\scaling x}} + (1-\kappa).
\end{equation}
Sending $M\to \infty$, \eqref{e:mu-GFF-estimate} implies that 
\begin{equation}
\mu_0^\GFF((-\infty, x-C]) \leq 1-\kappa
\end{equation}
which is a contradiction when sending $x\to \infty$.

This means that $(\ZDM_s)_s$ is tight and by \eqref{e:laplace-transforms-s-to-0} the Laplace transforms converge as $s\to 0$.
Thus, L\'evy's continuity theorem for Laplace transforms implies that there is a random variable $\ZDM^\SG$ such that $\ZDM_s\to \ZDM^\SG$ and
\begin{equation}
\E[e^{-\alpha^* \ZDM_s e^{-\scaling x}}] \to \E[e^{-\alpha^* \ZDM^\SG e^{-\scaling x}}]
\end{equation}
as $s\to 0$, proving \eqref{e:thm-max-convergence}.
That $\ZDM^\SG>0$ a.s.\ follows from the tightness of $\max_{\Omega_\epsilon} \Phi^{\SG}_0-m_\epsilon$.
\end{proof}

\subsection{Approximation of small scale field}

As a next step we decompose our Gaussian small scale field $\Phi_s^{\GFF}-\Phi_0^{\GFF}$
(recall the convention that we omit the index $\epsilon>0$)
further in order to connect to the setup of \cite{MR3433630}.
First observe that its covariance can be written as
\begin{equation}
  \int_0^s e^{\Delta t-m^2 t} \, dt = (-\Delta+m^2+1/s)^{-1} + g_s(-\Delta+m^2)
\end{equation}
where
\begin{equation} \label{e:function-g}
  g_s(\lambda) = \frac{1}{\lambda}(1-e^{-\lambda s}) - \frac{1}{\lambda + 1/s} \geq 0.
\end{equation}

Next, we divide $\Omega$ along the grid $\Gamma$,
which is a union of horizontal and vertical lines intersecting at the vertices $\frac{1}{K}\Z^2 \cap \Omega$,
into boxes $V_i \subset \Omega$, $i=1,\dots, K^2$ of (macroscopic) side length $1/K$
such that $\Omega = \cup_{i=1}^{K^2} V_i \cup \Gamma$.
We will always assume that $1/K$ is a multiple of $\epsilon$ such that $\Gamma$ can be
regarded as a subset of $\Omega_\epsilon$ and
$\Omega_\epsilon = \cup_{i=1}^{K^2} (V_i \cap \Omega_\epsilon) \cup \Gamma$.
When no confusion can arise, 
we use the notation $V_i$ for both the subset of $\Omega$ and the corresponding lattice version as subset of $\Omega_\epsilon$. 

Now, let $\Delta_\Gamma$ be the Laplacian on $\Omega$ with Dirichlet boundary conditions on $\Gamma$,
and let $\Delta$ be the Laplacian with periodic boundary conditions on $\Omega$.
The domain of $\Delta$ is the space of $1$-periodic functions,
and that of $\Delta_\Gamma$ is the smaller space of $1$-periodic functions vanishing on $\Gamma$.
This implies that $-\Delta_\Gamma \geq -\Delta$ and $(-\Delta+m^2 + 1/s)^{-1} \geq (-\Delta_\Gamma+m^2 + 1/s)^{-1}$ as quadratic form inequalities.
(The form inequality for the Green function is equivalent to the use of the Markov property in \cite{MR3433630}.)

Hence we can decompose $\Phi^\GFF_0-\Phi^\GFF_s$ in distribution as
\begin{equation}
  \Phi^{\GFF}_0-\Phi^{\GFF}_s  \stackrel{d}{=} \tilde X_{s,K}^f + \coarsesKGauss + X^h_s,
  \label{e:decomp-gaussian-part}
\end{equation}
where the  three fields on the right-hand side are independent Gaussian fields with covariances
\begin{align}
  \label{e:Xf}
  \cov(\tilde X_{s,K}^f) &= (-\Delta_\Gamma + m^2 + 1/s)^{-1}  \\
    \label{e:Xc}
  \cov(\coarsesKGauss) &=  (-\Delta+m^2 + 1/s)^{-1} - (-\Delta_\Gamma+m^2 + 1/s)^{-1}
  \\
  \label{e:Xh}
  \cov(X^h_s) &= g_s(-\Delta+m^2)
\end{align}
where the function $g_s$ is as in \eqref{e:function-g}. 

By definition, when 
restricted to any of the boxes $V_i$, the Laplacian $\Delta_\Gamma$ coincides with the
Dirichlet Laplacian on $V_i$, and 
$(-\Delta_\Gamma + m^2 + 1/s)^{-1} (x,y)=0$ for $x\in V_i$ and $y\in V_j$ with $i\neq j$.
Therefore the Gaussian fields $\tilde X_{s,K}^f|_{V_i}$ and $\tilde X_{s,K}^f|_{V_j}$ are independent
for $i\neq j$.
Moreover, $\tilde X_{s,K}^f$ is essentially a  massless GFF with Dirichlet boundary conditions on~$\Gamma$.
Indeed, we will later choose $1/K^2 \ll (m^{2}+1/s)^{-1}$,
and the mass term $m^2 + 1/s$
will then be irrelevant.
To connect directly to the setup of \cite{MR3433630},
it is convenient to replace $\tilde X_{s,K}^f$ by a field $\fineGFF$
which is in fact exactly a massless GFF with Dirichlet boundary conditions on $\Gamma$, i.e.,
a centred Gaussian field with covariance
\begin{equation}
  \label{e:XKf}
  \cov(\fineGFF) 
  = (-\Delta_\Gamma)^{-1}.
\end{equation}
The next two simple lemmas show that the distribution of the maximum of $\Phi_0^{\GFF}-\Phi_s^{\GFF}$
is indeed well approximated by that of $\fineGFF+\tilde X^c_{s,K}+ X^h_s$
and that $X_s^h$ is H\"older continuous.

\begin{lemma}
  \label{lem:G-concentration}
  Let $G$ be a centred Gaussian field with covariance 
  $(-\Delta_\Gamma)^{-1} - (-\Delta_\Gamma + m^2 + 1/s)^{-1}$.
  Then for $K^2 \geq m^{2}+ 1/s$,
  \begin{equation}
    \P\pB{ \max_{\Omega_\epsilon} G > t}
    \lesssim K^2 \exp\left(-\frac{ct^2}{\sigma^2}\right),
      \qquad
      \sigma^2 = \frac{m^2+1/s}{K^2} + e^{-cK^2s}.
    \end{equation}
   In particular, $\P(\max G > u) \to 0$ for any $u>0$ as $K\to\infty$ with $m$ and $s$ fixed. 
\end{lemma}

\begin{proof}
    Let $C_G$ be the covariance of $G$.
  We first verify the following estimates:
  \begin{align}
    \label{e:cov-est-G}
    C_G(x,x)+C_G(y,y)-2C_G(x,y)
    &\lesssim \sigma^2 (K|x-y|)
    \\
    C_G(x,x)
    &\lesssim \sigma^2
      .
      \label{e:G-variance}
  \end{align}
  These follow from standard estimates on the Dirichlet heat kernel given in
  Lemma~\ref{lem:ptGamma}.
  Indeed, to see \eqref{e:G-variance},
  we use $e^{t\Delta_\Gamma}(x,y) \lesssim (1/t)e^{-ctK^2}$ which holds by \eqref{e:pGammabd} with $\alpha=0$ to get
  \begin{equation}
    C_G(x,x)
    = \int_0^\infty (1-e^{-(m^2+1/s)t})p^\Gamma_{t}(x,x) \, dt
    \lesssim  \int_0^\infty \frac{1-e^{-(m^2+1/s)t}}{t} e^{-K^2 t} \, dt
    \lesssim \sigma^2
    .
  \end{equation}

  To see \eqref{e:cov-est-G}, we similarly apply \eqref{e:pGammabd} with $\alpha=1$ to see
  that the left-hand side of     \eqref{e:cov-est-G} is 
    \begin{multline}
      \int_0^\infty (1-e^{-(m^2+1/s)t}) (p_t^\Gamma(x,x)+p_t^\Gamma(y,y)-2p_t^\Gamma(x,y)) \, dt
      \\
      \lesssim |x-y|
        \int_0^\infty \frac{1-e^{-(m^2+1/s)t}}{t} t^{-1/2} e^{-cK^2t} \, dt
        \lesssim \sigma^2 (K |x-y|).
  \end{multline}

  As a consequence of \eqref{e:cov-est-G} and Fernique's criterion
  (see, for example, \cite[Lemma~3.5]{MR3433630} applied to $G/\sigma$),
  we obtain $\E[\max_{V_i} G] \lesssim \sigma$
  uniformly in $\epsilon$.
  The tail estimate then follows from a union bound over all $K^2$ boxes
  and the Borell-Tsirelson concentration inequality for the maximum (as, for example, stated in \cite[Lemma 3.4]{MR3433630}).
  Indeed, using \eqref{e:G-variance}, this gives
  \begin{align}
    \P(\max_{x\in \Omega_\epsilon}G > u)
    \leq K^2 \P(\max_{V_i}G > u)
    &\leq 2K^2 \P(|\max_{V_i}G - \E[\max_{V_i}G]| > u - O(\sigma))
        \nnb
    &\leq 2 K^2 \exp\left(-\frac{(u - O(\sigma))^2}{O(\sigma^2)}\right)
      \lesssim  K^2 \exp\left(-\frac{cu^2}{\sigma^2}\right)
  \end{align}
  which gives the claim.
\end{proof}

The next lemma is the H\"older continuity of $X^h_s$.

\begin{lemma}\label{lem:hoelder-expectation}
Let $X^h_s$ be a Gaussian field with covariance   \eqref{e:Xh}.
There is $\alpha>0$ such that
\begin{equation}\label{e:hoelder-expectation}
  \sup_{\epsilon \geq 0}\E\pa{\sup_{x\in\Omega_\epsilon} |X_s^h(x)| + \sup_{x \neq y\in\Omega_\epsilon} \frac{|X_s^h(x)-X_s^h(y)|}{|x-y|^\alpha}} \leq O_{s,m}(1).
\end{equation}
\end{lemma}

\begin{proof}
Let $C^h$ be the covariance \eqref{e:Xh} and note that by translation invariance we may identify $C^h(x,y)$
with a function $C^h(x-y)$.
Since $g_s(\lambda) = O_s(1/(1+\lambda^2))$ it follows that
\begin{equation}
  |\hat C^h(k)| \leq O_{s}(1/(1+|k|^4)).
\end{equation}
In particular, $C^h(x,y)$ is Lipschitz continuous uniformly in  $\epsilon$, and Kolmogorov's continuity criterion
implies that \eqref{e:hoelder-expectation} holds.
\end{proof}

\subsection{Coarse field regularity}

In what follows, as in \cite{MR3433630}, it suffices to consider
the maximum of the fields restricted to points that are of macroscopic distance to the grid $\Gamma$.
For $\delta \in (0,1)$, define 
\begin{equation}
  V_i^{\delta}= \{x \in V_i \colon \dist(x, \Gamma) \geq \delta/K \},
  \qquad
  \Omega^\delta= \bigcup_{i=1}^{K^2} V_i^{\delta},
  \qquad
  \DiscTorusGrid= \Omega^\delta \cap (\epsilon \Z^2).
  \label{e:Omega-delta}
\end{equation}

Note that, differently from \cite{MR3433630},
$\dist$ here denotes the macroscopic distance on $\Omega_\epsilon \subset \Omega$
rather than the microscopic distance which differs by a factor $\epsilon$.
The following lemma provides estimates for the covariance of $\coarsesKGauss$.
It is analogous to \cite[Lemma~2.1 and Lemma~3.10]{MR3433630},
except that we have a mass term and periodic boundary conditions.

\begin{lemma}[Version of {\cite[Lemma~2.1 and Lemma~3.10]{MR3433630}}]
  \label{lem:Xc-cov}
  Define the covariance $C^{c} = C^{c,\epsilon}$ as in \eqref{e:Xc}.
  Then for $K^{2}  \geq m^{2}+ 1/s$ and as $\epsilon \to 0$,
  \begin{equation} \label{e:Xc-conv}
    \max_{x,y\in V_i^{\delta}} |C^{c,\epsilon}(x,y)-C^{c,0}(x,y)| \to 0,
  \end{equation}
  and the following estimates hold uniformly in $\epsilon$  and $K$
  (with constants depending on $\delta,s,m$):
  \begin{align}
    \label{e:Xc-cov3}
    C^c(x,y)&\lesssim \log K && (x,y \in \Omega^\delta),
    \\
    \label{e:Xc-cov2}
    |C^c(x,y)-C^c(x,y')| &\lesssim K|y-y'| && (x \in \Omega^\delta, \; y,y' \in V_i^\delta ),
    \\
    \label{e:Xc-cov}
    C^c(x,x)+C^c(y,y)-2C^c(x,y) &\lesssim (K|x-y|)^2 &&(x,y \in V_i^\delta).
  \end{align}
\end{lemma}

We remark that in \cite[Lemma~3.10]{MR3433630}, there is also a lower bound for the analogue of the left-hand side
of \eqref{e:Xc-cov}. This lower bound is however only used in the proof of \cite[Proposition~4.1]{MR3433630}
for which we do not need an analogue given that we can use the already established tail asymptotics
for the Dirichlet GFF in a box (Proposition~\ref{prop:finefield-asymptotics-delta} below).

\begin{proof}
First, observe that $C^c$ can be conveniently expressed as
\begin{equation}
  C^c(x,y)
  = \int_0^\infty (e^{\Delta t} -e^{\Delta_\Gamma t})(x,y)e^{-(m^2 + 1/s)t} \, dt
  = \int_0^\infty (p_t^L(x,y)-p_t^\Gamma(x,y))e^{-(m^2 + 1/s)t} \, dt.
\end{equation}
Hence, we will use the standard estimates on heat kernels given in Lemmas~\ref{lem:pt}--\ref{lem:ptGamma}
and first verify \eqref{e:Xc-cov3}--\eqref{e:Xc-cov}.
Indeed, \eqref{e:pGamma-inside} and \eqref{e:pttorusbounds} give
\begin{equation}
  \partial^\alpha p^\Gamma_t(x,y)
  = \partial^\alpha p_t^{L}(x,y) + O(t^{-d/2-|\alpha|/2}e^{-c\delta/\sqrt{K^2t}})
    + O(t^{-d/2-|\alpha|/2}(\sqrt{t}/L)^{d}e^{-cL/\sqrt{t}}).
\end{equation}
For $t \leq 1$ the second error term is smaller than the first one and conversely for $t>1$.

To see \eqref{e:Xc-cov3}, we integrate this bound with $d=2$ and $\alpha=0$:
\begin{align}
  C^c(x,y)
  &\lesssim 
  \int_0^1 t^{-1} e^{-c\delta /\sqrt{K^2t}}\, dt
  + \int_1^\infty e^{-(m^2 +1/s)t} \, dt 
  \lesssim \log (K/\delta) + (m^2+1/s)^{-1}
  \lesssim \log K.
\end{align}

To see the upper bound in \eqref{e:Xc-cov}, we similarly integrate the above bound with $\alpha=2$:
\begin{align}
  C^c(x,x)+C^c(y,y)-2C^c(x,y) 
  &\lesssim
    |x-y|^2 \pa{\int_0^1 t^{-2} e^{-c\delta /\sqrt{K^2t}} \, dt
    + \int_1^\infty t^{-1}e^{-(m^2 +1/s)t} \, dt }
  \nnb
    &\lesssim (K/\delta)^2|x-y|^2
    \lesssim K^2|x-y|^2.
\end{align}
The bound \eqref{e:Xc-cov2} follows analogously if $x$ is in the same square as
$y$ and $y'$. Otherwise,
$e^{\Delta t}(x,y)-e^{\Delta_\Gamma t}(x,y) = e^{\Delta t}(x,y)$ and likewise with $y'$ instead of $y$.
Since $|x-y|\wedge |x-y'| \geq\delta$, it therefore again follows from \eqref{e:ptbounds} that
\begin{align}
  |C^c(x,y)-C^c(x,y')|
 & \lesssim |y-y'|\pa{\int_0^1 t^{-3/2} e^{-c\delta/\sqrt{t}} \, dt
  +\int_1^\infty  t^{-3/2} e^{-c\delta/\sqrt{t}} e^{-(m^2 +1/s)t}  \,dt}
  \nnb
  &\lesssim |y-y'| \leq K|y-y'|.
\end{align}
The convergence \eqref{e:Xc-conv} follows similarly using \eqref{e:ptlimit} and \eqref{e:ptGammalimit}.
\end{proof}

Finally, we also record the smoothness of $\Phi_s^\GFF$ and $\Phi_s^{\SG}$.

\begin{lemma}
  \label{lem:coarse-smooth}
  For any $s>0$, the fields $\Phi_s^{\GFF}$ and $\Phi_s^{\SG}$ are smooth
  uniformly in $\epsilon \geq 0$.
  In particular, for $\# \in \{\SG,\GFF\}$,
  \begin{equation}\label{e:phi-smooth}
    \sup_{\epsilon\geq 0}\E\pbb{\max_{x\in\Omega_\epsilon} |\Phi_s^{\#}(x)| + \max_{x,y \in \Omega_\epsilon} \frac{|\Phi_s^{\#}(x)-\Phi_s^{\#}(y)|}{|x-y|}} < \infty
    .
  \end{equation}
\end{lemma}

\begin{proof}
  The Gaussian field $\Phi_s^\GFF$ has smooth covariance
  $\int_s^\infty e^{\Delta t -m^2 t}\, dt$ and the standard estimate follows,
  e.g., from Lemma~\ref{lem:PhiGFF-limit}.
  For $\Phi_s^{\Delta}=\Phi_s^{\SG}-\Phi_s^{\GFF}$ the estimate follows from
  \eqref{e:delPhiDelta-bd}.
\end{proof}

\subsection{Generalisation of required GFF results}
In the previous two subsections,
we have argued that the maximum of the sine-Gordon field can be replaced with negligible
error as $s\to 0$ by that of the field
\begin{equation}
  \PhisKSG = \fineGFF + \coarsesKGauss + X^h_s + \Phi^{\SG}_s,
  \label{e:PhisKSG-decomposition}
\end{equation}
where all four terms on the right-hand side are independent fields.
In the remainder of this section we will proceed along similar steps as in \cite{MR3433630}
to establish convergence in law of the centred maximum of $\PhisKSG$.
Since $\fineGFF$ is a GFF with Dirichlet boundary conditions, exactly as in \cite{MR3433630},
most of the steps can be used verbatim from \cite{MR3433630}.
In this section,
we summarise these results for $X^f_K$, and also
generalise the required results from \cite{MR3433630} that involve
the coarse field as well
(which now corresponds to our non-Gaussian field $\tilde X^{c}_{s,K} + X^h_s + \Phi^{\SG}_s$).

The first result concerns the limiting behaviour of the maximum and maximiser of $\fineGFF$.

\begin{proposition}[Exactly {\cite[Propositions~2.2--2.3]{MR3433630}}]
  Let $g(K) \to \infty$ as $K\to\infty$.
  Let $\fineGFF$ be a Gaussian field with covariance \eqref{e:XKf}
  and denote $V^\delta=V_i^\delta$.
  Then there is $\alpha^*>0$ and a non-negative continuous function $\psi: (0,1)^2 \to \R^2$ with
  $\int_{(0,1)^2} \psi = 1$ such that,
  for any $\delta \geq 0$ small enough,
  with $m_\delta = \int_{(\delta, 1-\delta)^2} \psi$ and $\psi^\delta = \psi/m_\delta$,
  and $z=z_\delta$ the maximiser of $\fineGFF$  in $V^{\delta}$, the following holds:
  \begin{equation}
    \lim_{K\to\infty} \limsup_{\epsilon\to 0} \frac{e^{\scaling g(K)}}{g(K)}\P(\cA_{\delta, K,\epsilon}) = \alpha^* m_\delta,
    \qquad
     \cA_{\delta, K,\epsilon} = \ha{ \max_{V^{\delta}} \sqrtpi \fineGFF \geq m_{\epsilon K} + g(K) }.
  \end{equation}
  Moreover, for any open set $A\subseteq (\delta, 1-\delta)^2$ and any sequence $(x_K)_K$ with $x_K\geq 0$ independent of $\epsilon$,
  \begin{multline}
   \lim_{K\to\infty}\limsup_{\epsilon \to 0}\frac{e^{\scaling x_K}g(K)}{g(K) + x_K}\P\pa{\max_{V^{\delta}} \sqrtpi \fineGFF \geq m_{\epsilon K} +g(K) +x_K, K z_\delta \in A \;\bigg|\; \cA_{\delta, K,\epsilon}} \\
   =\lim_{K\to\infty}\liminf_{\epsilon \to 0}\frac{e^{\scaling x_K}g(K)}{g(K) + x_K}\P\pa{\max_{V^{\delta}} \sqrtpi \fineGFF \geq m_{\epsilon K} +g(K) +x_K, K z_\delta\in A \;\bigg|\; \cA_{\delta, K,\epsilon}} 
    =\int_A\psi^\delta(y) dy.
   \label{e:asymptotic-cond-prob}
  \end{multline}
  with the convergence being uniform in the sequence $(x_K)_K$.
  \label{prop:finefield-asymptotics-delta}
\end{proposition}

In the following results,
which also involve the coarse field,
it will be convenient to separate the Gaussian part of $\PhisKSG$.
With $\Phi^\Delta_s$ as in Theorem~\ref{thm:coupling-bd},
we may write
\begin{equation}
  \PhisKSG=  \PhisKGFF + \Phi_s^\Delta,
\label{e:phi-tilde-with-phi-delta}
\end{equation}
where
\begin{equation}
  \PhisKGFF=  \fineGFF + (\coarsesKGauss +  X^h_s +\Phi_s^\GFF).
\label{e:eta-s-K}
\end{equation}

Here $\Phi_s^\Delta$ is non-Gaussian, but
by Theorem~\ref{thm:coupling-bd},
$\max_{\Omega_\epsilon} |\Phi_s^\Delta| = O(1)$ uniformly in $\epsilon$.
Moreover, $\fineGFF$ is independent of $\coarsesKGauss + X^h_s + \Phi_s^\GFF$ in \eqref{e:eta-s-K}.
Hence, using \eqref{e:decomp-gaussian-part}  we see that 
\begin{equation}
  \PhisKGFF \stackrel{d}{=} \Phi_0^\GFF +X^h_s+ G,
\label{e:eta-GFF}
\end{equation}
where $G$ is centred Gaussian with covariance as in Lemma~\ref{lem:G-concentration} and independent of $\Phi_0^\GFF$.

The first result generalises the robustness result for the maximum of the GFF
from \cite[Lemma~3.9]{MR3433630} and \cite[Proposition~1.1]{MR3729618}
to the maximum of $\PhisKSG$.

\begin{lemma}[Version of {\cite[Lemma~3.7]{MR3729618}}]\label{lem:robust-maximum}
Let $\{\phi_u,\, u\in \Omega_\epsilon\}$ be a collection of  random variables independent of $\PhisKSG$ such that for all $u\in \Omega_\epsilon$ and all $y\geq 0$,
\begin{equation}
\P(\phi_u \geq 1+ y) \leq e^{-y^2}.
\end{equation}
Then there exist absolute constants $c,C>0$ such that for any $\kappa >0$, $\epsilon >0$ and $-c \kappa^{-1/2} \leq x$ 
\begin{equation}
  \P(\max_{\Omega_\epsilon} (\PhisKSG  + \kappa \phi )\geq m_\epsilon + x)
  \leq \P(\max_{\Omega_\epsilon} \PhisKSG \geq m_\epsilon + x -\sqrt{\kappa} ) (1+ O(e^{-C^{-1}\kappa^{-1}} )).
\end{equation}
Furthermore, for $\kappa<1$,
\begin{equation}
\E[\max_{\Omega_\epsilon}( \PhisKSG + \kappa \phi) ] \leq \E[\max_{\Omega_\epsilon} \PhisKSG] + C\sqrt{\kappa}.
\end{equation}
\end{lemma}

\begin{proof}
Both inequalities are proved using identical calculations as in the proof of \cite[Lemma 3.9]{MR3433630} together with the estimate 
\begin{equation}
\P(\max_{A} \sqrtpi \PhisKSG \geq m_\epsilon+z-y) \lesssim\left(\frac{\abs{A}}{\abs{\Omega_\epsilon}} \right)^{1/2}z e^{-\scaling(z-y)} + o_{s,\delta,K}(1)
\end{equation}
where $z\geq 1$, $y\geq 0$,  $A\subseteq \Omega_\epsilon$
and $o_{s,\delta,K}(1) \to 0$ as $\epsilon \to 0$.
The latter follows from the analogue inequality (3.23) in \cite[Lemma~3.8]{MR3433630}
(or rather \cite[Proposition~1.1]{MR3729618} applied to the Gaussian field $\PhisKGFF$)
and the fact that $\max_{\Omega_\epsilon}|\Phi_s^\Delta|= O(1)$ uniformly in $\epsilon>0$.
\end{proof}

The next result states that the event that the maximum is not attained close to the grid $\Gamma$  occurs with high probability. 

\begin{proposition}[Version of {\cite[Proposition~5.1]{MR3433630}}]
  \label{prop:maximiser-not-on-grid}
Let $\DiscTorusGrid$ be as in \eqref{e:Omega-delta}. Then,
\begin{equation}
\lim_{\delta \to 0}\limsup_{K\to \infty} \limsup_{\epsilon \to 0} \P (\max_{\DiscTorusGrid}\PhisKSG\neq \max_{\Omega_\epsilon}\PhisKSG)=0.
\end{equation}
\end{proposition}

\begin{proof}
  The proof is essentially the same as that of \cite[Proposition 5.1]{MR3433630} for the GFF. 
  Indeed, by \eqref{e:eta-GFF} the sequence $(\max_{\Omega_\epsilon} \sqrtpi \PhisKGFF -m_\epsilon)_{\epsilon}$ is tight and hence, by \eqref{e:phi-tilde-with-phi-delta}, so is the sequence $(\max_{\Omega_\epsilon} \sqrtpi \PhisKSG -m_\epsilon)_{\epsilon}$.
  Thus, using again \eqref{e:phi-tilde-with-phi-delta} and that $\Phi^\Delta_s$ is bounded
  by a deterministic constant independent of $\epsilon$,
  the claim follows from the result that, for any fixed $x\in\R$,
\begin{equation}
\lim_{\delta \to 0} \limsup_{K\to \infty} \limsup_{\epsilon \to 0} \P(\max_{\Omega_\epsilon \setminus \DiscTorusGrid} \sqrtpi \PhisKGFF-m_\epsilon\geq x)=0,
\label{e:lim-max-close-to-boundary}
\end{equation}
which holds again by \cite[Proposition~1.1]{MR3729618}
\end{proof}

The following result states that, to approximate the maximum of the  full field,
one may evaluate the field at the local maximizers of the field inside each box.

\begin{proposition}[Version of {\cite[Proposition~5.2]{MR3433630}}]
  \label{prop:fine-field-maximisers}
Let $\PhisKSG$ be as in \eqref{e:phi-tilde-with-phi-delta}. Let $z_i \in V_i^{\delta}$ be such that 
\begin{equation}
\max_{V_i^{\delta}} \fineGFF=  \fineGFF(z_i)
\end{equation}
and let $\bar z$  be such that 
\begin{equation}
\max_i \PhisKSG(z_i) = \PhisKSG(\bar z).
\end{equation}
Then for any fixed $\kappa>0$ and small enough $\delta >0$,
\begin{equation}
\lim_{K\to \infty} \limsup_{\epsilon \to 0} \P (\max_{\DiscTorusGrid} \PhisKSG \geq \PhisKSG(\bar z)+\kappa)=0.
\label{e:Psi-fine-field-maximisers}
\end{equation}
Moreover, there is a function $g\colon \N \to \R_0^+$ with $g(K) \to \infty$ as $K\to \infty$, such that 
\begin{equation}
\lim_{K\to \infty} \limsup_{\epsilon \to 0} \P ( \sqrtpi \fineGFF(\bar z)\leq m_{\epsilon K}+g(K))=0.
\label{e:GFF-fine-field-maximisers}
\end{equation}
\end{proposition}

\begin{proof}
Again, our argument is an adaption of the proof of the analogous statement for the GFF,
\cite[Proposition 5.2]{MR3433630}, where it is shown that the event on left-hand side
of \eqref{e:Psi-fine-field-maximisers} is contained in the union of suitable small
events $A_0,\dots,A_3$, defined as follows.
Let $\coarseGauss\equiv \coarsesKGauss+X_s^h +\Phi_s^\GFF$ be the overall coarse part of $\PhisKGFF$. Set $k=\log K$. Similarly as in the proof of \cite[Proposition 5.2]{MR3433630}, we define
\begin{align}
A_0 &= \{ \max_{\DiscTorusGrid} \PhisKGFF < m_\epsilon -C\}
\label{e:def-A0}\\
A_1&=\{\max_{v\in \DiscTorusGrid} \max_{w \in \DiscTorusGrid \colon \abs{v-w} \leq \epsilon f(k)} \abs{\coarseGauss(v)-\coarseGauss(w)} \geq  \kappa/2 \}\\
A_2&=\{ \max_i \max_{u,v \in V_i^{\delta}} \left(\PhisKGFF(u) + \coarseGauss(u)-\coarseGauss(v) \right) \geq m_\epsilon +C'  \}\\
A_3&= \{\exists i, v \colon 1/K \geq \abs{v-z_i} > \epsilon f(k), \, \PhisKGFF(v) \geq m_\epsilon -C, \PhisKGFF(z_i) \geq m_\epsilon -2C-C'\}
\label{e:def-A3}
\end{align}
for some constants $C,C'>0$ and $\kappa>0$
and where $f(k)$ is a function that grows to $\infty$ as $k \to \infty$.
We emphasise that, differently from \cite{MR3433630}, we always use the macroscopic distance
$\abs{\cdot}$ on $\Omega_\epsilon$ which differs from the microscopic lattice distance used in \cite{MR3433630}
by a factor $\epsilon$.
As in the proof of \cite[Proposition~5.2]{MR3433630}, we have that  
\begin{equation}
\lim_{C,C'\to \infty}\limsup_{K\to \infty}\limsup_{\epsilon \to 0} \P(A_l)=0,
\label{e:conv-Ak}
\end{equation}
for $l=0,\ldots,3$.
Indeed, $\PhisKGFF$ has precisely the same fine field as in \cite{MR3433630},
and the events $A_l$ only involve the Gaussian part of our coarse field. 
The minor difference between the field $\coarsesKGauss$ and the one in \cite{MR3433630},
i.e., periodic compared to Dirichlet boundary conditions and a mass term,
does not matter since the field satisfies the same covariance estimates,
by Lemma~\ref{lem:Xc-cov}, which is all that enters the proof of \eqref{e:conv-Ak}.
Moreover, using \eqref{e:hoelder-expectation} and \eqref{e:phi-smooth} it is easy to see that the continuous and $K$-independent fields $\Phi_s^\GFF$ and $X^h_s$ in the coarse part of $\PhisKSG$ are irrelevant for the convergence in \eqref{e:conv-Ak}.
Hence, the conclusion for the limits of $\P(A_l)$ is the same.

Having \eqref{e:conv-Ak} for the Gaussian part of our field in place, it remains to consider the
effect of the non-Gaussian part $\Phi^\Delta_s$.
For this,
it suffices to prove that with
\begin{equation}
  B:=
  \hbb{ \max_{V_j^{\delta}} \PhisKGFF \geq \PhisKGFF(z_j) +  \kappa/2, \text{~~$V_j^{\delta}$ the box where $\PhisKSG |_{\Omega^\delta}$ is maximal} },
\end{equation}
we have:
\begin{enumerate}
\item
  For any given $\kappa>0$ and $K$ large enough,
  $\h{\max_{\DiscTorusGrid} \PhisKSG \geq \PhisKSG(\bar z) +  \kappa }
    \subseteq B$.
  \item $B \subseteq A_0 \cup A_1 \cup A_2 \cup A_3$ where the $A_l$ are as in
    \eqref{e:def-A0}--\eqref{e:def-A3}.
\end{enumerate}
The event $B$ is defined similarly as in the proof of \cite[Proposition 5.2]{MR3433630},
except that compared to their argument we consider the box where $\PhisKSG$ is maximal
rather than the one where $\PhisKGFF$ is maximal.
In the proof of both statements we will use that
Theorem~\ref{thm:coupling-bd} implies that there is $c>0$, such that for $x,y \in V_i$,
\begin{equation}\label{e:continuity-phi-delta}
  |\Phi_s^{\Delta}(x)-\Phi_s^{\Delta}(y)| \leq c_{s,K} = 
  O_s(K^{-c}).
\end{equation}

\smallskip
\noindent
\textit{Proof of (i):} Using  \eqref{e:continuity-phi-delta}
and since $V_j^\delta$ is the box on which $\PhisKSG |_{\Omega^\delta}$ is maximal,
we have
\begin{equation}
  \max_{V_j^{\delta}} \PhisKGFF = \max_{V_j^{\delta}} (\PhisKSG-\Phi_s^\Delta)
  \geq \max_{V_j^{\delta}} \PhisKSG - \Phi_s^\Delta(z_j) - c_{s,K} 
  = \max_{\DiscTorusGrid} \PhisKSG -\Phi_s^\Delta(z_j) - c_{s,K}
  .
\end{equation}
Thus assuming that the event $\{\max_{\DiscTorusGrid}\PhisKSG \geq \PhisKSG(\bar z) + \kappa\}$ occurs, we can further estimate
\begin{equation}
  \max_{V_j^{\delta}}\PhisKGFF \geq \PhisKSG(\bar z) + \kappa - \Phi_s^\Delta(z_j) - c_{s,K}
  \geq \PhisKGFF(z_j) + \kappa -c_{s,K}
  \geq \PhisKGFF(z_j) + \kappa/2.
\end{equation}
This completes the proof of (i).

\smallskip
\noindent
\textit{Proof of (ii):}
The inclusion for $B$ follows similarly as in the proof of \cite[Proposition 5.2]{MR3433630}.
Fix some constants $C,C'$. We will first show that
\begin{equation}
B \subseteq \tilde A_0 \cup A_1 \cup A_2 \cup A_3
\label{e:inclusion}
\end{equation}
where $A_1, A_2, A_3$ are as in \eqref{e:def-A0}--\eqref{e:def-A3} and  where
\begin{equation}
\tilde A_0=\{\max_{V_j^{\delta}} \PhisKGFF < m_\epsilon -C\}.
\end{equation}

Let $\tau\in V_j^\delta$ be the maximiser of $\PhisKGFF$ in $V_j^\delta$, recall that $V_j^{\delta}$ is the box where $\PhisKSG$ attains its global maximum on $\DiscTorusGrid$. To see that  \eqref{e:inclusion} holds, we first claim that  $(\tilde A_0 \cup A_2 \cup A_3)^c \subseteq \{ \abs{\tau -z_j} \leq \epsilon f(k) \}$. Otherwise, on $(\tilde A_0 \cup A_2)^c$, we have
\begin{equation}
\coarseGauss(\tau)-\coarseGauss(z_j)=\left( \PhisKGFF(\tau) + \coarseGauss(\tau) - \coarseGauss(z_j)  \right)- \PhisKGFF(\tau) \leq m_\epsilon + C' - (m_\epsilon - C)= C+C'
\end{equation}
and thus, since $z_j$ maximises $\fineGFF=\PhisKGFF-\coarseGauss$ in $V_j^\delta$,
\begin{equation}
\PhisKGFF(z_j) \geq \PhisKGFF(\tau) - \left(\coarseGauss(\tau)-\coarseGauss(z_j)\right) \geq m_\epsilon -C - (C+C')= m_\epsilon - 2C -C'.
\end{equation}
This is not consistent with being in $A_3^c$. Consequently,
\begin{equation}
(\tilde A_0 \cup A_2 \cup A_3)^c \subseteq \{\abs{\tau-z_j} \leq \epsilon f(k)\}.
\end{equation}
But then, we also have
\begin{align}
  (\tilde A_0 \cup A_1 \cup A_2 \cup A_3)^c \subseteq A_1^c \cap \{\abs{\tau-z_j} \leq \epsilon f(k)\}
  &\subseteq \{\abs{\coarseGauss(\tau)-\coarseGauss(z_j)} <  \kappa/2 \}  \nnb
&\subseteq \{\PhisKGFF(\tau)-\PhisKGFF(z_j) < \kappa/2 \} =B^c
\end{align}
which is \eqref{e:inclusion}.

It remains to show that $\tilde A_0\subseteq A_0$, where $A_0$ is as in \eqref{e:def-A0} but for a possibly different choice of the constant $C$. Denote by $\zeta\in V_j^\delta$ the global maximiser of $\PhisKSG$ in $\DiscTorusGrid$.
Since $\PhisKGFF$  is maximal at $\tau$ in $V_j^\delta$ and using the continuity of $\Phi_s^\Delta$, see \eqref{e:continuity-phi-delta}, we obtain for $K$ large enough that
\begin{equation}
  \max_{\DiscTorusGrid} \PhisKSG -\PhisKSG(\tau)
  = \PhisKSG(\zeta)- \PhisKSG(\tau)= \PhisKGFF(\zeta) - \Phi_s^\Delta(\zeta) - \left(\PhisKGFF(\tau) -\Phi_s^\Delta(\tau)\right) 
\leq \Phi_s^\Delta(\tau) - \Phi_s^\Delta(\zeta) < 1.
\end{equation}
Hence, we have with $\tilde C >0$ such that $\max_{\Omega_\epsilon} \abs{\Phi_s^\Delta} \leq \tilde C$ by Theorem~\ref{thm:coupling-bd},
\begin{align}
  \tilde A_0 = \{\PhisKGFF(\tau) < m_\epsilon -C\} \subseteq \{ \PhisKSG(\tau) \leq m_\epsilon - C + \tilde C \}
  &\subseteq \{ \max_{\DiscTorusGrid} \PhisKSG \leq m_\epsilon -C + \tilde C + 1\} \nnb
&\subseteq \{\max_{\DiscTorusGrid} \PhisKGFF \leq m_\epsilon -C + 2 \tilde C + 1\}
\end{align}
where we used in the last step that the global maxima of $\PhisKSG$ and $\PhisKGFF$ differ by at most $\tilde C$. This completes the proof of (ii).

It remains to prove \eqref{e:GFF-fine-field-maximisers}. While our fine field $\fineGFF$ is the same as in \cite{MR3433630},
the only difference to the analogous equation in \cite[Proposition~5.2]{MR3433630} is that our $\bar z$ is defined via $\PhisKSG$ rather than via $\PhisKGFF$.
We first fix a constant $\gamma>0$, which will be adjusted throughout the argument. From \eqref{e:conv-Ak} for $A_0$ and again $\max_{\Omega_\epsilon} |\Phi_s^\Delta | \leq \tilde C$, we obtain 
\begin{equation}
\lim_{K\to \infty} \liminf_{\epsilon\to 0} \P(\max_{\DiscTorusGrid}\PhisKSG \geq m_\epsilon - \gamma \log k) = 1
\label{e:max-PhisKSG-geq-m-eps}
\end{equation}
 where $k = \log K$. By \eqref{e:Psi-fine-field-maximisers}, and adjusting $\gamma$, we also have
\begin{equation}
\lim_{K\to \infty} \liminf_{\epsilon \to 0} \P(\PhisKSG(\bar z) \geq m_\epsilon - \gamma \log k)=1.
\end{equation}
Hence, since $m_\epsilon- m_{\epsilon K} = \scalinglog k + o_K(1)$ with $o_K(1) \to 0$ as $\epsilon\to 0$ for fixed $K$, it suffices to find an appropriate function $g(K)$, such that
\begin{equation}
\lim_{K\to \infty} \limsup_{\epsilon\to 0} \P(\max_{i} \coarseGauss(z_i)  \geq \scalinglog k - \gamma \log k - g(K) )=0.
\label{e:g-function-coarse-field}
\end{equation}
Since our Gaussian coarse field $\coarsesKGauss$ satisfies the same covariance estimates as the coarse field in \cite{MR3433630}
and since the continuous and $K$-independent fields $\Phi_s^\GFF$ and $X^h_s$ are again irrelevant,
   exactly the same argument as in the proof of the analogous result in \cite[Proposition~5.2]{MR3433630} shows that $g(K)=\alpha \log k$ for some $\alpha>0$ suffices.
Finally, writing
\begin{multline}
\P(\fineGFF(\bar z) \leq m_{\epsilon K}+g(K) ) = \P(\PhisKSG(\bar z)-\coarseGauss(\bar z)-\Phi_s^\Delta(\bar z) \leq m_{\epsilon K} +g(K)) \\
\leq \P(\max_{i} \coarseGauss(z_i) \geq \scalinglog  k + C_\epsilon(K) - \gamma \log K - g(K)) + \P(\PhisKSG(\bar z) < m_\epsilon - \gamma\log K),
\end{multline}
\eqref{e:GFF-fine-field-maximisers} follows by taking first $\epsilon \to 0$ and then $K\to \infty$, as the first probability vanishes by \eqref{e:g-function-coarse-field} and that the second probability vanishes by \eqref{e:max-PhisKSG-geq-m-eps}.
\end{proof}

\subsection{Approximation of the maximum}

In this section,
we adapt the approximation of the centred maximum by $\epsilon$-independent random variables
from \cite[Section~2.3]{MR3433630}.
Let
\begin{equation}
\maxPhisKSG = \max_{\Omega_\epsilon}\PhisKSG.
\end{equation}
Following \cite[Section 2.3]{MR3433630} we will approximate $\maxPhisKSG-m_\epsilon$ by
$G^*_{s,K}= \max_{i}G_{s,K}^i$ where 
\begin{equation}
G_{s,K}^i= \rho_{K}^i (Y_{K}^i+g(K)) + \limcoarseSG (\bu_\delta^i) - \scalinglog \log K,
\label{e:approximation-maximum}
\end{equation}
and also define
\begin{equation} \label{e:cZdef}
 \ZDM_{s,K}=m_\delta \frac{1}{K^2} \sum_{i=1}^{K^2} (\scalinglog \log K - \limcoarseSG(\bu_\delta^i)) e^{ -2\log K +\scaling \limcoarseSG(\bu_\delta^i) }
\end{equation}
Here the sequence $g(K)$ is as in Proposition~\ref{prop:finefield-asymptotics-delta}
and the random variables in \eqref{e:approximation-maximum} and \eqref{e:cZdef}
are all independent and defined as follows. 
\begin{itemize}
\item
  The random variables $\rho_{K}^i \in \{0,1\}$, $i=1,\dots, K^2$, are independent Bernoulli random variables
with $\bbP(\rho_{K}^i=1)=\alpha^*  m_\delta g(K) e^{-\scaling g(K)}$
with $\alpha^*$ and $m_\delta$ as in Proposition~\ref{prop:finefield-asymptotics-delta}.
\item
  The random variables $Y_{K}^i \geq 0$, $i=1,\dots, K^2$,
  are independent and characterised by $\bbP(Y_{K}^i\geq x)= \frac{g(K)+x}{g(K)}e^{- \scaling x}$ for $x\geq 0$.
\item
  The random field $\limcoarseSG(x)$, $x\in\Omega$, is the limit of the overall coarse field
  $\coarseSG\equiv \coarsesKGauss+X_s^h+\Phi_s^\SG$ as $\epsilon\to 0$.
  The existence of this limit is guaranteed by Theorem~\ref{thm:Phi-limit}
  and Lemma~\ref{lem:Xc-cov}.
\item
  The random variables $\bu_\delta^i \in V_i^\delta$, $i=1,\dots,K^2$,
  have the limiting distribution of the maximisers $z_i$ of $\fineGFF$ in $V_i^{\delta}$
  as $\epsilon\to 0$.
  Thus $\bu_\delta^i$ takes values in the $i$-th subbox of $\Omega=\T^2$ and,
  scaled to the unit square, its
  density is $\psi^\delta$ as in Proposition~\ref{prop:finefield-asymptotics-delta}.
\end{itemize}
Note that that the correction in \eqref{e:approximation-maximum} can be understood from
\begin{equation}
m_{\epsilon  K}-m_{\epsilon}= - \scalinglog \log K + O_K(\epsilon).
\label{e:correction-approximation}
\end{equation}

\begin{theorem}[Extension of  {\cite[Theorem 2.4]{MR3433630}}]
\label{thm:levy-convergence-psi}
With the notation introduced around \eqref{e:Levy},
\begin{equation}
\lim_{\delta \to 0} \limsup_{K\to \infty} \limsup_{\epsilon \to 0} d(\sqrtpi \max_{\Omega_\epsilon}\PhisKSG - m_\epsilon,G_{s,K}^*)=0.
\label{e:levy-convergence-psi}
\end{equation}
\end{theorem}

The remainder of this section is devoted to the proof of Theorem~\ref{thm:levy-convergence-psi},
following \cite[Section~6.1]{MR3433630} closely.
The approach of the proof is to couple the field $\PhisKSG$ with the random variables $G_{s,K}^i$.
In \cite[Section~6.1]{MR3433630} such a coupling is constructed between the independent random variables $(\rho_K^i, Y_K^i, \bu_\delta^i)$ and the values and locations of the maxima of the fine field $\fineGFF$.
This is done for each box independently, so we restrict for now to the first box $V_1^{\delta}$ and the random variables $(\rho_K^1,Y_K^1, \bu_\delta^1)\equiv(\rho,Y, \bu_\delta)$. Since the field $\PhisKSG$ has the same fine field, we may use the same coupling to prove Theorem \ref{thm:levy-convergence-psi}. We first recall this coupling by stating the following three results for $\fineGFF$ from \cite{MR3433630}. Recall that we set $k=\log K$.

\begin{lemma}[Exactly {\cite[Lemma 6.1]{MR3433630}}]
  \label{lem:coupling-C*}
There exists a constant $C ^*>0$ such that, for all $K$,
\begin{equation}
  \limsup_{\epsilon \to 0}  \P(\max_{V_1} \fineGFF \geq m_{\epsilon K} + C^*k) \leq K^{-3}
  .
\end{equation}
\end{lemma}

Set $\theta_K(x)= e^{\scaling (x+g(K))}/(g(K)+x)$ for $x\geq 0$. 

\begin{lemma}[Exactly {\cite[Lemma 6.2]{MR3433630}}]
  \label{lem:coupling-alpha}
There exists a sequence $(\epsilon_K)_K$, $\epsilon_K \to 0$ as $K\to \infty$ and a sequence of numbers $\alpha_{1,K,\epsilon}< \alpha_{2,K,\epsilon}< \ldots < \alpha_{C^*k,K,\epsilon}$ satisfying
\begin{equation}
\abs{\alpha_{j,K,\epsilon}-(g(K) +j-1)} \leq \epsilon_K
\end{equation}
such that for all $j=1,\ldots, C^*k$,
\begin{equation}
\theta_K(0)^{-1} \P(Y_K^i \in [j-1,j])= \P(\max_{V_1} \fineGFF-m_{\epsilon K}\in [\alpha_{j,K,\epsilon},\alpha_{j+1,K,\epsilon})).
\end{equation}
\end{lemma}

Recall that $\fineGFF(z_1)= \max_{V_1^{\delta}} \fineGFF$.
The following coupling result relates  the fine field maximum
to the $\epsilon$-independent random variables in \eqref{e:approximation-maximum}.

\begin{proposition}[Exactly {\cite[Proposition 6.3]{MR3433630}}]
\label{prop:coupling-with-G*}
Let $C^*$ be as in Lemma \ref{lem:coupling-C*} and let $\alpha_{j,K,\epsilon}$, $j=1,\ldots C^*k$, be as in Lemma \ref{lem:coupling-alpha}.
There exists a sequence $(\bar \epsilon_K)_K$ depending on $\delta$,  $\bar \epsilon_K \to 0$ as $K\to \infty$, such that $(\bar \rho_{\delta,K,\epsilon}, \fineGFF(z_1)-m_{\epsilon K}, z_1)$ and $(\rho, Y, \bu_\delta)$ can be constructed on the same probability space with $\rho=\bar \rho_{\delta,K,\epsilon}$ holding with probability 1, and such that, on the event $\{\fineGFF(z_1) - m_{\epsilon K} \leq \alpha_{C^*k,K,\epsilon}\}$,
\begin{equation}
\rho\abs{g(K) +Y -(\fineGFF(z_1)-m_{\epsilon K})} + K\abs{\bu_\delta -  z_1} \leq \bar \epsilon_K.
\label{e:coupling-distances}
\end{equation}
\end{proposition}


As in \cite{MR3433630} we need a continuity result which shows that the maximum of $\PhisKSG$
effectively does not change if the coarse field is evaluated close to the fine field maximum.
The proof of the GFF analogue of this statement in \cite[Lemma~6.4]{MR3433630} and \cite[Lemma 10.3.3]{acostathesis}
is based on Gaussian techniques for the coarse field,
in particular the Sudakov--Fernique inequality.
Our coarse field  is non-Gaussian, but we show that using the decomposition \eqref{e:PhisKSG-decomposition}
we may condition on the non-Gaussian part $\Phi_s^\SG+X_s^h$
and apply the Sudakov--Fernique inequality with nonzero mean to the conditional measure.
Therefore, we denote  $\tilde \eta = \fineGFF + \coarsesKGauss$ so that
\eqref{e:PhisKSG-decomposition} reads $\PhisKSG =  \tilde \eta + \Phi_s^\SG + X_s^h$.
Note that in both decompositions the terms are independent.
Also recall that $\coarseSG \equiv \coarsesKGauss+X_s^h+\Phi_s^\SG$
is the overall (non-Gaussian) coarse part of $\PhisKSG$.

\begin{lemma}[Version of {\cite[Lemma~6.4]{MR3433630}}]  \label{lem:coarse-field-perturbation}
  Let $z_i$ be as in Proposition \ref{prop:fine-field-maximisers},
  let $(\bar \epsilon_K)_K$ be as in Proposition \ref{prop:coupling-with-G*},
  and let $z'_i$, $i=1,\ldots,K^2$, denote a family of independent random variables chosen so that $z'_i$ is measurable with respect to $\sigma( \fineGFF(v)\colon v \in V_i^{\delta})$ and that $K \abs{z_i-z'_i} \leq \bar \epsilon_K$.
  Then
\begin{equation}\label{e:coarse-field-perturbation}
\lim_{K\to \infty} \limsup_{\epsilon \to 0} d(\max_{i}(\fineGFF(z_i) + \coarseSG(z_i)),\max_{i}(\fineGFF(z_i) + \coarseSG(z'_i)))=0.
\end{equation}
\end{lemma}

\begin{proof}
Note that by Lemma~\ref{lem:coarse-smooth} and Lemma \ref{lem:hoelder-expectation}
we may immediately replace $\Phi_s^\SG(z'_i)$ by $\Phi_s^\SG(z_i)$ and $X_s^h(z_i')$ by $X_s^h(z_i)$  in
\eqref{e:coarse-field-perturbation} since their H\"older continuity implies that, for any $\kappa >0$,
\begin{align}
  \lim_{K\to\infty}\limsup_{\epsilon \to 0} \P(\max_i|\Phi_s^\SG(z_i) - \Phi_s^\SG(z'_i) |>\kappa) = 0,
  \\
  \lim_{K\to\infty}\limsup_{\epsilon \to 0} \P(\max_i|X_s^h(z_i) - X_s^h(z'_i) |>\kappa) = 0.
\end{align}
Thus,
with $M=\max_i (\fineGFF(z_i) + \Phi_s^\SG(z_i) +X_s^h(z_i)+  \coarsesKGauss(z_i))$
and $M'=\max_i(\fineGFF(z_i) + \Phi_s^\SG(z_i) + X_s^h(z_i) + \coarsesKGauss(z'_i))$,
it suffices to prove
\begin{equation}
  \label{e:M-M-prime-in-prob}
  \lim_{K\to\infty}\limsup_{\epsilon\to 0}\P(\abs{M-M'}>\kappa) =0.
\end{equation}

In the remainder of the proof we proceed along similar steps as in the proofs of \cite[Lemma 6.4]{MR3433630} and \cite[Lemma 10.3.3]{acostathesis}, i.e.\ we write 
\begin{align}
\P&(|M-M'|>\tilde \epsilon) \leq \P(M-M'> \kappa) +
\P(M'-M>\kappa )\equiv  p_1 + p_2
\end{align}
and show that both probabilities on the right-hand side converge to $0$ when taking first $\epsilon \to 0$ and then $K\to \infty$. 
In each case we introduce fields which reflect the coarse field perturbation and condition on the non-Gaussian, 
but independent part $\Phi_s^\SG$. This allows to connect to corresponding unconditional results proved in \cite{MR3433630}  and \cite{acostathesis}. 
Define $V_{\bar \epsilon_K/K}^{\times 2}= \{(u,v) \colon u,v \in \Omega_\epsilon^{\delta},  \abs{u-v} \leq \bar \epsilon_K/K\}$.

For $p_1$, we use a similar argument as in \cite[Lemma 5.4]{MR3433630}. Restricting to the box $V_j^\delta$ where $M$ is attained, we have
\begin{align}
\label{e:p1-estimate}
p_1 
\leq\P(\coarsesKGauss(z_j)  - \coarsesKGauss (z'_j) >\kappa).
\end{align}
Now, recall that $\tilde \eta = \fineGFF + \coarsesKGauss$ and set 
\begin{align}
\xi_1(x,y)&=\PhisKSG(x) +\coarsesKGauss(x) - \coarsesKGauss(y)\\
\tilde \xi_1(x,y)&= \tilde \eta(x) + \coarsesKGauss(x)-\coarsesKGauss(y),
\end{align}
for $(x,y)\in V_{\bar \epsilon_K/K}^{\times 2}$. Moreover, define $\xi_1^*= \max_{V_{\bar\epsilon_K/K}^{\times 2}} \xi_1$ and similarly $\tilde \xi_1^*$.
Then \eqref{e:p1-estimate} and Markov's inequality imply
\begin{align}
p_1 \leq  \P(\xi_1(z_j, z'_j) - \PhisKSG(z_j) >\kappa) \leq \P(\xi_1^*- \PhisKSG(z_j) > \kappa)  \leq \E[\xi_1^* - \PhisKSG(z_j)]/\kappa.
\label{e:conv-p1}
\end{align}
Since the event $\{\PhisKSG(z_j) \geq \maxPhisKSG - \kappa'\}$  occurs with high probability by Proposition \ref{prop:fine-field-maximisers}, 
it suffices to prove that $\E[\xi^* - \maxPhisKSG] \to 0$ as first $\epsilon \to 0$ and then $K\to \infty$.

Note that conditional on $\Phi_s^\SG+X_s^h$, we have that
$\xi_1$ is a Gaussian field indexed by $V_{\bar \epsilon_K/K}^{\times 2}$ and that 
\begin{equation}
\label{e:xi-tilde-xi}
\xi_1(x,y)= \tilde \xi_1(x,y) + \Phi_s^\SG(x) + X_s^h(x).
\end{equation}
Moreover, $\tilde \xi_1$ is the analogue to the field $Y$ defined in \cite[page 112]{MR3433630}, except that $Y$ is defined for the set 
$V_{1/K,\delta}^{\times 2}= \{(u,v)\colon u,v \in V_i^\delta \text{~for some~} i \}$.
Choosing $K$ large enough, we can ensure that $V_{\bar \epsilon_K/K}^{\times 2} \subseteq V_{1/K,\delta}^{\times 2}$.
Therefore, the estimate in \cite[Lemma 5.3]{MR3433630} holds for $\tilde \xi_1$, i.e.\ there is a constant $c_1$ which is independent of $\epsilon$ and $K$, such that for all $(x,y),(x',y') \in V_{1/K,\delta}^{\times 2}$,
\begin{align}
\label{e:tilde-xi-estimate}
\E[(\tilde \xi_1(x,y) -\tilde \xi_1(x',y'))^2] \leq \E[(\tilde \eta(x)-\tilde \eta(x'))^2] + c_1 ({\mathbf 1}_{x=x'} (K|y-y'|)^2 + {\mathbf 1}_{x\neq x'} ).
\end{align} 
Indeed, examining the proof of \cite[Lemma 5.3]{MR3433630}, 
we see that only the independence of the decomposition $\tilde \eta= \fineGFF + \coarsesKGauss$, the estimates in \cite[Lemma 3.1]{MR3433630}, which we use with $\Phi_0^\GFF$, and the estimates in \cite[Lemma 3.10]{MR3433630} for the coarse field enter, which are completely analogous to our estimates for $\coarsesKGauss$;
see Lemma~\ref{lem:Xc-cov}. 

We want to replace $\xi_1^*$ in \eqref{e:conv-p1} by a field of the same structure, but with a coarse field perturbation that is independent of $\PhisKSG$, so that Lemma \ref{lem:robust-maximum} applies. Hence, we define 
\begin{align}
\psi_1(x,y)&=\PhisKSG + C_1\bar Z(x,y)\\
\tilde \psi_1(x,y) &= \tilde \eta(x) + C_1\bar Z(x,y),
\end{align}
for $(x,y)\in V_{\bar \epsilon_K/K}^{\times 2}$,
where $C_1>0$ is some constant to be determined later and $\bar Z$ is as in \cite[page 114]{MR3433630} with
\begin{align}
\label{e:bar-Z-1}
\E[(\bar Z(x,y) - \bar Z(x',y') )^2] &\leq C_\delta K |y-y'| \text{~for~} y,y' \in V_i^\delta, \\
\label{e:bar-Z-2}\
\E[\bar Z (x,y) ^2]&\leq C_\delta K |x-y| \text{~for~} x,y \in V_i^\delta.
\end{align}
and $(\bar Z(x,y))_y$ and $(\bar Z(x',y))_y$ being independent for $x\neq x'$. 
As for $\xi_1$, the field $\psi_1$ is Gaussian conditional on $\Phi_s^\SG+X_s^h$ and 
\begin{equation}
\label{e:psi-tilde-psi}
\psi_1(x,y)= \tilde \psi_1(x,y) + \Phi_s^\SG(x) + X_s^h(x).
\end{equation}
Moreover, note that $\tilde \psi_1$ is the analogue to $\bar Y$ in \cite[page 114]{MR3433630}.
Using \eqref{e:bar-Z-1}--\eqref{e:bar-Z-2} together with \eqref{e:tilde-xi-estimate}, it is easy to see that we can choose $C_1$ large enough such that for $(x,y), (x,'y')\in V_{\bar \epsilon_K/K}^{\times 2}$, we have
\begin{equation}
\label{e:sud-fern-assumptions}
\E[(\tilde \xi_1(x,y) - \tilde \xi_1(x',y'))^2]\leq \E[(\tilde \psi_1(x,y) - \tilde \psi_1(x',y'))^2].
\end{equation}

By independence in \eqref{e:xi-tilde-xi} and \eqref{e:psi-tilde-psi} and basic properties of conditional expectation,
we have that for all $(x,y), (x',y')\in  V_{\bar \epsilon_K/K}^{\times 2}$
\begin{equation}
  \E[(\xi_1(x,y) - \xi_1(x',y'))^2|\Phi_s^\SG+X_s^h]\leq \E[(\psi_1(x,y) - \psi_1(x',y'))^2 | \Phi_s^\SG+X_s^h].
\end{equation}
Hence, by the Sudakov--Fernique inequality for non-centred Gaussian fields as stated in
\cite[Theorem~2.2.3]{MR2319516},
\begin{equation}
\E[ \xi_1^*| \Phi_s^\SG+X_s^h] \leq \E[\psi_1^* |\Phi_s^\SG+X_s^h].
\end{equation}
Thus, taking expectation on both sides, we obtain
\begin{equation}
\E[\xi_1^*] \leq \E[\psi_1^*].
\end{equation}

Going back to \eqref{e:conv-p1}, we see that it suffices to prove that
\begin{equation}
\E[\psi_1^*-\maxPhisKSG] \to 0
\end{equation}
as first $\epsilon \to 0$ and then $K\to 0$.
To this end, denote by $B_{\kappa}(x)$ the ball of size $\kappa$ around $x$ intersected with $\Omega_\epsilon^\delta$ and set 
\begin{equation}
\theta_1(x)= C_1 \max_{y\in B_{\bar \epsilon_K/K}(x)} \bar Z(x,y)
\end{equation}
for $x\in \Omega_\epsilon^\delta$.
Then, by \eqref{e:bar-Z-1}, we have for $y,y' \in B_{\bar \epsilon_K/K}(x)$
\begin{equation}
\E[(\bar Z(x,y)-\bar Z(x,y'))^2] 
\leq C_\delta \frac{|y-y'|}{\bar \epsilon_K/K} \bar \epsilon_K.
\end{equation}
Thus, the Fernique criterion as stated in \cite[Lemma 3.5]{MR3433630} applied to the Gaussian field $\bar Z(x,\cdot)/\sqrt{\bar \epsilon_K}$ in $B_{\bar \epsilon_K/K}(x)$ yields
\begin{equation}
\E[\theta_1(x)]\leq C \sqrt{\epsilon_K}
\end{equation}
for some constant $C$.
Moreover, by \eqref{e:bar-Z-2}, we have
$\E[\bar Z(x,y)^2] \leq C_\delta \bar \epsilon_K$ for all $y\in B_{\bar \epsilon_K/K}(x)$. Hence, it follows from the Borell-Tsirelson concentration inequality as stated in \cite[Lemma 3.4]{MR3433630} that
\begin{equation}
\P(\theta_x-\E[\theta_x]>y \sqrt{\bar \epsilon_K}) \leq 2 e^{-\frac{y^2}{2 C_\delta}}.
\end{equation}
Setting $\phi_x= \theta_x/ C \sqrt{\bar\epsilon_K}$ this implies
\begin{equation}
\P(\phi_x -1 > y) \leq 2 e^{-C'y^2}
\end{equation} 
for some absolute constant $C' >0$.
Now an application of Lemma \ref{lem:robust-maximum} (with the maximum over $\Omega_\epsilon \setminus \Omega_\epsilon^\delta$ instead of $\Omega_\epsilon$) allows us to deduce that
\begin{equation}
\E[\psi_1^*] \leq \E[\maxPhisKSG]  + \bar \epsilon_K^q
\end{equation}
for some power $q>0$ depending on $C'$. This completes the proof of $p_1 \to 0$.

To see that $p_2$ converges to $0$, when taking the said limits,
we use a similar argument but with the fields involved defined differently. 
Now, we set
\begin{align}
\xi_2(x,y)&=\PhisKSG(x) + \coarsesKGauss(y) - \coarsesKGauss(x)\\
\tilde \xi_2(x,y) &= \tilde \eta (x) + \coarsesKGauss(y) - \coarsesKGauss(x)
\end{align}
for $(x,y)\in V_{\bar \epsilon_K/K}^{\times 2}$ and define $\xi_2^*= \max_{V_{\bar \epsilon_K/K}^{\times 2}} \xi_2$ and similarly $\tilde \xi_2^*$.
Then, Markov's inequality implies
\begin{align}
p_2 &
\leq \P(\xi_2(z_i,z'_i) - M > \kappa) \leq \E[\xi_2^* - M]/\kappa
\label{e:conv-p2}
\end{align}
As above we use Proposition \ref{prop:fine-field-maximisers} to replace the random variable $M$ by the overall maximum $\maxPhisKSG$.
Hence, it suffices to prove that $\E[\xi^* - \maxPhisKSG]\to 0$ as first $\epsilon \to 0$ and then $K \to \infty$.

To this end, we follow the ideas in the proof of \cite[Lemma 10.3.3]{acostathesis},
 i.e.\ we subdivide  $\DiscTorusGrid$ into  subboxes of sidelength $\bar \epsilon_K/K$ 
and construct Gaussian fields $Z^l$ , $l=1,2$ on $\Omega^\delta$ independent of $\PhisKSG$ and independent of each other which satisfy
\begin{align}
\label{e:Z-1}
\var(Z^l(x))&= C\bar \epsilon_K \text{~for~} x\in V_i^\delta \\
\label{e:Z-2}
c_2 K \abs{x-y} \leq \E[(Z^l(x)&-Z^l(y))^2] \leq C_2 K \abs{x-y}
\end{align}
where $x,y$ in the same subbox of sidelength $\bar \epsilon_K/K$ and $c_2, C_2,C>0$ are absolute constants to be chosen later. 
Then we set for $(x,y) \in V_{\bar \epsilon_K/K}^{\times 2}\cap V_{1/K,\delta}^{\times 2}$
\begin{align}
\psi_2(x,y)= \PhisKSG(x) + Z^1(y) + Z^2(x) \\
\tilde \psi_2(x,y) = \tilde \eta (x) + Z^1(y) + Z^2(x)
\end{align}
and define $\psi_2^* = \max_{ V_{\bar \epsilon_K/K}^{\times 2}} \psi_2$, and similarly $\tilde \psi_2^*$.
As before we first replace the expectation of $\xi_2^*$ in \eqref{e:conv-p2} by that of $\psi_2^*$ conditioning on $\Phi_s^\SG+X_s^h$ and using the Sudakov-Fernique inequality. Here, the analogue of \eqref{e:sud-fern-assumptions} is proved in \cite[Lemma 10.3.3]{acostathesis} by distinguishing different cases depending on the locations of $x$ and $y$. The same calculations continue to hold for the fields $\tilde \xi_2$ and $\tilde \psi_2$,
as the only assumptions that enter are the estimates of Lemma~\ref{lem:Xc-cov}.

From there on, the rest of the argument is completely analogous: using the Borell-Tsirelson inequality together with the Fernique criterion and \eqref{e:Z-1}--\eqref{e:Z-2},
we see that Lemma \ref{lem:robust-maximum} applies, and hence that
\begin{equation}
\E[\psi_2^*]\leq \E[\maxPhisKSG] + \bar \epsilon_K^q
\end{equation}
for some $q>0$. This concludes the proof of Lemma~\ref{lem:coarse-field-perturbation}.
\end{proof}
  
\begin{proof}[Proof of Theorem \ref{thm:levy-convergence-psi}]
Let $\restmaxPhisKSG =\max\{\PhisKSG(z_i) \colon i=1,\dots,K^2, \; \fineGFF(z_i) \geq m_{\epsilon K} + g(K)\}$,
with $g(K)$ as in \eqref{e:GFF-fine-field-maximisers}.
Then $\restmaxPhisKSG \leq \PhisKSG(\bar z)$.
On the other hand,
\begin{equation}
\forall \kappa>0 \colon \exists \delta >0, K_0 \in \N, \epsilon_0 >0 \colon \forall K \geq K_0, \forall \epsilon \leq \epsilon_0 \colon \P(\maxPhisKSG > \restmaxPhisKSG + \kappa) \leq \kappa,
\label{e:proof-levy-dist-convergence}
\end{equation}
which follows using \eqref{e:Psi-fine-field-maximisers}, Proposition \ref{prop:maximiser-not-on-grid}, and \eqref{e:GFF-fine-field-maximisers} since
\begin{equation}
  \P(\maxPhisKSG> \restmaxPhisKSG+ \kappa)
 \leq \P(\max_{\DiscTorusGrid} \PhisKSG > \PhisKSG(\bar z) + \kappa)+ \P(\maxPhisKSG\neq \max_{\DiscTorusGrid} \PhisKSG) + \P(\PhisKSG(\bar z) > \restmaxPhisKSG).
\end{equation}

Let $\nuKdel$ be the law of $\restmaxPhisKSG - m_\epsilon$.
Then by \eqref{e:proof-levy-dist-convergence}
 and the definition of the L\'evy distance \eqref{e:Levy},
\begin{equation}
\forall \kappa >0 \colon \exists \delta>0,K_0 \in \N, \epsilon_0>0 \colon \forall K \geq K_0, \epsilon \leq \epsilon_0\colon  d(\maxPhisKSG-m_\epsilon, \nuKdel)< \kappa.
\label{e:proof-levy-dist-conv-2}
\end{equation}

Now, we use the coupling from Proposition \ref{prop:coupling-with-G*} and Lemma \ref{lem:coarse-field-perturbation} for each box $V_i$ to relate the various maxima of $\PhisKSG$ to the random variables $G_{s,K}^*$ from \eqref{e:approximation-maximum}.
Let $(\rho_K^i,Y_K^i,\bu_\delta^i)_i$ be independent and distributed as $(\rho, Y, \bu_\delta)$.
Recall that $\coarseSG \equiv \coarsesKGauss + X_s^h+ \Phi_s^\SG$ is the overall lattice coarse field of $\PhisKSG$ and let $\bnuKdel$ be the law of 
\begin{equation}
\max_{\{i\colon \rho_K^i=1\}} Y_K^i + \coarseSG(\bu_\delta^i) - (m_\epsilon-m_{\epsilon K}) +g(K).
\end{equation}
Here, by writing $\coarseSG(x)$ for $x\in \Omega$ we refer to the coarse field evaluated at the vertex $x^\epsilon$ defined as in section \ref{sec:C-limit}.
Then
\begin{equation}
\forall \kappa >0 \colon \exists \delta>0 ,  K_1 \in \N, \epsilon_1>0 \colon \forall K\geq K_1, \forall \epsilon \leq \epsilon_1 \colon d(\bnuKdel,\nuKdel) \leq \kappa.
\label{e:proof-levy-dist-conv-3}
\end{equation}
Indeed, recalling that $\rho_K^i =1 \iff \fineGFF(z_i) \geq m_{\epsilon K} + \alpha_{1,K, \epsilon}$ almost surely, it follows
from the monotonicity of the $\alpha_{j,K,\epsilon}$ that, up to an event of probability $0$,
\begin{equation}
\{\rho_K^i=0\}\subseteq \{\fineGFF(z_i) < m_{\epsilon K} + \alpha_{C^*k,K,\epsilon}\} \equiv A_i.
\end{equation}
Note that on $\{\rho_K^i =1\}\cap A_i$ we have by \eqref{e:coupling-distances}
\begin{equation}
\abs{g(K) +Y_K^i -( \fineGFF(z_i)-m_{\epsilon K})} \leq \bar \epsilon_K.
\end{equation}
Hence, for any open $B\subseteq \R$,
\begin{align}
\bnuKdel(B) &\leq \P(\max_{\{i\colon\rho_K^i=1\}} \fineGFF(z_i)-m_{\epsilon K} + \coarseSG(\bu_\delta^i) - (m_\epsilon-m_{\epsilon K}) \in B^{\bar \epsilon_K}, \cap_i A_i) + \P((\cap_i A_i)^c) \nnb
&\leq \P(\max_{\{i\colon \rho_K^i=1 \}} \fineGFF(z_i)  + \coarseSG(\bu_\delta^i) - m_\epsilon \in B^{\bar \epsilon_K}) + \P((\cap_i A_i)^c).
\label{e:estimate-nu}
\end{align}

By Lemma~\ref{lem:coarse-field-perturbation},
we may now replace $\coarseSG(\bu_\delta^i)$ by $\coarseSG(z_i)$,
and
we may also  replace $\{i \colon \rho_K^i=1\}$ by ${\{i\colon \fineGFF(z_i) \geq m_{\epsilon K} + g(K)\}}$ on the event $A\equiv \{\fineGFF(\bar z) \geq m_{\epsilon K} + g(K) +\epsilon_K\}$ with $\epsilon_K$ as in Lemma~\ref{lem:coupling-alpha}, as both random variables agree
on this event.
Thus
\begin{equation}
\bnuKdel(B) \leq \nuKdel(B^{\bar \epsilon_K + \kappa/2}) + \kappa/2 + \P((\cap_i A_i)^c) + \P(A^c).
\end{equation}
Using Lemma~\ref{lem:coupling-C*} and a union bound we see that 
\begin{align}
\P(\cup_i A_i^c) &= \P(\exists i \colon \fineGFF(z_i) \geq m_{\epsilon K} + \alpha_{C^*k,K, \epsilon}) \nnb
&\leq K^2  \P(\fineGFF(z_i) \geq m_{\epsilon K}+ C^*k)  \leq 1/K.
\label{e:union-Aic-small}
\end{align}
Thus, by \eqref{e:union-Aic-small}, \eqref{e:GFF-fine-field-maximisers} and the convergence of the sequences $(\epsilon_K)_K$ and $(\bar \epsilon_K)_K$,
\eqref{e:proof-levy-dist-conv-3} follows.

Next we observe that $d(\bnuKdel, \mu_{s,K})\to 0$ as $\epsilon\to 0$. 
Indeed, by Lemmas~\ref{lem:Xc-cov}--\ref{lem:coarse-smooth},
we have the weak convergence of $\coarseSG(x)$ to $\limcoarseSG(x)$ as $\epsilon\to 0$,
for every fixed $x\in \Omega_\epsilon$.
Using independence of $\rho_K^i$, $Y_K^i$ and $\bu_\delta^i$ from $\coarseSG$,
it follows that, for any open $B\subseteq \R$,

\begin{align}
\bnuKdel(B) &= \P(\max_{\{i\colon \rho_K^i =1\}} g(K) + Y_K^i + \coarseSG(\bu_\delta^i) - (m_\epsilon-m_{\epsilon K})\in B ) \nnb
  &\leq \P(\max_{\{i\colon \rho_K^i =1\}} g(K) + Y_K^i + \limcoarseSG(\bu_\delta^i) - (m_\epsilon-m_{\epsilon K})\in B^{\kappa} ) + \kappa, 
\end{align}
which, using \eqref{e:correction-approximation}, is the claimed convergence.

In summary, by
\eqref{e:proof-levy-dist-conv-2} and \eqref{e:proof-levy-dist-conv-3}, we conclude that for any $\kappa > 0$ there is $\delta=\delta(\kappa)$ such that
\begin{equation}
\limsup_{K\to\infty}\limsup_{\epsilon\to 0} d(\maxPhisKSG-m_\epsilon, \mu_{s,K}) < 2 \kappa,
\end{equation}
which proves \eqref{e:levy-convergence-psi}, when sending $\kappa \to 0$, $\delta \to 0$.
%
\end{proof}

\subsection{Limiting distribution of regularised field}

In this section, we apply the previous preparation to generalise
the proof of \cite[Theorem 2.5]{MR3433630} and prove the following theorem.
The proof is entirely analogous, but we include it for completeness.

\begin{theorem}[Analogue of {\cite[Theorem 2.5]{MR3433630}}]
  \label{thm:Gumbel-distribution}
  Let $\mu_{s,K}$ be the law of $G_{s,K}^*$, let $\ZDM_{s,K}$ be as in \eqref{e:cZdef},
  and let $\alpha^*$ be as in Proposition~\ref{prop:finefield-asymptotics-delta}.
  Then
\begin{equation}
\lim_{\delta \to 0} \lim_{K\to\infty} \frac{\mu_{s,K}((-\infty,x])}{\E[e^{-\alpha^* \ZDM_{s,K}e^{-\scaling x}}]}=1.
\label{e:limit-ratio-distr-functions}
\end{equation}
In particular, there is a probability measure $\tilde \mu_{s}$ on $\R$
 and a positive random variable $\ZDM_s$ such that $\mu_{s,K} \to \tilde \mu_s$ weakly and
\begin{equation}\label{e:mu-s-infty}
  \tilde \mu_{s}((-\infty,x])=\E[e^{-\alpha^* \ZDM_s e^{-\scaling x}}].
\end{equation}
\end{theorem}

\begin{remark}
The constant $\alpha^*$ is independent of $s$. Indeed, it is characterised by Proposition~\ref{prop:finefield-asymptotics-delta} which only involves the fine field. Hence, this value is the same as in the corresponding GFF result.
\end{remark}

The following result will be needed in the course of the proof.

\begin{lemma}[Analogue of {\cite[Lemma 6.5]{MR3433630}}]
  \label{lem:GrowthCoarseField}
There exists $\gamma>0$ such that
\begin{equation}
\lim_{K \to \infty} \bbP(\max_{i} \limcoarseSG(\bu_\delta^i) \geq  \scalinglog \log K - \gamma \log \log K)=0.
\end{equation}
\end{lemma}

\begin{proof}
  This is the analogue of \cite[Lemma 6.5]{MR3433630} and the proof is
  identical since our fields are coupled up to order $1$ with those of the GFF.
\end{proof}

\begin{proof}[Proof of Theorem~\ref{thm:Gumbel-distribution}]
  In order to demonstrate \eqref{e:limit-ratio-distr-functions} we first establish the analogue of (6.9) in \cite{MR3433630}.
  Let $\cF^c=\sigma\left( \limcoarseSG(\bu_\delta^i)\colon i=1, \ldots, K^2\right)$
and denote $\limcoarseSGatMax=\limcoarseSG(\bu_\delta^i) - \scalinglog \log K$.
Then, for $x\in \R$,
\begin{equation}
\P(G_{s,K}^*\leq x) = \E\left[\prod_{i=1}^{K^2} \pbb{1- \P(\rho_K^i(Y_K^i+g(K))) > x- \limcoarseSGatMax \mid \cF^c ) } \right].
\label{e:distr-funct-max-via-expection}
\end{equation}
We will prove this equation by first conditioning on $\cF^c$ and then using independence of $(\rho_K^i,Y_K^i)_i$ and $\cF^c$. By the definition of $G_{s,K}^*$ we have
\begin{align}
\P(G_{s,K}^*\leq x)  
&= \E\left[\P\left(\forall i=1,\ldots, K^2 \colon \rho_{K}^i(Y_{K}^i+g(K)) \leq x  -\limcoarseSGatMax | \cF^c\right)  \right].
\label{e:cond-expectation-1}
\end{align}
Since $(\rho_K^i,Y_K^i)_i$ are independent of $\limcoarseSGatMax$, a standard result on conditional expectation allows us to write \eqref{e:cond-expectation-1} as
\begin{align}
\P\left(G_{s,K}^*\leq x \right)=\E\left[ \P\left(\forall i=1,\ldots,K^2\colon(Y_K^i+g(K))\rho_K^i \leq x - w_i\right)\big|_{w_i=\limcoarseSGatMax}\right].
\end{align}
Since the random variables $Y_K^i,\rho_K^i$ are independent for different $i$, we can write the probability as a product, i.e.\
\begin{align}
\P(G_{s,K}^*\leq x)&=\E\Big[\prod_{i=1}^{K^2} \P\left((Y_K^i+g(K))\rho_K^i \leq x - w_i\right) \Big|_{w_i=\limcoarseSGatMax}\Big]\nnb
&=\E\Big[\prod_{i=1}^{K^2}\left( 1- \P((Y_K^i+g(K))\rho_K^i >x - \limcoarseSGatMax \mid \cF^c )\right)\Big]
\label{e:distr-func-max-expectation-product}
\end{align}
which proves \eqref{e:distr-funct-max-via-expection}.

Next we exploit the independence of $\rho_K^i$ and $Y_K^i$ to further simplify the probabilities within the expectation.
To this end, we consider the events $D_K^i= \{x-\limcoarseSGatMax \geq g(K)\}$ and $D_K=\bigcap_{i=1}^{K^2} D_K^i$. By Lemma \ref{lem:GrowthCoarseField} we may assume that $D_K$ occurs.
Using the definition of the $\rho_K^i$ and $Y_K^i$ we have
\begin{align}
  {\mathbf 1}_{D_K^i}\P\left(\rho_K^i(Y_K^i+g(K))>x-\limcoarseSGatMax \mid \cF^c\right)
             &= {\mathbf 1}_{D_K^i}\P\left(\rho_K^i=1\right)\P\left(Y_K^i+g(K) >x-\limcoarseSGatMax \mid \cF^c\right)\nnb
&={\mathbf 1}_{D_K^i} \alpha^* m_\delta (x- \limcoarseSGatMax)e^{-\scaling (x-\limcoarseSGatMax)}.
\label{e:derivative-martingale-prototype}
\end{align}
By the definition of $D_K^i$ the right-hand side in \eqref{e:derivative-martingale-prototype} converges to $0$ as $K\to \infty$. Using $1-x < e^{-x}$ for $x>0$, we thus obtain that, on $D_K$,
\begin{align}
  1-\P\left(\rho_K^i(Y_K^i+g(K))>x -\limcoarseSGatMax \mid \cF^c\right)
  &\leq \exp\left(-\alpha^* m_\delta (x- \limcoarseSGatMax)e^{-\scaling (x-\limcoarseSGatMax)}\right)
     \nnb
&\leq \exp\left(-(1-\epsilon_K)\alpha^* m_\delta (- \limcoarseSGatMax)e^{-\scaling (x-\limcoarseSGatMax)}\right)
\end{align}
for an $x$-dependent sequence $(\epsilon_K)_K$ with $ \epsilon_K\to 0$ as $K\to \infty$.
On the other hand, using $1-x > e^{-(1+\kappa) x}$ for $\kappa>0$ and all $x>0$ small enough,
\begin{align}
  1-\P\left(\rho_K^i(Y_K^i+g(K))>x-\limcoarseSGatMax \mid \cF^c\right)
&\geq \exp\left(-(1+\varepsilon_K)\alpha^* m_\delta (-\limcoarseSGatMax)e^{-\scaling(x-\limcoarseSGatMax)}\right)
\end{align}
for a possibly different sequence $(\epsilon_K)_K$ with $\epsilon_K\to 0$ as $K\to\infty$. Hence, using these bounds together with \eqref{e:distr-funct-max-via-expection}, the limit in \eqref{e:limit-ratio-distr-functions} follows.
\end{proof}

\subsection{Proof of Theorem~\ref{thm:convergence-to-Gumbel}}
\label{ss:proof-main-result}

By collecting the previous results, we complete the proof of
Theorem~\ref{thm:convergence-to-Gumbel}.

\begin{proof}[Proof of Theorem~\ref{thm:convergence-to-Gumbel}] 
  By Lemma \ref{lem:reduction-to-phi-tilde},
  it suffices to show that $\max_{\Omega_\epsilon}\tilde\Phi_s^{\SG}-m_\epsilon \to \tilde\mu_s$, where
  \begin{equation}    \label{e:mu-s-infty-bis2}
    \tilde \mu_s((-\infty,x])=\E[e^{-\alpha^* \ZDM_s e^{-\scaling x}}]
  \end{equation}
  and $\ZDM_s$ are as in Theorem~\ref{thm:Gumbel-distribution}.
Note that \eqref{e:mu0-ZSG} can equivalently be stated as
\begin{equation}
  \mu_0 \stackrel{D}{=} \frac{1}{\scaling} (X + \log \ZDM^\SG + \log \alpha^*).
\end{equation}
where $X$ is an independent standard Gumbel variable.
%
  Moreover, by \eqref{e:PhisKSG-decomposition} and Lemma~\ref{lem:G-concentration},
  \begin{equation}
    \lim_{\delta \to 0} \limsup_{K\to \infty}\limsup_{\epsilon\to 0} d(\max_{\Omega_\epsilon} \tPhisSG-m_\epsilon, \maxPhisKSG-m_\epsilon) =0,
  \end{equation}
  by Theorem~\ref{thm:levy-convergence-psi},
  \begin{equation}
    \lim_{\delta \to 0} \limsup_{K\to \infty}\limsup_{\epsilon\to 0} d(\maxPhisKSG-m_\epsilon, G_{s,K}^*) =0,
  \end{equation}
   and by Theorem~\ref{thm:Gumbel-distribution},
  \begin{equation}
    \lim_{\delta\to 0}\lim_{K\to\infty} d(G_{s,K}^*, \tilde\mu_{s}) = 0,
  \end{equation}
  so that $d(\max_{\Omega_\epsilon}\tilde\Phi_s^{\SG}-m_\epsilon, \tilde\mu_s)\to 0$, as needed.
\end{proof}

\appendix
\section{Heat kernel estimates}
\label{app:pt}


We denote by $p_t^\epsilon(x)$ the heat kernel on $\epsilon\Z^d$
and by $p_t^0(x)$ the continuous heat kernel on $\R^d$:
\begin{equation}
  p_t^\epsilon(x) = \epsilon^{-d}\tilde p_{t/\epsilon^2}(x/\epsilon),
  \qquad
  p_t^0(x) = \frac{e^{-|x|^2/4t}}{(4 \pi t)^{d/2}}
  ,
\end{equation}
where $\tilde p_t$ is the heat kernel on the unit lattice $\Z^d$.
We will also write $\partial f = \partial_\epsilon f$ for the vector of lattice gradients
of a function $f: \epsilon\Z^d \to \R$. Thus if $\alpha= (\alpha_1,\dots,\alpha_{|\alpha|})$
is a sequence of $|\alpha|$ unit directions in $\Z^d$, i.e.,
$\alpha_i \in \{(\cdots, 0, \pm 1, 0, \cdots)\}$ then 
$\partial^\alpha=\prod_{i=1}^{|\alpha|} \partial^{\alpha_i}$
where $\partial^{\alpha_i} f(x) = \epsilon^{-1} (f(x+\alpha_i)-f(x))$.

\begin{lemma} \label{lem:pt}
The heat kernel $p_t^\epsilon$ on $\epsilon\Z^d$ satisfies the following upper bounds
for $t \geq \epsilon^2$, $x\in\Z^d$,
and all sequences of unit vectors $\alpha$:
\begin{equation}
  \label{e:ptbounds}
  |\partial^\alpha p_t^\epsilon(x)| \leq O_\alpha(t^{-d/2-|\alpha|/2}e^{-c|x|/\sqrt{t}}).
\end{equation}
Moreover, for all $x \in \epsilon\Z^d$ and $k \geq 4$,
\begin{align} \label{e:ptlimit}
   |p_t^\epsilon(x)- p_t^0(x)|
  &\leq O(\frac{1}{t^{d/2}}) \times \frac{\epsilon^2}{t}
  \qa{
    \pa{(\frac{|x|}{\sqrt{t}})^k+1} e^{-|x|^2/t}
    + (\frac{\epsilon^2}{t})^{(k-3)/2}
  }.
\end{align}
\end{lemma}

\begin{proof}
  For the upper bound \eqref{e:ptbounds},
  see for example \cite[Lemma A.1]{1907.12308}
  and rescale the statements there by $(t,x) \to (t/\epsilon^2,x/\epsilon)$.
  The convergence estimate \eqref{e:ptlimit} is \cite[Theorem~2.1.3]{MR2677157}
  again after the previous rescaling.
\end{proof}

The heat kernel on a torus of side length $L$ is given by
\begin{equation}
  p_t^{\epsilon,L}(x) = \sum_{y\in\Z^d} p_t^\epsilon(x+yL).
  \label{e:pttorus}
\end{equation}

\begin{lemma} \label{lem:pttorus}
The torus heat kernel satisfies, for $t\geq \epsilon^2$,  $|x|_\infty \leq L/2$, 
\begin{equation} 
  \label{e:pttorusbounds}
  \partial^\alpha p_t^{\epsilon,L}(x)
  = \partial^\alpha p_t^\epsilon(x)
  + O_\alpha(t^{-|\alpha|/2}  L^{-d} e^{-cL/\sqrt{t}} ).
\end{equation}
Moreover, for any $t>0$, as $\epsilon \to 0$,
\begin{equation}
  \label{e:pttoruslimit}
  \sup_{x\in \Omega_\epsilon} |p_t^{\epsilon,L}(x)-p_t^{0,L}(x)|
  \lesssim
    \p{t^{-d/2}+L^{-d}}
    \frac{\epsilon^2}{t} \log(\frac{\epsilon^2}{t})^d
  .
\end{equation}
\end{lemma}

\begin{proof}
  The estimate \eqref{e:pttorusbounds} is straightforward from   \eqref{e:pttorus}
  and Lemma~\ref{lem:pt}. Moreover, with \eqref{e:ptbounds} for $|y|\leq N$
  and \eqref{e:ptlimit} for $|y|>N$, we obtain from \eqref{e:pttorus} that
  \begin{align} \label{e:pttoruslimit-pf}
    |p_t^{\epsilon,L}(x)-p_t^{0,L}(x)|
 &   
    \lesssim \inf_{N\geq 1} \qa{
    \sum_{0< |y| \leq N} t^{-d/2}\frac{\epsilon^2}{t}
   + \sum_{|y| > N} t^{-d/2} e^{-c|y|L/\sqrt{t}}}
   \nnb
&  \lesssim
    \inf_{N\geq 1} \qa{
      (\frac{N}{\sqrt{t}})^{d} \frac{\epsilon^2}{t}
    + (\frac{1}{L})^{d} e^{-cNL/\sqrt{t}}}
  \lesssim
    \p{t^{-d/2}+L^{-d}}
    \frac{\epsilon^2}{t} \log(\frac{\epsilon^2}{t})^d
  \end{align}
  where in the last inequality we have chosen
  $N = (1/c)\log(t/\epsilon^2) \sqrt{t}/L+1$.
\end{proof}

We also need the heat kernel on a square with Dirichlet boundary conditions.
Thus let $\Gamma \subset \R^2$ be the union of horizontal and vertical lines that form a
regular grid of spacing $1/K$,
centred so that $0$ lies at a vertex of this grid.
Let $\Delta_\Gamma$ be the Laplacian with Dirichlet boundary conditions on $\Gamma$.
We denote the square containing $0$ bounded by $\Gamma$
by $V$ and set $V^\delta = \{x \in V: \dist(x,\Gamma) \geq \delta/K\}$.
Moreover, if $\Gamma \subset \epsilon \Z^2$ these definitions have obvious $\epsilon$-dependent
versions.
In the following, we will often omit the index $\epsilon$ and write,
for example,
 $p^\Gamma_t(x,y)$ instead of $p^{\epsilon,\Gamma}_t(x,y)$
when either $\epsilon=0$ or $\epsilon>0$.

\begin{lemma} \label{lem:ptGamma}
  The Dirichlet heat kernel $p^\Gamma_t(x,y) = e^{t\Delta_\Gamma}(x,y)$ (where $\epsilon=0$ or $\epsilon>0$) satisfies
  \begin{equation} \label{e:pGammabd}
    \sup_{x,y \in V}|\partial^\alpha p^\Gamma_t(x,y)| \leq O_\alpha(t^{-d/2-|\alpha|/2} e^{-c tK^2})
  \end{equation}
  and, for any $t>0$, as $\epsilon \to 0$,
  \begin{equation}\label{e:ptGammalimit}
    \sup_{x,y \in V} |p_t^{\Gamma,\epsilon}(x,y)-p_t^{\Gamma,0}(x,y)| \to 0.
  \end{equation}
  Moreover,
  \begin{equation} \label{e:pGamma-inside}
    \max_{x,y \in V^\delta} |\partial^\alpha p^\Gamma_t(x,y) - \partial^\alpha p_t(x,y)|
    \leq O_\alpha(t^{-d/2-|\alpha|/2}e^{-c\delta/\sqrt{K^2t}})
    .
  \end{equation}
\end{lemma}

\begin{proof}
  Restricted to $V$, the Dirichlet Laplacian $-\Delta_\Gamma^{\epsilon}$ has
  eigenfunctions and eigenvalues
  \begin{equation}
    e_k(x) = 2^{d/2} K^d \prod_{i=1}^d \sin(k_ix_i)
    ,
    \qquad
    \lambda^\epsilon(k) = \epsilon^{-2}\sum_{i=1}^d (2-2\cos(\epsilon k_i)) \sim |k|^2,
  \end{equation}
  where $k \in V^* = K\pi \{ l \in \Z^2: 0 <  l_i < 1/\epsilon\}$;
  see, for example, \cite[Section~9.5]{MR887102}.
  Therefore
  \begin{equation}
    p^\Gamma_t(x,y) = \sum_{k \in V^*} e^{-\lambda^\epsilon(k) t} e_k(x)e_k(y),
  \end{equation}
  and it follows that
  \begin{equation}
    |\partial^\alpha p^\Gamma_t(x,y)|
    \lesssim  K^d \sum_{k \in V^*} e^{-c|k|^2t} |k|^{|\alpha|}
    \lesssim  K^d \sum_{q\in \Z^2\setminus 0} e^{-cK^2|q|^2 t} (K|q|)^{|\alpha|}
    \lesssim t^{-d/2-|\alpha|/2} e^{-c K^2 t}
  \end{equation}
  which is \eqref{e:pGammabd}.

To see \eqref{e:ptGammalimit} and \eqref{e:pGamma-inside},
we start from the following alternative representation for the Dirichlet heat kernel.
For $y \in \R^2\setminus \Gamma$, let $\{y_j\}$ be the set of points obtained by
reflecting $y=y_0$ about the lines in $\Gamma$,
and denote by $\sigma_j$ the number of reflections
needed to obtain $y_j$ from $y_0$.
Then 
\begin{equation} \label{e:etDeltaGamma}
  e^{t\Delta_\Gamma}(x,y) = \begin{cases}
    \sum_{j=0}^\infty (-1)^{\sigma_j} e^{t\Delta}(x,y_j) & (\text{$x,y$ in the same connected component of $\R^2 \setminus \Gamma$})\\
    0 & \text{(otherwise)}.
  \end{cases}
\end{equation}
Indeed, the proof of \eqref{e:etDeltaGamma} is analogous to that
in \cite[Section~7.5]{MR887102} for the Green function:
Denoting the right-hand side of \eqref{e:etDeltaGamma} by $p^\Gamma_t(x,y)$,
one has $\partial_t p^\Gamma_t(x,y) = \Delta p^\Gamma_t(x,y)$,
and the reflections ensure that $p^\Gamma_t(x,y) = 0$ if $y \in \Gamma$.
Therefore, for $x$ and $y$ in the same connected component of $\R^2\setminus \Gamma$,
\begin{equation}
  p_t(x,y)-p_t^\Gamma(x,y) = -\sum_{j\neq 0} (-1)^{\sigma_j} p_t(x,y_j).
\end{equation}
The convergence \eqref{e:ptGammalimit} then follows analogously to \eqref{e:pttoruslimit-pf}.
To see \eqref{e:pGamma-inside}, notice that the condition
$\dist(x,\Gamma), \dist(y,\Gamma) \geq \delta/K$ implies that
$\dist(x,y_j) \geq 2\delta/K$ and hence that
$|\partial^\alpha p_t(x,y_j)| \lesssim t^{-d/2-|\alpha|/2} e^{-c\delta/\sqrt{K^2t}}$.
More generally, the sum over all reflected points is bounded by
\begin{align}
  \sum_{j\neq 0} |\partial^\alpha p_t(x,u_j)|
  \lesssim t^{-d/2-|\alpha|/2} \sum_{q \in \Z^2 \setminus 0} e^{-c\delta |q|/\sqrt{K^2t}}
  \lesssim t^{-d/2-|\alpha|/2} (\sqrt{K^2t}/\delta)^{d} e^{-c\delta/\sqrt{K^2 t}}.
\end{align}
This gives
\begin{equation} \label{e:pGamma-inside-pf1}
  |\partial^\alpha p^\Gamma_t(x,y) - \partial^\alpha p(x,y)|
    \leq O_\alpha(t^{-d/2-|\alpha|/2}(\sqrt{K^2t}/\delta)^de^{-c\delta/\sqrt{K^2t}}).
  \end{equation}
On the other hand, by \eqref{e:ptbounds} and \eqref{e:pGammabd} also
\begin{equation} \label{e:pGamma-inside-pf2}
  |\partial^\alpha p^\Gamma_t(x,y) - \partial^\alpha p_t(x,y)|
  \leq
  |\partial^\alpha p^\Gamma_t(x,y)| + |\partial^\alpha p_t(x,y)|
  \leq O_\alpha(t^{-d/2-|\alpha|/2}).
\end{equation}
Using \eqref{e:pGamma-inside-pf1} if $K^2t/\delta \leq 1$
and \eqref{e:pGamma-inside-pf2} if $K^2t/\delta > 1$ gives \eqref{e:pGamma-inside}.
\end{proof}

\section*{Acknowledgements}

We thank Marek Biskup and Pierre-Fran\c{c}ois Rodriguez for essential
discussions that inspired the maximum problem studied in this article
and the approach through a coupling,
and for important feedback on a preliminary version of this manuscript.
We thank Thierry Bodineau and Christian Webb for very helpful general discussions
related to the sine-Gordon model,
and we thank Ofer Zeitouni for a helpful discussion which 
brought~\cite{acostathesis} to our attention.
RB gratefully acknowledges funding from the European Research Council (ERC)
under the European Union's Horizon 2020 research and innovation programme
(grant agreement No.~851682 SPINRG).
MH  was partially supported by the UK EPSRC grant EP/L016516/1 
for the Cambridge Centre for Analysis.

\bibliography{all}

\begin{thebibliography}{10}

\bibitem{acostathesis}
J.~Acosta.
\newblock {\em Convergence in law of the centered maximum of the mollified
  {G}aussian free field in two dimensions}.
\newblock PhD thesis, University of Minnesota, 2016.

\bibitem{MR2319516}
R.J. Adler and J.E. Taylor.
\newblock {\em Random fields and geometry}.
\newblock Springer Monographs in Mathematics. Springer, New York, 2007.

\bibitem{MR3098680}
E.~A\"{\i}d\'{e}kon.
\newblock Convergence in law of the minimum of a branching random walk.
\newblock {\em Ann. Probab.}, 41(3A):1362--1426, 2013.

\bibitem{MR3594368}
L.-P. Arguin, D.~Belius, and P.~Bourgade.
\newblock Maximum of the characteristic polynomial of random unitary matrices.
\newblock {\em Commun. Math. Phys.}, 349(2):703--751, 2017.

\bibitem{2004.01513}
N.~Barashkov and M.~Gubinelli.
\newblock The $\phi^4_3$ measure via {G}irsanov's theorem.
\newblock Preprint, arXiv:2004.01513.

\bibitem{1805.10814}
N.~Barashkov and M.~Gubinelli.
\newblock A variational method for $\phi^4_3$.
\newblock Preprint, arXiv:1805.10814.

\bibitem{1907.12308}
R.~Bauerschmidt and T.~Bodineau.
\newblock {L}og-{S}obolev {I}nequality for the {C}ontinuum {S}ine-{G}ordon
  {M}odel.
\newblock {\em Comm. Pure Appl. Math.}
\newblock to appear.

\bibitem{1610.04195}
D.~Belius and W.~Wu.
\newblock Maximum of the {G}inzburg-{L}andau fields.
\newblock Preprint, arXiv:1610.04195.

\bibitem{MR2461991}
G.~Benfatto, P.~Falco, and V.~Mastropietro.
\newblock Massless sine-{G}ordon and massive {T}hirring models: proof of
  {C}oleman's equivalence.
\newblock {\em Commun. Math. Phys.}, 285(2):713--762, 2009.

\bibitem{MR649810}
G.~Benfatto, G.~Gallavotti, and F.~Nicol\`o.
\newblock On the massive sine-{G}ordon equation in the first few regions of
  collapse.
\newblock {\em Commun. Math. Phys.}, 83(3):387--410, 1982.

\bibitem{MR4043225}
M.~Biskup.
\newblock Extrema of the two-dimensional discrete {G}aussian free field.
\newblock In {\em Random graphs, phase transitions, and the {G}aussian free
  field}, volume 304 of {\em Springer Proc. Math. Stat.}, pages 163--407.
  Springer, Cham, 2020.

\bibitem{MR3509015}
M.~Biskup and O.~Louidor.
\newblock Extreme local extrema of two-dimensional discrete {G}aussian free
  field.
\newblock {\em Commun. Math. Phys.}, 345(1):271--304, 2016.

\bibitem{MR3787554}
M.~Biskup and O.~Louidor.
\newblock Full extremal process, cluster law and freezing for the
  two-dimensional discrete {G}aussian free field.
\newblock {\em Adv. Math.}, 330:589--687, 2018.

\bibitem{MR4082182}
M.~Biskup and O.~Louidor.
\newblock Conformal {S}ymmetries in the {E}xtremal {P}rocess of
  {T}wo-{D}imensional {D}iscrete {G}aussian {F}ree {F}ield.
\newblock {\em Commun. Math. Phys.}, 375(1):175--235, 2020.

\bibitem{MR956064}
M.~Bramson.
\newblock Convergence to traveling waves for systems of {K}olmogorov-like
  parabolic equations.
\newblock In {\em Nonlinear diffusion equations and their equilibrium states,
  {I} ({B}erkeley, {CA}, 1986)}, volume~12 of {\em Math. Sci. Res. Inst.
  Publ.}, pages 179--190. Springer, New York, 1988.

\bibitem{MR3433630}
M.~Bramson, J.~Ding, and O.~Zeitouni.
\newblock Convergence in law of the maximum of the two-dimensional discrete
  {G}aussian free field.
\newblock {\em Comm. Pure Appl. Math.}, 69(1):62--123, 2016.

\bibitem{MR914427}
D.C. Brydges and T.~Kennedy.
\newblock Mayer expansions and the {H}amilton-{J}acobi equation.
\newblock {\em J. Statist. Phys.}, 48(1-2):19--49, 1987.

\bibitem{PhysRevE.63.026110}
D.~Carpentier and P.~Le~Doussal.
\newblock Glass transition of a particle in a random potential, front selection
  in nonlinear renormalization group, and entropic phenomena in liouville and
  sinh-gordon models.
\newblock {\em Phys. Rev. E}, 63:026110, Jan 2001.

\bibitem{2006.15933}
A.~Chandra, T.S. Gunaratnam, and H.~Weber.
\newblock Phase transitions for $\phi^4_3$.
\newblock Preprint, arXiv:2006.15933.

\bibitem{1808.02594}
A.~Chandra, M.~Hairer, and H.~Shen.
\newblock The dynamical sine-{G}ordon model in the full subcritical regime.
\newblock Preprint, arXiv:1808.02594.

\bibitem{MR3848391}
R.~Chhaibi, T.~Madaule, and J.~Najnudel.
\newblock On the maximum of the {${\rm C}\beta {\rm E}$} field.
\newblock {\em Duke Math. J.}, 167(12):2243--2345, 2018.

\bibitem{cpa.21899}
N.~Cook and O.~Zeitouni.
\newblock {M}aximum of the {C}haracteristic {P}olynomial for a {R}andom
  {P}ermutation {M}atrix.
\newblock {\em Comm. Pure Appl. Math.}
\newblock to appear.

\bibitem{MR2016604}
G.~Da~Prato and A.~Debussche.
\newblock Strong solutions to the stochastic quantization equations.
\newblock {\em Ann. Probab.}, 31(4):1900--1916, 2003.

\bibitem{MR1672504}
J.~Dimock.
\newblock Bosonization of massive fermions.
\newblock {\em Commun. Math. Phys.}, 198(2):247--281, 1998.

\bibitem{MR1777310}
J.~Dimock and T.R. Hurd.
\newblock Sine-{G}ordon revisited.
\newblock {\em Ann. Henri Poincar\'e}, 1(3):499--541, 2000.

\bibitem{MR3101848}
J.~Ding.
\newblock Exponential and double exponential tails for maximum of
  two-dimensional discrete {G}aussian free field.
\newblock {\em Probab. Theory Related Fields}, 157(1-2):285--299, 2013.

\bibitem{MR3729618}
J.~Ding, R.~Roy, and O.~Zeitouni.
\newblock Convergence of the centered maximum of log-correlated {G}aussian
  fields.
\newblock {\em Ann. Probab.}, 45(6A):3886--3928, 2017.

\bibitem{1912.13184}
M.~Fels and L.~Hartung.
\newblock Extremes of the 2d scale-inhomogeneous discrete {G}aussian free
  field: {C}onvergence of the maximum in the regime of weak correlations.
\newblock Preprint, arXiv:1912.13184.

\bibitem{MR0434278}
J.~Fr\"{o}hlich.
\newblock Classical and quantum statistical mechanics in one and two
  dimensions: two-component {Y}ukawa- and {C}oulomb systems.
\newblock {\em Commun. Math. Phys.}, 47(3):233--268, 1976.

\bibitem{MR887102}
J.~Glimm and A.~Jaffe.
\newblock {\em Quantum physics}.
\newblock Springer-Verlag, New York, second edition, 1987.
\newblock A functional integral point of view.

\bibitem{MR3406823}
M.~Gubinelli, P.~Imkeller, and N.~Perkowski.
\newblock Paracontrolled distributions and singular {PDE}s.
\newblock {\em Forum Math. Pi}, 3:e6, 75, 2015.

\bibitem{MR3274562}
M.~Hairer.
\newblock A theory of regularity structures.
\newblock {\em Invent. Math.}, 198(2):269--504, 2014.

\bibitem{MR3452276}
M.~Hairer and H.~Shen.
\newblock The dynamical sine-{G}ordon model.
\newblock {\em Commun. Math. Phys.}, 341(3):933--989, 2016.

\bibitem{2003.12535}
Y.~Huang.
\newblock Another probabilistic construction of $\phi^{2n}$ in dimension 2.
\newblock Preprint, arXiv:2003.12535.

\bibitem{1806.02118}
J.~Junnila, E.~Saksman, and C.~Webb.
\newblock {I}maginary multiplicative chaos: {M}oments, regularity and
  connections to the {I}sing model.
\newblock Preprint, arXiv:1806.02118.

\bibitem{MR4047992}
J.~Junnila, E.~Saksman, and C.~Webb.
\newblock Decompositions of log-correlated fields with applications.
\newblock {\em Ann. Appl. Probab.}, 29(6):3786--3820, 2019.

\bibitem{1903.01394}
H.~Lacoin, R.~Rhodes, and V.~Vargas.
\newblock A probabilistic approach of ultraviolet renormalisation in the
  boundary {S}ine-{G}ordon model.
\newblock Preprint, arXiv:1903.01394.

\bibitem{MR3339158}
H.~Lacoin, R.~Rhodes, and V.~Vargas.
\newblock Complex {G}aussian multiplicative chaos.
\newblock {\em Commun. Math. Phys.}, 337(2):569--632, 2015.

\bibitem{MR2677157}
G.F. Lawler and V.~Limic.
\newblock {\em Random walk: a modern introduction}, volume 123 of {\em
  Cambridge Studies in Advanced Mathematics}.
\newblock Cambridge University Press, Cambridge, 2010.

\bibitem{MR3606477}
T.~Lebl\'{e}, S.~Serfaty, and O.~Zeitouni.
\newblock Large deviations for the two-dimensional two-component plasma.
\newblock {\em Commun. Math. Phys.}, 350(1):301--360, 2017.

\bibitem{MR3414451}
T.~Madaule.
\newblock Maximum of a log-correlated {G}aussian field.
\newblock {\em Ann. Inst. Henri Poincar\'{e} Probab. Stat.}, 51(4):1369--1431,
  2015.

\bibitem{MR3463679}
C.~Marinelli and M.~R\"{o}ckner.
\newblock On the maximal inequalities of {B}urkholder, {D}avis and {G}undy.
\newblock {\em Expo. Math.}, 34(1):1--26, 2016.

\bibitem{MR2855536}
J.~Miller.
\newblock Fluctuations for the {G}inzburg-{L}andau {$\nabla\phi$} interface
  model on a bounded domain.
\newblock {\em Commun. Math. Phys.}, 308(3):591--639, 2011.

\bibitem{10.1002/cpa.21925}
A.~Moinat and H.~Weber.
\newblock {S}pace-{T}ime {L}ocalisation for the {D}ynamic $\phi^4_3$ model.
\newblock {\em Comm. Pure Appl. Math.}
\newblock to appear.

\bibitem{MR849210}
F.~Nicol\`o, J.~Renn, and A.~Steinmann.
\newblock On the massive sine-{G}ordon equation in all regions of collapse.
\newblock {\em Commun. Math. Phys.}, 105(2):291--326, 1986.

\bibitem{MR3848227}
E.~Paquette and O.~Zeitouni.
\newblock The maximum of the {CUE} field.
\newblock {\em Int. Math. Res. Not. IMRN}, (16):5028--5119, 2018.

\bibitem{MR3274356}
R.~Rhodes and V.~Vargas.
\newblock Gaussian multiplicative chaos and applications: a review.
\newblock {\em Probab. Surv.}, 11:315--392, 2014.

\bibitem{MR4089492}
F.~Schweiger.
\newblock The maximum of the four-dimensional membrane model.
\newblock {\em Ann. Probab.}, 48(2):714--741, 2020.

\bibitem{MR3825880}
P.~Tsatsoulis and H.~Weber.
\newblock Spectral gap for the stochastic quantization equation on the
  2-dimensional torus.
\newblock {\em Ann. Inst. Henri Poincar\'{e} Probab. Stat.}, 54(3):1204--1249,
  2018.

\bibitem{1907.08868}
M.~Wirth.
\newblock Maximum of the integer-valued {G}aussian free field.
\newblock Preprint, arXiv:1907.08868.

\bibitem{MR3933043}
W.~Wu and O.~Zeitouni.
\newblock Subsequential tightness of the maximum of two dimensional
  {G}inzburg-{L}andau fields.
\newblock {\em Electron. Commun. Probab.}, 24:Paper No. 19, 12, 2019.

\end{thebibliography}
\bibliographystyle{plain}

\end{document}